\documentclass[aos,preprint]{imsart}

\RequirePackage[OT1]{fontenc}
\RequirePackage{amsmath,amssymb,amsfonts,amsthm,bbm}
\RequirePackage[numbers]{natbib}
\RequirePackage[colorlinks,citecolor=blue,urlcolor=blue]{hyperref}
\usepackage{color,graphicx}
\hypersetup{linktocpage}

\startlocaldefs
\newtheorem{theorem}{Theorem}

\newtheorem{corollary}{Corollary}
\newtheorem{lemma}{Lemma}
\newtheorem{proposition}{Proposition}

\newtheorem{remark}{Remark}

\renewcommand{\H}{\mathbb{H}}
\newcommand{\E}{\mathbb{E}}
\newcommand{\R}{\mathbb{R}}
\newcommand{\T}{\mathbb{T}}

\newcommand{\G}{\mathbb{G}}

\newcommand{\mF}{\mathcal{F}}

\newcommand{\supp}{\text{supp}}

\newcommand{\eps}{\varepsilon}

\newenvironment{enumerate*}%

\endlocaldefs

\begin{document}

\begin{frontmatter}
\title{Nonparametric statistical inference for drift vector fields of multi-dimensional diffusions}
\runtitle{Nonparametric inference for diffusions}

\begin{aug}
\author{\fnms{Richard} \snm{Nickl}\ead[label=e1]{r.nickl@statslab.cam.ac.uk}}
\and
\author{\fnms{Kolyan} \snm{Ray}\ead[label=e2]{kolyan.ray@kcl.ac.uk}}

\runauthor{R. Nickl and K. Ray}

\affiliation{University of Cambridge and King's College London}

\address{Statistical Laboratory\\
Department of Pure Mathematics\\
\quad and Mathematical Statistics\\ 
University of Cambridge\\
CB3 0WB Cambridge\\
United Kingdom\\
\printead{e1}}

\address{Department of Mathematics\\
King's College London\\
Strand\\
London WC2R 2LS\\
United Kingdom\\
\printead{e2}}

\end{aug}

\begin{abstract}
The problem of determining a periodic Lipschitz vector field $b=(b_1, \dots, b_d)$ from an observed trajectory of the solution $(X_t: 0 \le t \le T)$ of the multi-dimensional stochastic differential equation
\begin{equation*}
dX_t = b(X_t)dt + dW_t, \quad t \geq 0,
\end{equation*}
where $W_t$ is a standard $d$-dimensional Brownian motion, is considered. Convergence rates of a penalised least squares estimator, which equals the maximum a posteriori (MAP) estimate corresponding to a high-dimensional Gaussian product prior, are derived. These results are deduced from corresponding contraction rates for the associated posterior distributions. The rates obtained are optimal up to log-factors in $L^2$-loss in any dimension, and also for supremum norm loss when $d \le 4$.  Further, when $d \le 3$, nonparametric Bernstein-von Mises theorems are proved for the posterior distributions of $b$. From this we deduce functional central limit theorems for the implied estimators of the invariant measure $\mu_b$. The limiting Gaussian process distributions have a covariance structure that is asymptotically optimal from an information-theoretic point of view.
\end{abstract}



\end{frontmatter}

\tableofcontents

\section{Introduction}

For $W_t= (W_t^1, \dots, W_t^d)$ a $d$-dimensional Brownian motion and $b=(b_1, \dots, b_d)$ a Lipschitz vector field, consider the multi-dimensional Markov diffusion process $(X_t=(X_t^1, \dots, X_t^d): t \ge 0)$ describing the solution to the stochastic differential equation (SDE) 
\begin{equation} \label{eq:SDE}
dX_t= b(X_t)dt + dW_t, \quad \quad X_0 = x_0\in \R^d,~  t \geq 0.
\end{equation}
The random process $(X_t: t \ge 0)$ describes a Brownian motion whose trajectories are subject to spatially variable displacements enforced by the drift vector field $b$. We are interested in recovering the parameter $b$ based on observing the process up to time $T$. A closely related problem is that of estimating the invariant measure $\mu_b$ of the diffusion, which in the \textit{ergodic} case   describes the probabilities 
\begin{equation} \label{erginv}
\mu_b(A) =^{a.s.} \lim_{T \to \infty} \frac{1}{T}\int_0^T 1_A(X_t)dt
\end{equation}
corresponding to the average asymptotic time spent by the process $(X_t)$ in a given measurable subset $A$ of the state space. 

While the one-dimensional case $d=1$ is well studied (e.g., \cite{K04}, \cite{D05, DR06, vdmeulen2006, LLL11, pokern2013, vanwaaij2016,A18, AS18b, AS18a}), comparably little is known about the important multi-dimensional setting, particularly when $b$ is modelled in a nonparametric or high-dimensional way. In the measurement model we consider here, Dalalyan and Rei\ss~\cite{DR07} first obtained convergence rates of multivariate nonparametric kernel-type estimators. Schmisser \cite{S13} established adaptive $L^2$-convergence rates of certain model selection based projection estimators and Strauch \cite{S15,S16, S18} obtained adaptive convergence rate results for $b$ in pointwise and $L^2$-loss, and for $\mu_b$ in $\|\cdot\|_\infty$-loss -- more discussion can be found below.


For observations $(X_t: 0 \le t \le T)$, the likelihood function is directly available from Girsanov's theorem and has a convenient `Gaussian' form in the parameter $b$. This motivates the use of likelihood based inference procedures: the estimators $\hat b_T$ for $b$ we study in the present paper are minimisers of a penalised likelihood (or least squares) criterion over a high-dimensional approximation space. In fact, since the penalties we use are squared Hilbert norms, $\hat b_T$ equals a Bayesian `maximum a posteriori' (MAP) estimate arising from a truncated Gaussian series prior. The Bayesian interpretation of $\hat b_T$ is exploited in our proofs and has further appeal since it directly suggests uncertainty quantification methodology (`posterior credible sets'). In particular, posterior sampling is feasible even for `real-world' discrete data by simulation techniques, see \cite{BPRF06, PPRS12, BFS16, vdMS17, vdMS17b} and references therein.  


Let us briefly describe our contributions: we obtain convergence rates of $\hat b_T$ to the `true' vector field $b_0$ generating equation (\ref{eq:SDE}) and also frequentist contraction rates about $b_0$ for the corresponding posterior distributions, both in $L^2$- and $\|\cdot\|_\infty$-distances. For $L^2$-loss the rates obtained are minimax optimal (up to log-factors) over H\"older classes in any dimension, and this remains true for $\|\cdot\|_\infty$-loss whenever dimension $d \le 4$. When $d \le 3$, we further prove nonparametric Bernstein-von Mises theorems that establish asymptotic normality of the re-centred and scaled posterior distributions $\sqrt T(b-\hat b_T)|(X_t: 0 \le t \le T)$ in a (large enough) function space.  From this we deduce central limit theorems for the implied plug-in estimators for the invariant density $\mu_b$. The proofs imply that the limiting covariances obtained coincide with the semiparametric information lower bounds for these estimation problems. We exploit that the non-linear identification map $b \mapsto \mu_b$ can be shown to be `one-smoothing' -- as inference on $b$ is approximately  a nonparametric regression problem \cite{DR07}, this offers an analytical explanation for why the invariant density $\mu_b$ of the process can be estimated at $1/\sqrt T$ rate in stronger norms than is the case in i.i.d.~density estimation. 


The multi-dimensional case $d \ge 2$ is fundamentally more challenging than the one-dimensional one for various reasons. First, when $d=1$ properties of diffusion local times can be used to take advantage of regularity properties of the sample paths of $(X_t)$ as in \cite{vdmeulen2006, DR06, pokern2013, vanwaaij2016, AS18a, AS18b}, whereas for $d>1$ these local times are no longer appropriately defined. Second, Markovian concentration properties can be derived using martingale techniques combined with mapping properties of the generator of the underlying semigroup (via the Poisson equation and It\^o's formula, see Lemma \ref{lem:maxim}). In dimension one this involves the study of an explicitly solvable ordinary differential equation (ODE), whereas for $d\ge 2$ the theory of elliptic partial differential equations (PDEs) is required. PDE techniques are an effective alternative to the functional inequalities used in \cite{DR07, S16, S18}, in particular the requirement that $b$ be a gradient vector field $\nabla B$ for some $B:\mathbb R^d \to \mathbb R$, and thus of \textit{reversibility} of $(X_t)$ -- used in the references \cite{DR07, S13, S16, S18} -- can be avoided this way; neither does $X_0 \sim \mu_b$ have to be started in equilibrium as in \cite{S13}. To simplify the PDE arguments in our proofs we restrict to periodic vector fields $b$. In our setting, periodicity ensures the required mixing properties of $(X_t)$, replacing spectral gap assumptions in \cite{DR07, S16, S18}. The techniques of the present paper extend in principle, albeit at the expense of considerable technicalities, to the non-periodic case if $b$ is known outside of a compact subset of $\mathbb R^d$ and upon employing assumptions on $b$ as in \cite{PV01}. Finally, that $b$ is not required to be a gradient field is crucial in the multi-dimensional setting for the use of Bayesian (or penalisation) methods as standard Gaussian priors for $b$ will draw gradient vector fields with probability zero. Moreover, for $d>1$ the potential absence of reversibility of $(X_t)$ introduces some fundamentally new features to the inference problem at hand, since the invariant measure $\mu_b$ no longer identifies the law $P_b$ of the process $(X_t)$ -- see after Proposition \ref{as:true_drift} below. Unlike in the one-dimensional case (e.g., \cite{NS17a}, \cite{A18}), Bayesian inference thus cannot be based on a prior assigned directly to the invariant measure $\mu_b$. In contrast we show how Gaussian priors for $b$ give valid Bayesian models for the data and allow one to make optimal inference on $b$ and $\mu_b$.


 Our proofs employ techniques from Bayesian non- and semi-parametric statistics, specifically \cite{CN13, CN14, castillo2014, CR15}. In this regard our results are related to recent investigations of Bayesian inverse problems \cite{KvdVvZ11, R13, ALS13, DLSV13, KS18, NvdGW18}, Bernstein-von Mises theorems \cite{N17, MNP17, NS17b}, \cite{RR12, R14, R17}, \cite{GK18} and diffusion models \cite{vdmeulen2006, pokern2013, vanwaaij2016, NS17a, vdMvZ13, GS14, A18}.

\section{Main results}

\subsection{Basic notation and definitions}

Let $\T^d$ denote the $d$-dimensional torus, isomorphic to $(0,1]^d$ if opposite points on the cube are identified.
By $L^2(\T^d)$ we denote the usual $L^2$-spaces with respect to Lebesgue measure $dx$ on $\T^d$ equipped with inner product $\langle \cdot, \cdot \rangle=\langle \cdot, \cdot \rangle_{L^2}$. Let $\mu$ be a probability measure on $\T^d$. If its Lebesgue density, also denoted by $\mu$, exists and is bounded and bounded away from zero, then an equivalent norm $\|\cdot\|_\mu$ on $L^2(\T^d)$ arises from the inner product $\langle f,g \rangle_\mu = \int fg d\mu$ for $f,g \in L^2(\T^d)$. The symbol $L^2_{0}(\T^d)$ denotes the subspace of functions $f$ for which $\int_{\T^d}f(x)dx=0$, and $L^2_\mu(\T^d)$ denotes the subspace for which $\int_{\T^d} f d\mu =0$.

We define the space $C(\T^d)=C^0(\T^d)$ of continuous functions on $\T^d$ normed by the usual supremum norm $\|\cdot\|_\infty$. For $s >0$, we denote by $C^s(\T^d)$ the usual H\"older spaces of $[s]$-times continuously differentiable functions on $\T^d$, where $[s]$ is the integer part of $s$. For $s\in \R$, let $H^s(\T^d)$ denote the usual Sobolev space of functions from $\T^d$ to $\R$ (defined by duality when $s<0$). They form the special case $p=q=2$ in the scale of Besov spaces $B^s_{pq}(\T^d), 1 \le p,q \le \infty,$ see Chapter 3 of \cite{ST87} for definitions, where it is also shown that $C^s(\T^d)$ embeds continuously into $B^s_{\infty \infty}(\T^d), s \ge 0$. When no confusion may arise, we employ the same function space notation for vector fields $f=(f_1,\dots,f_d)$. For instance $f\in H^s \equiv (H^s)^{\otimes d}$ will then mean that each $f_j \in H^s(\T^d)$ and the norm on $H^s$ is given by $\|f\|_{H^s}^2 = \sum_{j=1}^d \|f_j\|_{H^s}^2.$ We shall repeatedly use multiplication inequalities for Besov-Sobolev norms,
\begin{equation}\label{mult}
\|fg\|_{B^s_{pq}} \leq c(s,p,q, d) \|f\|_{B^s_{pq}} \|g\|_{B^s_{\infty \infty}} \le c'(s,p,q, d) \|f\|_{B^s_{pq}} \|g\|_{C^s}, s \ge 0.
\end{equation}
Starting with a periodised Daubechies' wavelet basis of $L^2(\T)$, we consider a tensor product wavelet basis of $L^2(\T^d)$ given by 
\begin{equation} \label{wavbas}
\begin{split}
\{\Phi_{l,r}:  r =0,&.., \max(0,2^{ld}-1), ~l=\{-1,0\} \cup \mathbb N \},\\
& V_J \equiv \textrm{span}(\Phi_{l,r}: r, l \le J),
\end{split}
\end{equation}
  see Section 4.3 of \cite{ginenickl2016}, where the base Daubechies wavelets are taken `$S$-regular', $S \in \mathbb N$. The dimension of $V_J$ is $O(2^{Jd})$ as $J \to \infty$, and the decay of wavelet coefficients in this basis, or equivalently the scaling of approximation errors from $L^2$-projections $P_{V_J}$ onto $V_J$, characterise the norms of the Besov spaces $B^s_{pq}(\T^d)$ and Sobolev spaces $H^s(\T^d)$ (p.370f.~in \cite{ginenickl2016}). 

If $\mu$ is a probability measure on some metric space, then $Z \sim \mu$ means that $Z$ is a random variable in that space drawn from the distribution $\mu$, also called the law $\mathcal L(Z)=\mu$ of $Z$. We write $Z_T \to^d Z$, or $Z_T \to^d \mathcal L(Z)$ when no confusion can arise, to denote the usual notion of weak convergence of the laws $\mathcal L(Z_T) \to \mathcal L(Z)$ as $T \to \infty$, see, e.g., Chapter 11 in \cite{D02}. 

For a normed linear space $(X, \|\cdot\|_X)$,  the topological dual space is $$X^* = (X,\|\cdot\|_X)^* := \{L: X \to \mathbb R \text{ linear s.t. } |L(x)| \le C\|x\|_X ~ \forall x \in X, C>0\},$$ which is a Banach space for the norm $\|L\|_{X^*} \equiv \sup_{x \in X, \|x\|_X \le 1} |L(x)|.$ We will sometimes use the symbols $\lesssim, \gtrsim, \simeq$ to denote one- or two-sided inequalities up to multiplicative constants that may either be universal or `fixed' in the context where the symbols appear. We also write $(\cdot)_+= \max(\cdot,0)$ to denote the non-negative part of a real number, and $a \vee b,  a\wedge b$ to denote maximum and minimum of real numbers $a,b$, respectively.

\subsection{Diffusions with periodic drift; likelihood, prior and posterior}

Consider the SDE (\ref{eq:SDE}) where the vector field $b:\R^d \mapsto \R^d$ is Lipschitz continuous and one-periodic, that is $b(\cdot+m)=b(\cdot)$ for every $m\in \mathbb{Z}^d$.  Then a strong pathwise solution of this SDE exists which is a $d$-dimensional diffusion Markov process $X_t=(X_t^1,\dots,X_t^d)$. We denote by $P_b=P_b^x$ the cylindrical probability measure describing the law of $(X_t)$ in path space $C([0, \infty) \to \mathbb R^d)$ when $X_0=x$; its restriction $P_b^T=P_b^{T,x}$ to the separable space $C([0,T] \to \mathbb R^d)$ describes the law of the process $X^T\equiv(X_t:t\in[0,T])$ until time $T$, see, e.g., Sections 24 and 39 in \cite{B11}. We suppress the dependence on the starting value $x$ as our results do not depend on it. 

We seek to recover the drift function $b: \T^d \to \mathbb R^d$ from an observed trajectory $X^T$. The periodic model (which has also been used in \cite{pokern2013, vanwaaij2016} when $d=1$) is convenient in our context as it effectively confines the diffusion process $(X_t)$ to a bounded state space $\T^d$. To be precise, while our diffusion takes values in the whole of $\R^d$ (in particular $(X_t)$ will not be globally recurrent), the values of the process $(X_t)$ modulo $\mathbb Z^d$ contain all relevant statistical information. In particular, we have (arguing as in the proof of Lemma \ref{lem:LAN_expansion} below),
 \begin{equation*}\label{erglln}
\frac{1}{T}\int_0^T \varphi(X_t)dt \to^{P_b} \int_{\T^d} \varphi d\mu_b \text{ as } T \to \infty,~~ \forall \varphi \in C(\T^d),
\end{equation*}
where $\mu_b$ is a uniquely defined probability measure on $\T^d$ and where we identify $\varphi$ with its periodic extension to $\R^d$ on the left-hand side. The measure $\mu_b$ has the usual probabilistic interpretation as an invariant measure appearing in the limit of ergodic averages, but for our purposes it is more convenient to define it in terms of a partial differential equation involving the generator of the diffusion Markov process. Heuristically, if $(P_t = e^{tL}: t \ge 0)$ is the transition operator of a diffusion process with invariant measure $\mu$ and generator $L$, then we can differentiate the invariant identity $\int P_t [\varphi] d\mu = \int \varphi d\mu ~\forall t$ at $t=0$, so that $\int L\varphi d\mu =0$ for all smooth $\varphi$. If $L^*$ is the adjoint operator for the standard $L^2$-inner product, then it must satisfy $\int \varphi L^*\mu =0$ for all smooth $\varphi$, and hence necessarily $L^*\mu=0$ (in the weak sense), which can be used to identify $\mu$ via the adjoint generator $L^*$. 

More precisely, in our setting the generator $L:H^2(\T^d) \rightarrow L^2(\T^d)$ is 
\begin{align}\label{eq:generator}
L = L_b= \frac{1}{2} \Delta + b.\nabla = \frac{1}{2}\sum_{i=1}^d \frac{\partial^2}{\partial x_i^2} + \sum_{i=1}^d b_i(\cdot) \frac{\partial}{\partial x_i},
\end{align}
and from integration by parts the adjoint operator for $\langle \cdot, \cdot \rangle_{L^2}$ equals
\begin{equation} \label{eq:cgen}
L^* = L_b^* = \frac{1}{2}\Delta - b . \nabla - div(b),~~~div(b)=\sum_{j=1}^d \frac{\partial b_j}{\partial x_j},
\end{equation}
so that $\mu_b$ can be identified as the (weak) solution of the PDE
\begin{equation} \label{cpde}
L_b^* \mu_b \equiv \frac{1}{2} \Delta \mu_b - b . \nabla \mu_b - div(b) \mu_b = 0.
\end{equation}
One can prove the following result (see after (\ref{stark}) in Section \ref{pdeeee} below).
\begin{proposition}\label{as:true_drift}
Let $b \in C^1(\T^d)$. A unique periodic solution $\mu_b $ to (\ref{cpde}) satisfying $\int_{\T^d} d\mu_b=1$ exists. Moreover, $\mu_b$ is Lipschitz continuous and bounded away from zero on $\T^d$, with $\|1/\mu_b\|_\infty$ and the Lipschitz constant $\|\mu_b\|_{Lip}$ depending on $b$ only through a bound for $\|b\|_\infty$.
\end{proposition}
One may show (e.g., as after (\ref{chain}) below) that for smoother vector fields $b$ the resulting invariant measure actually equals a classical $C^2$-solution of (\ref{cpde}), but for existence of $\mu_b$ a weak solution suffices.

If $b$ arises as a gradient vector field $\nabla B$ for some $B \in C^2(\T^d)$, one can check directly that $\mu_b \propto e^{2B}$ is a classical solution of (\ref{cpde}), and we can then recover $b$ from $\mu_b$ via $b=(1/2)\nabla \log \mu_b$. But the invariant measure $\mu_b$ does \textit{not} identify $b$ or the law $P_b$ of $(X_t: t \ge 0)$ for general vector fields $b$ (unless $d=1$). To see this, start with  $b=\nabla B$ and invariant measure $\mu_b \propto e^{2B}$. For any smooth \textit{divergence free} vector field $\bar v$ and $v=\bar v/\mu$ (so that $div(v\mu)=0$) one checks by integration by parts that $\int \phi L^*_{b+v} \mu_b=\int \mu_b L_{b+v}\phi=0$ for all smooth $\phi$, and as a consequence $\mu_b$ is also the invariant measure for $L_{b+v}$. Thus any statistical approach to recover $b$ via first estimating $\mu_b$ is bound to fail in our general setting. 

We instead propose likelihood-based inference methods. The log-likelihood function $\ell_T(b)$ of our measurement model can be obtained from Girsanov's theorem (Section 17.7 in \cite{B11}): for any periodic and Lipschitz $b: \T^d \to \mathbb R^d$,
\begin{equation}\label{eq:likelihood}
e^{\ell_T(b)} = \frac{dP_b^T}{dP_0^T}(X^T) = \exp \Big( -\frac{1}{2} \int_0^T \|b(X_t)\|^2 dt + \int_0^T b(X_t).dX_t \Big),
\end{equation}
where $P_0^T$ is the law of a  $d$-dimensional Brownian motion $(W_t:t\in[0,T])$.

Our approach to inference on $b$ amounts to computing a penalised maximum likelihood estimator over a high-dimensional wavelet approximation space. More precisely, set
\begin{equation} \label{mapest}
\hat b_T =  \hat b(X^T) = \text{argmin}_{b \in V_J^{\otimes d}} \big[-\ell_T(b) + \frac{1}{2}\|b\|_{\mathbb H}^2\big],
\end{equation}
where $V_J^{\otimes d}=\otimes_{j=1}^d V_J$ (cf.~(\ref{wavbas})) and $\|\cdot\|_\mathbb H$ is a Hilbert tensor norm on $V_J^{\otimes d}$. The estimator $\hat b_T$ has a natural Bayesian interpretation as the maximum a posteriori (MAP) estimate arising from a mean zero Gaussian prior $\Pi = \otimes_{j=1}^d \Pi_j$ on $V_J^{\otimes d}$ with reproducing kernel Hilbert space $\mathbb H$. Indeed, the posterior distribution $\Pi(\cdot|X^T)$ arising from observing $X^T \sim P_b^T$ is of the form
\begin{equation} \label{postform}
d\Pi(b|X^T) = \frac{e^{\ell_T(b)}d\Pi(b)}{\int e^{\ell_T(b)}d\Pi(b)} \propto e^{\ell_T(b) - \frac{1}{2}\|b\|_{\mathbb H}^2}, ~~b \in V_J^{\otimes d}.
\end{equation}
Our proofs imply that the denominator in the last expression is finite and non-zero with probability approaching one under the law of $X^T$ as $T \to \infty$. The map $(b,c) \mapsto \int_0^T b(X_t)c(X_t)dt+\langle b, c\rangle_\mathbb H$ induces an inverse covariance $D^{-1}_H$ on some linear subspace $H \subset V_J^{\otimes d}$. [Since $1 \in V_J^{\otimes d}$, $dim H \neq 0$, and our proofs imply in fact that $H=V_J^{\otimes d}$ with probability approaching one as $T \to \infty$.] By characterisations of Gaussian laws (e.g., Theorem 9.5.7 in \cite{D02}) and linearity of $b \mapsto \int_0^T b(X_t).dX_t$, the distribution $\Pi(\cdot|X^T)$ is thus Gaussian on $V_J^{\otimes d}$ and the MAP estimate (\ref{mapest}) equals the posterior mean $E^{\Pi}[b|X^T]$.

The Gaussian process priors $\Pi=\Pi_T$ we will use here are constructed from high-dimensional wavelet expansions for $b =(b_1,\dots,b_d)$ of the form:
\begin{equation}\label{eq:Gauss_prior_wav}
b_j = \sum_{l \le J} \sum_{r=0}^{2^{ld}-1} \sigma_l g_{l,r,j} \Phi_{l,r},~~ g_{l,r,j} \sim^{iid} \mathcal N(0, 1), \quad j=1, \dots, d,
\end{equation}
where the $\Phi_{l,r}$ form a $S$-regular periodised wavelet basis of $L^2(\mathbb T^d)$ (cf.~(\ref{wavbas})), where $J = J_T \to \infty$ as $T \to \infty$ in a way to be chosen below, and where the weights $\sigma_l$ govern the regularisation prescribed by the penalty functional. We will tacitly assume throughout that $S$ is large enough (depending on parameters $s,a,\alpha$ to be specified). We choose wavelets for convenience and $B$-spline bases, which give rise to the same MAP estimates, could have been used as well. Recall (p.75 in \cite{ginenickl2016}) that the Gaussian process \eqref{eq:Gauss_prior_wav} has reproducing kernel Hilbert space (RKHS) inner product of tensor form
\begin{align}\label{eq:RKHS_inner_prod}
\langle g_1,g_2 \rangle_{\H} = \sum_{j=1}^d \sum_{l\leq J} \sum_{r=0}^{2^{ld}-1} \sigma_l^{-2} \langle g_{1,j},\Phi_{l,r} \rangle_{L^2} \langle g_{2,j},\Phi_{l,r} \rangle_{L^2}, ~~g_1,g_2 \in V_J^{\otimes d}.
\end{align}

\subsection{Contraction rates for the posterior distribution and MAP estimate}\label{res1}

We now give results concerning the concentration of the posterior measure $\Pi(\cdot|X^T)$ around the `ground truth' vector field $b_0$ that generated $X^T$ according to the diffusion equation (\ref{eq:SDE}). This implies convergence rates of the same order of magnitude for the MAP estimate $\hat b_T$ (see Corollary \ref{MAP}). We denote the `true' invariant measure from Proposition \ref{as:true_drift} by $\mu_0=\mu_{b_0}$.

Our first theorem gives a contraction rate in the `natural distance' induced by the statistical experiment, following the general theory \cite{GvdV17, vdmeulen2006}. Initially this distance is a `random Hellinger semimetric' (see Theorem \ref{thm:contraction_general} below). In dimension $d=1$, the theory of diffusion local times can then be used to compare this metric to the standard $\|\cdot\|_{\mu_0}, \|\cdot\|_{L^2}$-distances \cite{vdmeulen2006, pokern2013, vanwaaij2016}, but when $d>1$ such local time arguments are not available. We instead exploit concentration properties of the high-dimensional random matrices induced by the Hellinger semimetric on $V_J^{\otimes d}$ (Lemma \ref{lem:bilinear}).

\begin{theorem}\label{thm:contraction}
Let $s>\max(d/2,1),d \in \mathbb N$. Suppose $b_0 \in C^{s}(\T^d) \cap H^{s}(\T^d)$. Consider the Gaussian prior $\Pi_T$ from \eqref{eq:Gauss_prior_wav} with $2^J \approx T^\frac{1}{2a+d}$ and $\sigma_l = 2^{-l(\alpha+d/2)}$ for $a>\max(d-1,1/2)$ and $0\leq \alpha \leq a$. Then for $\varepsilon_T = T^{-\frac{a\wedge s}{2a+d}}(\log T)$ and every $M_T \rightarrow \infty$, as $T\rightarrow \infty$,
\begin{align*}
\Pi_T \left( b:\|b-b_0\|_{\mu_0} \geq M_T\varepsilon_T | X^T \right) \rightarrow^{P_{b_0}} 0.
\end{align*}
In particular, if $a=s$ then $\varepsilon_T = T^{-\frac{s}{2s+d}}(\log T)$.
\end{theorem}

Since we wish to perform the primary regularization via the truncation level $J$ rather than the variance scaling $\alpha$ we have taken $0 \leq \alpha \leq a$. 

\begin{remark}[Adaptation] \normalfont
The previous theorem extends to adaptive priors, where $J$ is randomised according to a hyperprior on $\mathbb N$ of the form $\Pi(J=j)\sim \exp\{-C2^{jd}\}$, without requiring knowledge of the smoothness $s$. Given the techniques underlying Theorem \ref{thm:contraction}, the proof of such a result follows standard patterns (e.g., \cite{R13}, \cite{A18}) and is left to the reader. 
\end{remark}

From the previous theorem, and imposing slightly stronger conditions on $b_0$ and $\Pi_T$, one can obtain perturbation approximations of the  Laplace transform of $\Pi(\cdot|X^T)$ by the Laplace transform of a certain Gaussian distribution (see Proposition \ref{prop:laplace_trans}), which makes more precise `semiparametric' tools available for the analysis of the posterior distribution. Following ideas in \cite{castillo2014} (see also \cite{CN14, CR15, C17, NS17b}) we obtain  contraction results in the $\|\cdot\|_\infty$-norm.

\begin{theorem} \label{super}
Let $a\wedge s>\max(3d/2 -1,1), d \in \mathbb N$. Suppose $b_0 \in C^{s}(\T^d) \cap H^{s}(\T^d)$. Consider the Gaussian  prior $\Pi_T$ from \eqref{eq:Gauss_prior_wav} with $2^J \approx T^\frac{1}{2a+d}$ and $\sigma_l = 2^{-l(\alpha+d/2)}$ for $0\leq \alpha < a\wedge s-d/2$. Assume further that $a\leq s+1$ if $d\leq 4$ or $a \leq s+d/2-1$ if $d\geq 5$. Then for every $\delta>5/2$,
\begin{align*}
\Pi_T \Big(  b: \sum_{j=1}^d \|b_j-b_{0,j}\|_{\infty} \geq  (\log T)^\delta T^{-\frac{s\wedge[a-(d/2-2)_+]}{2a+d}} \big| X^T \Big) \rightarrow^{P_{b_0}} 0~\text{ as } T \to \infty.
\end{align*}
In particular, if $a=s, 0 \leq \alpha \leq s-d/2$ and $d \le 4$, then the convergence rate is $(\log T)^\delta T^{-\frac{s}{2s+d}}$.
\end{theorem}

By Gaussianity of the posterior distribution, the previous theorems translate into convergence rates of the MAP estimates from (\ref{mapest}).

\begin{corollary}\label{MAP}
Let $\hat{b}_T = E^{\Pi_T}[b|X^T]$. Under the conditions of Theorem \ref{thm:contraction},  for every $M_T \to \infty$, $$\|\hat{b}_T - b_0\|_{\mu_0} = O_{P_{b_0}} (M_T T^{-\frac{a \wedge s}{2a+d}}\log T) \quad  \text{ as } T \to \infty,$$ while under the conditions of Theorem \ref{super}, for every $\delta>5/2$, $$\|\hat{b}_T - b_0\|_{\infty} = O_{P_{b_0}} (T^{-\frac{s\wedge[a-(d/2-2)_+]}{2a+d}} (\log T)^\delta) \quad \text{as }T\rightarrow \infty.$$
\end{corollary}

\begin{proof}
Consider the function $$H(b')  = \Pi_T (b: \|b-b'\|_{\mu_0} \leq M_T \varepsilon_T |X^T), ~b' \in V_J^{\otimes d}.$$ The posterior is a Gaussian measure on the finite-dimensional space $V_J^{\otimes d}$, centered at $\hat{b}_T$. Since $\|\cdot\|_{\mu_0}$-norm balls centred at the origin are convex symmetric sets, Anderson's Lemma (Theorem 2.4.5 of \cite{ginenickl2016}) yields that $\hat{b}_T$ is a maximizer of $H$. Using Theorem 2.5 in \cite{ghosal2000} with the contraction rate from Theorem \ref{thm:contraction}, we deduce that $\|\hat{b}_T-b_0\|_{\mu_0} = O_{P_{b_0}}(M_T \varepsilon_T)$ as $T\rightarrow \infty$. The $\|\cdot\|_\infty$-rate follows similarly using the contraction rate from Theorem \ref{super}.
\end{proof}

Up to log-factors, the $\|\cdot\|_{L^2}$-rates obtained are minimax optimal for any dimension $d$ (the lower bounds follow, e.g., via the asymptotic equivalence results in \cite{DR07}, see also \cite{S15,S16}). The $\|\cdot\|_\infty$-rates are then also optimal whenever $d \le 4$, up to log-factors. The sub-optimality of our rate for $d>5$  is related to the presence of common semiparametric `bias terms' in the approximation-theoretic Lemma \ref{lem:semipara_bias} below. 

\subsection{Bernstein-von Mises theorems for $b$}

We now adopt the framework of nonparametric Bernstein-von Mises theorems from \cite{CN13, CN14}, see also the recent contributions \cite{C17, R17, N17, NS17b, MNP17}. The idea is to obtain a Gaussian approximation for the posterior distribution in a  function space in which $1/\sqrt T$-convergence rates can be obtained. We will view the re-centred and re-scaled posterior draws  $\sqrt T(b - \hat b_T)|X^T$ as (conditionally on $X^T$) random vector fields acting linearly on test functions $\phi=(\phi_1, \dots, \phi_d)$ by integration
$$\Big(\phi \mapsto \sqrt T \int_{\T^d} (b-\hat b_T).\phi: \phi \in B^\rho_{1\infty} \big| X^T\Big),$$
and show that a Bernstein-von Mises theorem holds true uniformly in $\phi$ belonging to any bounded subset of the Besov space $B^\rho_{1\infty} ,\rho>d/2$, $d\leq 3$. Equivalently, the limit theorem holds for the probability laws induced by these stochastic processes in the `dual' Banach space $(B^\rho_{1\infty})^*$. The limit will be the tight Gaussian probability measure $\mathcal N_{b_0}$ on $(B^\rho_{1\infty})^*$ induced by the centred Gaussian white noise process $(\mathbb W_0(\phi): \phi \in B^\rho_{1\infty})$ with covariance $$E \mathbb W_0(\phi) \mathbb W_0(\phi') = \langle \phi, \phi' \rangle_{1/\mu_0} =\sum_{j=1}^d\int_{\T^d} \phi_j(x) \phi_j'(x) \mu^{-1}_0(x),~~\phi, \phi' \in B^\rho_{1\infty};$$ its existence is established in the proof of the following theorem. 

By embedding other spaces into $B^\rho_{1\infty}$ one may deduce various further limit theorems from the results below, for example in negative Sobolev spaces $H^{-\rho}=(H^{\rho})^*, \rho>d/2$. For the applications to estimation of $\mu_b$ in the next subsection, this particular choice of Besov space is, however, crucial, and restriction to the simpler scale of Sobolev spaces would be insufficient to obtain the results in Section \ref{bvmsec} below.

For two probability measures $\tau, \tau'$ on a metric space $(S,e)$, define the bounded Lipschitz (BL) metric for weak convergence (p.157 in \cite{D14}) by 
$$\beta_S(\tau,\tau')= \sup_{F: S \to \mathbb R, \|F\|_{Lip} \le 1} \left|\int_S F d(\tau-\tau') \right|,$$
$$\|F\|_{Lip} \equiv \sup_{x\in S}|F(x)| +\sup_{x\neq y, x,y \in S} \frac{|F(x)-F(y)|}{e(x,y)}.$$

\begin{theorem}\label{bvm1}
Let $1 \le d \le 3$, $\rho>d/2$, $a >\max(3d/2 -1,1)$ and let $s \ge a$ be such that $s>a-1+d/2$. Suppose $b_0 \in C^{s}(\T^d) \cap H^{s}(\T^d)$.  Let $\Pi_T$ be the Gaussian prior from (\ref{eq:Gauss_prior_wav}) with $\sigma_l = 2^{-l(\alpha+d/2)}, 0 \le \alpha < a \wedge s-d/2$ and $J$ chosen such that $2^J \approx T^{1/(2a+d)}$. Let $\tilde \Pi_T(\cdot|X^T)$ be the conditional law $\mathcal L(\sqrt T(b - \hat b_T)|X^T)$, where $b \sim \Pi_T(\cdot|X^T)$ and $\hat b_T=E^{\Pi_T}[b|X^T]$ is the posterior mean, and let $\mathcal N_{b_0}$ denote the law in $(B^\rho_{1\infty})^*$ of a centred Gaussian white noise process for $\langle \cdot, \cdot \rangle_{1/\mu_0}$. Then, as $T \to \infty$,
\begin{equation} \label{bvmdis}
\beta_{(B^\rho_{1\infty})^*}(\tilde \Pi(\cdot|X^T), \mathcal N_{b_0}) \to^{P_{b_0}} 0.
\end{equation}
\end{theorem}

As in related situations in \cite{CN13, N17},  the condition $\rho>d/2$ cannot be relaxed as otherwise the limiting process does not exist as a tight probability measure in $(B^\rho_{1\infty})^*$. Also the choices $p=1, q=\infty$ are maximal for Besov spaces. From convergence of moments in (\ref{bvmdis}) we deduce the following.

\begin{theorem}\label{mapas}
Under the conditions of the previous theorem, the MAP estimate $\hat b_T = E^\Pi[b|X^T]$ satisfies, as $T \to \infty$, 
$$\sqrt T (\hat b_T - b_0) \to^d \mathcal N_{b_0} \text{ in } (B^\rho_{1\infty})^*.$$
\end{theorem}

A confidence set for $b$ can now be constructed by using the posterior quantiles to create a multiscale ball around $\hat b_T$, which can be further intersected with smoothness information as in \cite{CN13, CN14} to obtain confidence bands that are valid and near-optimal also in $\|\cdot\|_\infty$-diameter.

As remarked at the end of Section \ref{res1}, the presence of semi-parametric bias terms prevents our proof from giving a Bernstein-von Mises theorem when $d \ge 4$, and also necessitates $s>a-1+d/2$ in Theorem \ref{bvm1}. Unlike in Theorem \ref{super}, the case $d=4$ is excluded as we need to suppress $\log T$-factors to obtain precise limit distributions. Similar phenomena occur in nonparametric smoothing (e.g. Section 3.6 in \cite{GN09}).

\subsection{Bayesian inference on the invariant measure}\label{bvmsec}

We now turn to the problem of making inference on the invariant measure $\mu_{b}$. Frequentist estimators of $\mu_b$ can be suggested directly based on (\ref{erginv}), e.g., \cite{S18}. For the Bayesian statistician, modelling $\mu_b$ directly by a prior is not coherent since $\mu_b$ does not identify the law $P_b^T$ generating the likelihood (\ref{eq:likelihood}) (cf.~the discussion after Proposition \ref{as:true_drift}). Instead, given the MAP estimate $\hat b_T$, we can (numerically) solve (\ref{cpde}) to obtain a point estimate $\mu_{\hat b_T}$. For uncertainty quantification we can generate posterior samples $\mu_{b}|X^T$ from $b \sim \Pi_T(\cdot|X^T)$. Although numerical solvers for elliptic PDEs such as (\ref{cpde}) are available, this algorithm may be computationally expensive. Nonetheless, it gives a principled Bayesian approach to inference on $\mu_b$ that, as the results in this section show, is optimal from an information theoretic point of view.

For the formulation of the following general result, we define spaces $$\mathbb B_r = B^r_{1\infty}(\T^d) \cap L^2(\T^d), \quad r>0,~d\leq 3,$$ normed by $\|\cdot\|_{L^2}  + \|\cdot\|_{B^r_{1\infty}}$; as before the conditional laws $\mathcal L(\sqrt T(\mu_b-\mu_{\hat b_T})|X^T)$ induce stochastic processes in the normed dual space $\mathbb B^*_r$ via actions $$g \mapsto \sqrt T \int_{\T^d} (\mu_b-\mu_{\hat b_T}) g, \quad ~g \in \mathbb B_r,$$ and weak convergence occurs in $\mathbb B^*_r$. We note that the inverse $L_{b_0}^{-1}$ of the generator $L_{b_0}$ from (\ref{eq:generator}) exists as a well-defined mapping from $L^2_{\mu_0}(\T^d)$ into $H^2(\T^d)\cap L^2_0(\T^d)$, see Lemma \ref{regest} in Section \ref{pdeeee}. We  postpone the special case $d=1$ to Theorem \ref{invaclt} below.

\begin{theorem}\label{invabvm}
Let $d=2,3$ and $r>d/2-1$. Under the conditions of Theorem \ref{bvm1}, if $\mu_b, \mu_{\hat b_T}$ are the solutions of (\ref{cpde}) (invariant measures) associated with a posterior draw $b \sim \Pi_T(\cdot|X^T)$ and $\hat b_T=E^{\Pi_T}[b|X^T]$, respectively, then for $\tau(\cdot|X^T)$ the conditional law $\mathcal L(\sqrt T (\mu_{b} - \mu_{\hat b_T})|X^T)$ in $\mathbb B^*_r$ we have
$$\beta_{\mathbb B^*_r} (\tau(\cdot|X^T), \mathcal N_{\mu_{0}}) \to^{P_{b_0}} 0,~\text{ and } \sqrt T (\mu_{\hat b_T} - \mu_0) \to ^d \mathcal N_{\mu_0}~in~\mathbb B^*_r$$ as $T \to \infty,$ where $\mathcal N_{\mu_{0}}$ is the tight Borel probability measure on $\mathbb B^*_r$ induced by the centred Gaussian process $\mathbb M$ with covariance metric $$E\mathbb M(g) \mathbb M(g') = \langle \nabla L_{b_0}^{-1} [\bar g],  \nabla L_{b_0}^{-1} [\bar g'] \rangle_{\mu_{0}},~~\bar g =g - \int_{\T^d} g d\mu_{0},~~ g, g' \in \mathbb B_r.$$ 
\end{theorem}

This theorem has various corollaries, upon using the richness of the spaces $\mathbb B_r, r>d/2-1$. For instance, since $H^r$ embeds continuously into $\mathbb B^r$ on the bounded domain $\T^d$, one deduces weak convergence in $P_{b_0}$-probability of the conditional laws in the negative Sobolev spaces $H^{-r}(\T^d)=(H^r(\T^d))^*$:
\begin{equation*}
\beta_{H^{-r}} \big(\mathcal L(\sqrt T (\mu_{b} - \mu_{\hat b_T})|X^T), \mathcal N_{\mu_0}\big) \to_{T \to \infty}^{P_{b_0}} 0, ~~ r>d/2-1, ~d=2,3.
\end{equation*}

\subsubsection{Bayesian inference on invariant probabilities}Indicator functions of measurable subsets $C$ of $\T^d$ of finite perimeter define elements of $B^1_{1\infty}(\T^d)$ (proved, e.g., as in Lemma 8b, \cite{GN08}) and we can thus make inference on invariant probabilities $\mu_b(C)= \int_{\T^d} 1_C d\mu$ for $d=2,3$. Let $\mathcal C = \mathcal C_K$ be a class of Borel subsets of $(0, 1]^d$ that have perimeter bounded by a fixed constant $K$.  This includes, in particular, all convex subsets of $\T^d$ (e.g., Remark 5 in \cite{GN08}). Then the collection of functions $\{1_C: C \in \mathcal C\}$ is bounded in $B^1_{1\infty}(\T^d) \cap L^2(\T^d)$, and for the resulting set-indexed process of posterior invariant probabilities $(\mu_b(C):C \in \mathcal C), b \sim \Pi_T(\cdot|X^T),$ we deduce from Theorem \ref{invabvm} and the continuous mapping theorem 
\begin{equation*} 
\beta\big( \mathcal L(\sqrt T( \mu_b(\cdot) - \mu_{\hat b_T}(\cdot))|X^T ), \mathcal N_{\mu_0}\big) \to^{P_{b_0}} 0, ~ \sqrt T (\mu_{\hat b_T}-\mu_{0}) \to^d \mathcal N_{\mu_0} \text{ in } \ell^\infty(\mathcal C),
\end{equation*}
as $ T\to \infty$, where $\beta=\beta_{\ell^\infty(\mathcal C)}$, $\ell^\infty(\mathcal C) \supset \mathbb B_r^*$ is the Banach space of bounded functions on $\mathcal C$ (see Proposition 3.7.24 in \cite{ginenickl2016} for a precise definition of $\beta_S$ for non-separable $S$). One further deduces that the estimated invariant probabilities induced by the MAP estimate $\hat b_T$ obey the limit law
\begin{equation*}
\sqrt T \sup_{C \in \mathcal C}|\mu_{\hat b_T}(C)-\mu_{0}(C)| \to^d \sup_{C\in \mathcal C}|\mathbb M(1_C)|<\infty ~a.s., ~T\to \infty.
\end{equation*}

\subsubsection{The one-dimensional case}

We finally turn to the special case $d=1$, where the proof of Theorem \ref{invabvm} needs adaptations as then $r>d/2-1$ can be negative. We obtain a central limit theorem for the invariant probability densities $(\mu_b(x), x \in \T)$ viewed as random functions in  $C(\T)$. 

\smallskip

For $d=1$ the solution map $L_{b}^{-1}$ from before Theorem \ref{invabvm} has a representation $L_{b}^{-1}[g]=\int_\T G_b(\cdot,y)g(y)dy, g \in L^2_{\mu_b}(\T),$ with periodic Green kernel $G_{b}: \T \times \T \to \mathbb R$ such that $G_b(\cdot,x)\in H^1(\T)$ for all $x \in \T$. This follows, e.g., from directly deriving explicit expressions for the solution $v$ of the ODE $bv' + v''/2 = (2\mu_b)^{-1}(\mu_b v')'= g$, where $\mu_b \propto e^{2 B}$ and $B'=b$. 

\begin{theorem}\label{invaclt}
Under the conditions of Theorem \ref{bvm1} with $d=1, a>3/2$, if $\mu_b(x), \mu_{\hat b_T}(x), x \in \T,$ are the invariant probability density functions associated to  $b \sim \Pi(\cdot|X^T), \hat b_T=E^{\Pi_T}[b|X^T]$, respectively, then, as $T \to \infty$,
\begin{equation*} 
\beta_{C(\T)}\big( \mathcal L(\sqrt T(\mu_b - \mu_{\hat b_T})|X^T), \bar {\mathcal N}_{b_0}\big) \to^{P_{b_0}} 0,\text{ and }\sqrt T (\mu_{\hat b_T} - \mu_0) \to^d \bar {\mathcal N}_{b_0} ~in~C(\T),
\end{equation*}
where $\bar {\mathcal N}_{b_0}$ is the Borel probability law in $C(\T)$ induced by the centred Gaussian random function $(\bar {\mathbb M}(x): x \in \T)$ with covariance $$E\bar{\mathbb M}(x) \bar{\mathbb M}(x') = \int_\T \frac{d}{dy} G_{b_0}(y,x) \frac{d}{dy} G_{b_0}(y,x') d\mu_0(y) ,~~x,x' \in \T.$$ 
\end{theorem}
In \cite{AS18a} an analogue of the second limit in the above theorem was obtained for an estimator based on smoothing the empirical measure $\hat \mu_T$ from (\ref{erginv}). Their proof is very different from ours and based on first establishing that their estimator is asymptotically close to the local time of the diffusion process, in conceptual analogy to the i.i.d.~setting \cite{GN09}. 

\subsubsection{Information lower bounds}
The LAN expansion of our measurement model under $P_{b_0}$ is obtained in Lemma \ref{lem:LAN_expansion} below, with LAN-inner product $\langle \cdot, \cdot \rangle_{\mu}$. Standard arguments from asymptotic semiparametric statistics (\cite{vdvaart1998}, Chapter 25) then imply that the asymptotic variance occurring in Theorems \ref{bvm1} and \ref{mapas} is optimal in an information-theoretic sense. This is also true in the case of Theorem \ref{invabvm}, where inference on a non-linear functional $\Phi_g(b)= \int_{\T^d} g d\mu_b, g \in L^2(\T^d)$, of $b$ is considered. Indeed, the expansion
$$\Phi_g(b+h) - \Phi_g(b)  = \langle \nabla L_b^{-1}[\bar g], h \rangle_{\mu_b} +o(\|h\|_\infty), \quad \bar g = g -\int_{\T^d} g d\mu_b,$$ follows from the proof of Theorem \ref{invabvm}. Thus arguing as in Section 7.5 in \cite{N17}, the information lower bound for estimating $\Phi_g(b)$ from our observations is  $$ \|\nabla L_b^{-1}[\bar g] \|_{\mu_b}^2 = \int_{\T^d}\left\|\nabla L_b^{-1}\left[g-\int_{\T^d} g d\mu_b\right](x)\right\|^2d\mu_b(x), ~\text{ any } d \ge 1.$$ Examining the proof of Theorem \ref{invaclt}, a similar remark applies to the covariance appearing in that theorem. See also \cite{AS18a} for the case $d=1$.

\section{Proofs of main results}

\subsection{Concentration of measure tools for multi-dimensional diffusions} \label{marto}

The following results provide uniform stochastic control of functionals of the diffusion process (\ref{eq:SDE}) with periodic drift $b_0$ in terms of metric entropy bounds via a metric $d_L$ involving the inverse generator $L_{b_0}^{-1}$ from Lemma \ref{regest}.

\begin{lemma} \label{lem:maxim}
Suppose $b_0 \in C^{(d/2+\kappa)\vee 1}(\T^d)$ and let $\mathcal{F}_T \subset L_{\mu_0}^2(\T^d) \cap H^{d/2+\kappa}(\T^d)$ for some $\kappa>0$ be such that $0 \in \mF_T$. Define the empirical process $$\G_T[f] = \frac{1}{\sqrt T} \int_0^T f(X_s)ds,~~f \in \mathcal F_T,$$ the pseudo-distance $d_L$ on $\mF_T$ by
\begin{align}\label{eq:metric}
d^2_L(f,g) = \sum_{i=1}^d \left\| \partial_{x_i} L_{b_0}^{-1}[f-g] \right\|_\infty^2,
\end{align}
and let $D_{\mF_T}$ be the $d_L$-diameter of $\mF_T$. Further set
$$J_{\mF_T} = J(\mF_T,6d_L,D_{\mF_T})= \int_0^{D_{\mF_T}} \sqrt{\log 2N(\mF_T,6d_L,\tau)}d\tau,$$
where $N(\mF_T,6d_L,\tau)$ denotes the covering number of the set $\mF_T$ by $d_L$-balls of radius $\tau/6$. Then
\begin{align*}
E_{b_0} \sup_{f\in \mF_T} |\G_T(f)| \leq \frac{2}{\sqrt{T}} \sup_{f\in \mF_T}\|L_{b_0}^{-1}[f]\|_\infty + 4\sqrt{2}J_{\mF_T},
\end{align*}
and for any $x>0$,
\begin{align*}
P_{b_0} \Big( \sup_{f\in \mF_T}  \left| \G_T(f) \right| \geq \sup_{f\in \mF_T}\frac{2\|L_{b_0}^{-1}[f]\|_\infty}{\sqrt{T}} + J_{\mF_T}(4\sqrt{2}+192x )  \Big) \leq e^{-x^2/2}.
\end{align*}
\end{lemma}

\begin{proof}
By Lemma \ref{regest} and the Sobolev embedding theorem, the Poisson equation $Lu = L_{b_0}u = f$ has a unique solution $L^{-1}[f] \in H^{d/2+\kappa+2} \cap L^2_0 \subset C^2$ satisfying $LL^{-1}[f]=f$ for any $f\in \mF_T$. We may therefore define for $f\in \mF_T$,
\begin{align*}
Z_T(f) & = \int_0^T \nabla L^{-1}[f](X_s).dW_s \\
& =  L^{-1}[f](X_T)-L^{-1}[f](X_0) - \int_0^T LL^{-1}[f](X_s)ds ,
\end{align*}
where we have used It\^o's lemma (Theorem 39.3 in \cite{B11}). Since
\begin{align*}
\sup_{f\in \mF_T} \left| \int_0^T f(X_s)ds \right| - 2 \sup_{f\in \mF_T} \|L^{-1}[f]\|_\infty \leq \sup_{f\in \mF_T} |Z_T(f)|,
\end{align*}
it suffices to control $\sup_{f\in \mF_T} |Z_T(f)|$. For fixed $f\in \mF_T$, $ Z_T(f)$ is a continuous square integrable local martingale with quadratic variation
\begin{align*}
[Z_\cdot (f)]_T =  \int_0^T \|\nabla L^{-1}[f](X_s)\|^2 ds \leq T\sum_{i=1}^d \| \partial_{x_i} L^{-1}[f]\|_\infty^2 = Td_L^2(f,0).
\end{align*}
Recall Bernstein's inequality for continuous local martingales (p.~153 of \cite{revuz1999}): if $M$ is a continuous local martingale vanishing at 0 with quadratic variation $[M]$, then for any stopping time $T$ and any $y,K>0$,
\begin{align}\label{eq:Bernstein}
P \big( \sup_{0\leq t \leq T}|M_t| \geq y, [M]_T \leq K \big) \leq 2e^{-\frac{y^2}{2K}}.
\end{align}
Applying this to $Z_T(f)$ gives for any $f\in\mF_T$ and $x>0$,
\begin{align*}
P_{b_0} \big( |Z_T(f)| \geq \sqrt{T} x \big) = P_{b_0} \big( |Z_T(f)|\geq \sqrt{T} x, [Z_\cdot(f)]_T \leq Td_L^2(f,0) \big) \leq 2e^{-\frac{x^2}{2d_L^2(f,0)}}.
\end{align*}
Since $L^{-1}$ is linear, so is $f\mapsto Z_T(f)$, and consequently
\begin{align*}
P_{b_0} \big( |Z_T(f)-Z_T(g)| \geq \sqrt{T} x \big) \leq 2\exp\left( -\frac{x^2}{2d_L^2(f,g)} \right),
\end{align*}
a non-asymptotic inequality. The process $(T^{-1/2} Z_T(f):f\in\mF_T)$ is thus mean-zero and subgaussian with respect to $d_L$. From this we deduce that $E_{b_0} \sup_{f\in \mF_T} T^{-1/2} |Z_T(f)| \leq 4\sqrt{2}J_{\mF_T}$ by the usual chaining bound for subgaussian processes (e.g.,~Theorem 2.3.7 of \cite{ginenickl2016} - the factor $6$ scales the subgaussian constant, see after Definition 2.3.5 of \cite{ginenickl2016}). This chaining bound extends to exponential $\psi_2$-Orlicz norms $\|\cdot\|_{\psi_2}$ (see Exercise 2.3.1 of \cite{ginenickl2016}), so one further has $\left\| \sup_{f\in \mF_T} T^{-1/2}|Z_T(f)|  \right\|_{\psi_2} \leq 16\sqrt{6} J_{\mF_T}$. Using Lemma 2.3.1 of \cite{ginenickl2016} and that for any random variable $X$, $\|X-\E X\|_{\psi_2} \leq 2\|X\|_{\psi_2}$, we obtain for any $x>0$,
\begin{align*}
P_{b_0} \Big( \sup_{f\in\mF_T} \frac{|Z_T(f)|}{\sqrt{T}} \geq E_{b_0} \sup_{f\in \mF_T} \frac{|Z_T(f)|}{\sqrt{T}} + x \Big) \leq \exp \left(-\frac{x^2}{2 (196J_{\mF_T})^2}\right).
\end{align*}
Using the expectation bound just derived, the above inequality yields
\begin{align*}
P_{b_0} \left( \sup_{f\in \mF_T} T^{-1/2} |Z_T(f)| \geq 4\sqrt{2}J_{\mF_T} + 196 J_{\mF_T}x \right) \leq e^{-\frac{x^2}{2}}.
\end{align*}
Combining the above gives the required subgaussian inequality.
\end{proof}

We now establish usable bounds for the metric $d_L$. The following is a special case of the Runst-Sickel lemma.

\begin{lemma}[\cite{runst1996}, p. 345]\label{lem:runst}
For $t > 0$ and any bounded $f,g \in H^t(\T^d)$,
\begin{align*}
\|fg\|_{H^t} \leq C(t,d) \left( \|f\|_{H^t}\|g\|_{\infty} + \|g\|_{H^t} \|f\|_{\infty} \right).
\end{align*}
\end{lemma}

\begin{lemma}\label{lem:metric}
Suppose $b_0 \in C^s(\T^d)$ for $s>\max(d/2-1,0)$. Then for any $0<\kappa<s-d/2+1$ (or $\kappa=0$ if $d=1$) and $f,g\in L_{\mu_0}^2(\T^d)$, the pseudo-distance $d_L$ in \eqref{eq:metric} satisfies
\begin{align*}
d_L(f,g) \leq C(d,\kappa,b_0) \|f-g\|_{H^{(d/2+\kappa-1)_+}},
\end{align*}
where $H^0=L^2$. Moreover, let $V_J$  denote the span of all wavelets up to resolution level $J$ of an $S$-regular wavelet basis of $L^2(\T^d)$. If $\gamma,\rho \in V_J$ and $0\leq p<S$, then for $C = C(p,d,\Phi,\|\mu_0\|_{\infty})$,
\begin{align*}
\left\|\gamma\rho - \langle \gamma , \rho \rangle_{\mu_0} \right\|_{H^p} \leq C 2^{J(p+d/2)}\|\gamma\|_{L^2}\|\rho\|_{L^2}.
\end{align*}
\end{lemma}

\begin{proof}
If $d\geq 2$, then for any $0<\kappa <s-d/2+1$, by the Sobolev embedding theorem and Lemma \ref{regest} there exists $C=C(d,\kappa,b_0)$ such that
\begin{equation*}
d_L^2(f,g) = \sum_{i=1}^d \|\partial_{x_i}L^{-1}[f-g]\|_\infty^2  \leq C\|L^{-1}[f-g]\|_{H^{d/2+\kappa+1}}^2 \leq C \|f-g\|_{H^{d/2+\kappa-1}}^2.
\end{equation*}
If $d=1$, one similarly has $d_L^2(f,g) \leq C\|L^{-1}[f-g]\|_{H^2}^2 \leq C \|f-g\|_{L^2}^2.$ For the second statement, if $p>0$, then the triangle inequality, Lemma \ref{lem:runst} and the Cauchy-Schwarz inequality bound the quantity in question by
\begin{align*}
C(p,d)( \|\gamma\|_{H^p}\|\rho\|_\infty + \|\rho\|_{H^p}\|\gamma\|_\infty) + \|\mu_0\|_{\infty} \|\gamma\|_{L^2} \|\rho\|_{L^2}\|1\|_{H^p}.
\end{align*}
If $p=0$, one instead uses the simpler bound $\|\gamma\rho\|_{L^2} \leq \|\gamma\|_{L^2} \|\rho\|_\infty$. By the wavelet characterisation of the Sobolev norm,
\begin{align}\label{eq:wav_bd_1}
\|\gamma\|^2_{H^p} = \sum_{l \le J}\sum_r 2^{2lp} |\langle \gamma, \Phi_{l,r}\rangle|^2  \le 2^{2Jp} \|\gamma\|^2_{L^2} .
\end{align}
Using Cauchy-Schwarz and that for all $l\geq 0$, $\|\sum_r \Phi^2_{l,r}\|_\infty \leq C(\Phi) 2^{ld}$,
\begin{align}\label{eq:wav_bd_2}
\|\gamma\|_\infty \le \sup_x \sum_{l \le J}\sum_r  |\langle \gamma, \Phi_{l,r}\rangle| |\Phi_{l,r}(x)| \le C(\Phi) 2^{Jd/2} \|\gamma\|_{L^2} .
\end{align}
Applying these bounds to $\gamma,\rho\in V_J$ gives the result.
\end{proof}

\subsection{A restricted isometry inequality for $h_T^2(\cdot, \cdot)$ on $V_J^{\otimes d}$}
We next consider the action on the spaces $V_J^{\otimes d}$ of the random distance $h_T$ defined by
\begin{align*}
Th_T^2(b_1,b_2) \equiv  \int_0^T \|b_1(X_s)-b_2(X_s)\|^2 ds = \sum_{j=1}^d  \int_0^T |b_{1,j}(X_s)-b_{2,j}(X_s)|^2 ds.
\end{align*}
From the preceding concentration inequalities, and using a commonly used contraction principle to bound minimal eigenvalues of random matrices (e.g., Section 5.6 in \cite{V10}), we establish the following key inequality.

\begin{lemma}\label{lem:bilinear}
Suppose $b_0\in C^s(\T^d)$ for $s>\max(d/2,1)$ and let $J \in \mathbb N$. For $V_J$ as in (\ref{wavbas}) set $v_J := \text{dim}(V_J) = O(2^{Jd})$. Then for any $0<\kappa<s-d/2+1$ (or $\kappa=0$ if $d=1$), there exist positive constants $c_0=c_0(b_0)$ and $C=C(d,b_0,\kappa,\Phi)$ such that for any $x>0$,
\begin{align*}
P_{b_0} \Bigg[ \sup_{b,\bar{b}\in V_J^{\otimes d}:b\neq \bar{b}} \left| \frac{h_T^2(b,\bar{b})-\|b-\bar{b}\|_{\mu_0}^2}{\|b-\bar{b}\|_{\mu_0}^2}\right| \geq \frac{2^{J[\frac{d}{2}+ (\frac{d}{2}+\kappa-1)_+]}}{\sqrt{T}/C} (1+x) \Bigg]  \leq d e^{c_0v_J -\frac{x^2}{2}}.
\end{align*}
\end{lemma}

\begin{proof}
Let $b_j,\bar{b}_j \in V_J$ and write $b=(b_1,\dots,b_d)$, $\bar{b} = (\bar{b}_1,\dots,\bar{b}_d)$ with $b_j = \sum_{l\leq J,r} \theta_{l,r,j} \Phi_{l,r}$ and $\bar{b}_j = \sum_{l\leq J,r} \bar{\theta}_{l,r,j} \Phi_{l,r}$. Then
\begin{align*}
h_T^2 (b,\bar{b}) & = \frac{1}{T} \sum_{j=1}^d \int_0^T \left( \sum_{l,r} (\theta_{l,r,j} - \bar{\theta}_{l,r,j})\Phi_{l,r}(X_s) \right)^2 ds \\
& = \sum_{j=1}^d \sum_{l,r} \sum_{l',r'} (\theta_{l,r,j} - \bar{\theta}_{l,r,j})(\theta_{l',r',j} - \bar{\theta}_{l',r',j}) \frac{1}{T}  \int_0^T \Phi_{l,r}(X_s)\Phi_{l',r'}(X_s)ds \\
& = \sum_{j=1}^d (\theta_{j,\cdot} - \bar{\theta}_{j,\cdot})^T \hat{\Gamma} (\theta_{j,\cdot} - \bar{\theta}_{j,\cdot}),
\end{align*}
where $\hat{\Gamma}_{(l,r)(l',r')} = \tfrac{1}{T} \int_0^T \Phi_{l,r}(X_s)\Phi_{l',r'}(X_s)ds$, so that $\hat{\Gamma}$ is a $v_J \times v_J$ symmetric matrix. Similarly,
\begin{align*}
\|b-\bar{b}\|_{\mu_0}^2 = \sum_{j=1}^d (\theta_{j,\cdot} - \bar{\theta}_{j,\cdot})^T \Gamma (\theta_{j,\cdot} - \bar{\theta}_{j,\cdot}),
\end{align*}
where $\Gamma_{(l,r)(l',r')} = \int_{\T^d} \Phi_{l,r}(x)\Phi_{l',r'}(x)d\mu_0(x)$. Denote the quantity on the r.h.s.~in the inequality in Lemma \ref{lem:bilinear} by $\zeta_T = CT^{-1/2} 2^{J[d/2+(d/2+\kappa-1)_+]}(1+x)$. Since $(\theta_{j,\cdot} - \bar{\theta}_{j,\cdot})^T \Gamma (\theta_{j,\cdot} - \bar{\theta}_{j,\cdot}) = \|b_j - \bar{b}_j\|_{\mu_0}^2 \geq 0$ for all $j$, applying a union bound to the probability in Lemma \ref{lem:bilinear} gives
\begin{align*}
\sum_{j=1}^d P_{b_0} \left( \sup_{\theta_{j\cdot},\bar{\theta}_{j\cdot} \in \R^{v_J}: (\theta_{j\cdot}-\bar{\theta}_{j\cdot})^T \Gamma(\theta_{j\cdot}-\bar{\theta}_{j\cdot}) \neq 0} \left| \frac{(\theta_{j\cdot}-\bar{\theta}_{j\cdot})^T (\hat{\Gamma}-\Gamma)(\theta_{j\cdot}-\bar{\theta}_{j\cdot})}{(\theta_{j\cdot}-\bar{\theta}_{j\cdot})^T \Gamma(\theta_{j\cdot}-\bar{\theta}_{j\cdot})} \right| \geq \frac{\zeta_T}{d} \right).
\end{align*}
(Note that at least one $(\theta_{j\cdot}-\bar{\theta}_{j\cdot})^T \Gamma(\theta_{j\cdot}-\bar{\theta}_{j\cdot}) \neq 0$ by assumption and the above supremum is maximized when $\theta_{j\cdot}\neq\bar{\theta}_{j\cdot}$, so the denominator is well-defined for all $j$). Setting $u=(\theta_{j\cdot}-\bar{\theta}_{j\cdot}) \in \R^{v_J}$ and using the bilinearity of the above quadratic form, each of the previous probabilities, which are all equal, are bounded by
\begin{align}\label{eq:bilin_prob}
P_{b_0} \left( \sup_{u\in \Theta} |u^T \Lambda u| \geq \zeta_T/d \right),
\end{align}
where $\Theta = \{ u \in \R^{v_J}: u^T \Gamma u \leq 1 \}$ and $\Lambda = \hat{\Gamma}-\Gamma$. Let $\|u\|_\Gamma^2 := u^T \Gamma u$, $u \in \R^{v_J}$, and for $0<\delta<1$, let $(u^l)_{l=1}^{N(\delta)}$ be a minimal $\delta$-covering of $\Theta$ in $\|\cdot\|_\Gamma$-distance. For every $u \in \Theta$, let $u^l = u^l(u)$ denote the closest point in this $\delta$-covering, so that $\|u - u^l\|_\Gamma \leq \delta$. By bilinearity, for any $u \in \Theta$,
\begin{align*}
|(u-u^l)^T \Lambda (u-u^l)| \leq \delta^2 \sup_{w \in \Theta} |w^T \Lambda w|.
\end{align*}
For any $u \in \Theta$, set $g_u = \sum_{l\leq J,r} u_{l,r}\Phi_{l,r}$. By Proposition \ref{as:true_drift}, $\|u\|_{\R^{v_J}} = \|g_u\|_{L^2} \leq \|1/\mu_0\|_\infty^{1/2} \|g_u\|_{\mu_0} = \|1/\mu_0\|_\infty^{1/2} \|u\|_\Gamma$. For $(\lambda_i)_{i=1}^{v_J}$ the eigenvalues of the symmetric matrix $\Lambda$ and $\lambda_{max} = \max_i |\lambda_i|$, applying the Cauchy-Schwarz inequality gives
\begin{align*}
|(u-u^l)^T\Lambda u^l | &\leq \|u-u^l\|_{\R^{v_J}} \|\Lambda u^l\|_{\R^{v_J}} \leq \delta \|1/\mu_0\|_\infty^{1/2} \lambda_{max} \|u^l\|_{\R^{v_J}}  \\
& \leq \delta \|1/\mu_0\|_\infty \sup_{v:\|v\|_{\R^{v_J}}\leq 1} |v^T\Lambda v|,
\end{align*}
where the last inequality follows from p.234 of \cite{horn2013}. Since $\sup_{v:\|v\|_{\R^{v_J}}\leq 1} |v^T\Lambda v| \leq \|\mu_0\|_\infty \sup_{w \in \Theta} |w^T \Lambda w|$, then $|(u-u^l)^T\Lambda u^l|  \leq \delta \|1/\mu_0\|_\infty \|\mu_0\|_\infty \sup_{w \in \Theta} |w^T \Lambda w|$ for all $u \in \Theta$. Combining the above yields for $0 <\delta<1$,
\begin{align*}
\sup_{u \in \Theta} |u^T\Lambda u| \leq (\delta^2 + 2 \delta \|1/\mu_0\|_\infty\|\mu_0\|_\infty) \sup_{w\in\Theta} |w^T\Lambda w| +  \max_{1\leq l \leq N(\delta)} |(u^l)^T \Lambda u^l|,
\end{align*}
and for $\delta_0 = \delta_0(\mu_0)$ small enough that $\delta_0^2 + 2\|1/\mu_0\|_\infty \|\mu_0\|_\infty \delta_0 \leq 1/2$, 
\begin{equation} \label{maxred}
\sup_{u \in \Theta} |u^T\Lambda u| \leq 2 \max_{1\leq l \leq N(\delta_0)} |(u^l)^T \Lambda u^l|.
\end{equation} A union bound now yields that \eqref{eq:bilin_prob} is bounded by $N(\delta_0) \sup_{u \in \Theta} P_{b_0} (|u^T\Lambda u| \geq \zeta_T/(2d))$. The covering number of the unit ball in a $v_J$-dimensional space is bounded by $N(\delta_0) \leq (C/\delta_0)^{v_J} = e^{c_0 v_J}$ (Proposition 4.3.34 of \cite{ginenickl2016}).

For $u \in \Theta$, set $f_u(x) = g_u(x)^2 - \langle g_u,g_u\rangle_{\mu_0} \in L_{\mu_0}^2(\T^d) \cap H^S(\T^d)$, where $S>d/2$ is the regularity of the wavelet basis. Since also $b_0 \in C^s$ with $s>d/2$, applying Lemma \ref{lem:maxim} to the class $\mF = \{f_u,0\}$ and noting that $u^T \Lambda u = T^{-1} \int_0^T f_u(X_t)dt$ yields
\begin{align}\label{eq:bilin_exp_ineq}
P_{b_0} \left( |u^T\Lambda u| \geq CT^{-1} \|L^{-1}[f_u]\|_\infty + CT^{-1/2}d_L(f_u,0) (1+x) \right) \leq e^{-x^2/2}.
\end{align}
For $0<\kappa<s-d/2+1$ (or $\kappa=0$ if $d=1$), applying Lemma \ref{lem:metric} with $\gamma = \rho = g_u \in V_J$ and $p=(d/2+\kappa-1)_+$ gives
\begin{align*}
d_L(f_u,0) \leq C \|g_u^2 - \langle g_u,g_u\rangle_{\mu_0} \|_{H^{(d/2+\kappa-1)_+}} \leq C 2^{J[d/2+(d/2+\kappa-1)_+]}\|g_u\|_{L^2}^2. 
\end{align*}
By Proposition \ref{as:true_drift}, $\|g_u\|_{L^2}^2 \leq \|1/\mu_0\|_\infty \|g_u\|_{\mu_0}^2 = \|1/\mu_0\|_\infty u^T \Gamma u \leq \|1/\mu_0\|_\infty$, so that $d_L(f_u,0) \leq C 2^{J[d/2+(d/2+\kappa-1)_+]}$ for any $u \in \Theta$. Applying the Sobolev embedding theorem, Lemma \ref{regest} and Lemma \ref{lem:metric} as above, $\|L^{-1}[f_u]\|_\infty \lesssim \|f_u\|_{H^{(d/2+\kappa-2)_+}} \lesssim  2^{J[d/2+(d/2+\kappa-1)_+]}$ for $\kappa$ as above and any $u \in \Theta$. Substituting this into \eqref{eq:bilin_exp_ineq} gives
\begin{align*}
\sup_{u\in\Theta} P_{b_0} \left( |u^T\Lambda u| \geq CT^{-1/2}2^{J[d/2+(d/2+\kappa-1)_+]}(1+x)  \right) \leq e^{-x^2/2},
\end{align*}
where the right-hand side equals $\zeta_T$ up to constants. Combining the last inequality with (\ref{maxred}) and the remarks after it  completes the proof.
\end{proof}

\subsection{Proof of Theorem \ref{thm:contraction}}

As a first step we obtain a convergence rate in the `random Hellinger distance $h_T$' defined before Lemma \ref{lem:bilinear} corresponding to the regression problem posed by equation (\ref{eq:SDE}). This random semimetric arises naturally in the classical testing approach (see \cite{GvdV17}, and formulated in the Brownian semi-martingale setting relevant here by van der Meulen et al. \cite{vdmeulen2006}), since the log-likelihood with respect to $P_{b_0}^T$ can be expressed as $M-\tfrac{1}{2}[M]$, where $M$ is a continuous local martingale with quadratic variation $[M]_T = Th_T^2(b,b_0)$. The next result is a combination of Theorem 2.1 and Lemma 2.2 of \cite{vdmeulen2006}, restated in the present context. The proof relies on martingale arguments which generalize to the multidimensional setting without difficulty, hence the proof is left to the reader. 

Consider the statistical experiments $(P_b^T:b\in \mathfrak{B}_T)$, where the parameter spaces $\mathfrak{B}_T$, which are allowed to vary with $T$, are arbitrary sets equipped with $\sigma$-algebras satisfying mild measurability conditions, see Section 2 of \cite{vdmeulen2006}. In particular, these are satisfied by the finite-dimensional spaces considered in Theorem \ref{thm:contraction}.

\begin{theorem}\label{thm:contraction_general}
Let $\varepsilon_T \rightarrow 0$ be such that $T\varepsilon_T^2 \rightarrow \infty$. Suppose that for any $C_1>0$, there exist measurable sets $\cal{B}_T$ and $C_2, C_3>0$ such that
\begin{equation}\label{eq:thm_remaining_mass}
\Pi_T(\mathcal{B}_T^c) \leq e^{-C_1T\varepsilon_T^2},
\end{equation}
\begin{equation}\label{eq:thm_entropy}
\log N( \mathcal{B}_T,\|\cdot\|_{\mu_0},\varepsilon_T) \leq C_2 T\varepsilon_T^2,
\end{equation}
\begin{equation}\label{eq:thm_small_ball}
\Pi_T (b: \|b-b_0\|_{\mu_0} \leq \varepsilon_T) \geq e^{-C_3T\varepsilon_T^2}.
\end{equation}
Assume further that for every $\gamma>0$ there exist $c_\gamma,C_\gamma>0, D_\gamma\geq 0$ such that
\begin{equation}\label{eq:norm_equiv}
\begin{split}
& \liminf_{T\rightarrow \infty} P_{b_0} \Big(c_\gamma \|b-b_0\|_{\mu_0} \leq h_T(b,b_0), \forall b\in \mathfrak{B}_T \textup{ s.t. } h_T(b,b_0) \geq D_\gamma \varepsilon_T, \text{ and} \\
&   h_T(b_1,b_2) \leq C_\gamma \|b_1-b_2\|_{\mu_0}, \forall b_1,b_2\in \mathfrak{B}_T \textup{ s.t. } h_T(b_1,b_2) \geq D_\gamma \varepsilon_T\Big) \geq 1-\gamma.
\end{split}
\end{equation}
Then for every $M_T\rightarrow \infty$, $\Pi_T(b\in\mathfrak{B}_T: h_T(b,b_0) \geq M_T\varepsilon_T|X^T) \rightarrow_{T \to \infty}^{P_{b_0}} 0.$
\end{theorem}
The proof of the theorem implies in particular that the denominator in (\ref{postform}) is non-zero on events of $P_{b_0}$-probability approaching one. We now turn to the proof of Theorem \ref{thm:contraction} and verify the conditions \eqref{eq:thm_remaining_mass}-\eqref{eq:norm_equiv} of Theorem \ref{thm:contraction_general}. By Proposition \ref{as:true_drift}, $\|\cdot\|_{L^2}$ and $\|\cdot\|_{\mu_0}$ are equivalent norms. Applying Theorem 4.5 of \cite{vdvaart2008} (see also Sections 11.3 and 11.4.5 in \cite{GvdV17}), there exist measurable sets $B_T \subset V_J$ such that for $\varepsilon_T = T^{-\frac{a\wedge s}{2a+d}}(\log T)$,

(i) $\log N(B_T,\|\cdot\|_\infty,3\varepsilon_T) \leq 6CT\varepsilon_T^2$, 

(ii) $\Pi_T (b_j\not\in B_T) \leq e^{-CT\varepsilon_T^2}$ and 

(iii) $\Pi_T(\|b_j-b_{0,j}\|_\infty < 4\varepsilon_T) \geq e^{-T\varepsilon_T^2}$. 

The set $\mathcal{B}_T =  B_T^{\otimes d}$ satisfies $\log N(\mathcal{B}_T,\|\cdot\|_{\mu_0},3d^{1/2}\|\mu_0\|_\infty^{1/2} \varepsilon_T) \leq 6dCT\varepsilon_T^2$ and $\Pi_T (\mathcal{B}_T^c) \leq de^{-CT\varepsilon_T^2}$, which verifies \eqref{eq:thm_remaining_mass} and \eqref{eq:thm_entropy} for (a constant multiple of) $\varepsilon_T$. Further,
\begin{align*}
\Pi_T \left( \|b-b_0\|_{\mu_0} \leq 4\sqrt{d}\|\mu_0\|_\infty^{1/2} \varepsilon_T \right) & \geq \Pi_T \left( \sup_{j=1,\dots,d} \|b_j-b_{0,j}\|_{L^2} \leq 4 \varepsilon_T \right) \\
& \geq \prod_{j=1}^d \Pi_T \left( \|b_j-b_{0,j}\|_\infty \leq 4\varepsilon_T \right) \geq e^{-dT\varepsilon_T^2},
\end{align*}
thereby verifying \eqref{eq:thm_small_ball} for (a constant multiple of) $\varepsilon_T$.

We now verify \eqref{eq:norm_equiv}. Since $\Pi_T(V_J^{\otimes d})=1$, we may take as parameter space $\mathfrak{B}_T = V_J^{\otimes d} \cup \{b_0\}$. Let $b_{0,j,J}$ denote the orthogonal projection of $b_{0,j}$ onto $V_J$ and set $b_{0,J} = (b_{0,1,J},\dots, b_{0,d,J})$. Since $b_0\in C^s$, $\|b_{0,j}-b_{0,j,J}\|_\infty \leq C(b_0) 2^{-Js} \leq C(b_0)\varepsilon_T$, so that for our choice of $J$ this yields $h_T(b_0,b_{0,J}) \leq r\varepsilon_T$ and $\|b_0-b_{0,J}\|_{\mu_0} \leq \|\mu_0\|_\infty^{1/2} \|b_0-b_{0,J}\|_{L^2} \leq r\varepsilon_T$ for some $r = r(d,b_0)$. By considering the cases $b_1 \in V_J^{\otimes d}$ and $b_1 = b_0$ separately, the event in \eqref{eq:norm_equiv} therefore contains the event
\begin{align*}
& \{ c_\gamma \|b-b_{0,J}\|_{\mu_0} + c_\gamma r\eps_T \leq h_T(b,b_{0,J}) - r\eps_T \\
& \quad \text{ and } h_T(b,b_{0,J}) + r\eps_T \leq C_\gamma \|b-b_{0,J}\|_{\mu_0} - C_\gamma r \eps_T, \\
&  \quad  \forall b\in V_J^{\otimes d} \textup{ with } h_T(b,b_{0,J}) \geq (D_\gamma -r) \varepsilon_T \} \\
&  \cap \{  h_T(b_1,b_2) \leq C_\gamma \|b_1-b_2\|_{\mu_0}, \forall b_1,b_2\in V_J^{\otimes d} \textup{ with } h_T(b_1,b_2) \geq D_\gamma \varepsilon_T \}.
\end{align*}
For $D_\gamma$ large enough that $(D_\gamma -r) \geq \max\{ (C_\gamma+1)r,2(c_\gamma+1)r\}$, the last event contains
\begin{align*}
& \{ 2c_\gamma \|b-b_{0,J}\|_{\mu_0} \leq  h_T(b,b_{0,J}) \leq \tfrac{1}{2} C_\gamma \|b-b_{0,J}\|_{\mu_0}, \\
& \quad \forall b\in V_J^{\otimes d} \textup{ with } h_T(b,b_{0,J}) \geq (D_\gamma -r) \varepsilon_T \} \\
&  \cap \{  h_T(b_1,b_2) \leq C_\gamma \|b_1-b_2\|_{\mu_0}, \forall b_1,b_2\in V_J^{\otimes d} \textup{ with } h_T(b_1,b_2) \geq D_\gamma \varepsilon_T \} \\
& \supset \{ 2c_\gamma \|b_1-b_2\|_{\mu_0} \leq  h_T(b_1,b_2) \leq \tfrac{1}{2} C_\gamma \|b_1-b_2\|_{\mu_0},  \forall b_1,b_2 \in V_J^{\otimes d}  \}
\end{align*}
since $b_{0,J} \in V_J^{\otimes d}$. It thus suffices to lower bound the $P_{b_0}$-probability of the last event.  For $C_\gamma>2$ and $0<c_\gamma<1/2$, this probability equals
\begin{equation}\label{eq:norm_equiv_prob}
\begin{split}
& P_{b_0} \left( 4c_\gamma^2-1 \leq \frac{h_T^2(b_1,b_2)}{\|b_1-b_2\|_{\mu_0}^2} -1 \leq \frac{1}{4}C_\gamma^2-1, \, \,\forall b_1,b_2 \in V_J^{\otimes d}, b_1 \neq b_2 \right) \\
& \geq 1 - P_{b_0} \left( \sup_{b_1,b_2\in V_J^{\otimes d}: b_1\neq b_2} \frac{|h_T^2(b_1,b_2)-\|b_1-b_2\|_{\mu_0}^2|}{\|b_1-b_2\|_{\mu_0}^2}  > K_\gamma \right),
\end{split}
\end{equation}
where $K_\gamma = \min \{ 1-4c_\gamma^2 , C_\gamma^2/4-1\}>0$. Since $b_0\in C^s, s> \max(d/2,1)$, Lemma \ref{lem:bilinear} with $x = \sqrt{2}M_02^{Jd/2}\rightarrow \infty$ and $M_0>1$ large enough yields
\begin{align*}
& P_{b_0} \left( \sup_{b_1,b_2 \in V_J^{\otimes d}:b_1\neq b_2} \left| \frac{h_T^2(b_1,b_2)-\|b_1-b_2\|_{\mu_0}^2}{\|b_1-b_2\|_{\mu_0}^2}\right| \geq \frac{CM_0}{\sqrt{T}} 2^{J[d+(d/2+\kappa-1)_+]} \right) \\
& \quad \quad \leq d e^{(c_0'-M_0^2) 2^{Jd}} \rightarrow 0,
\end{align*}
where $0<\kappa<s-d/2+1$ (or $\kappa=0$ if $d=1$). Since $T^{-1/2} 2^{J[d+(d/2+\kappa-1)_+]}  \rightarrow 0$ as $T\rightarrow \infty$ for $a>\max(d-1,1/2)$ and $\kappa>0$ small enough, the right-hand side of \eqref{eq:norm_equiv_prob} equals $1-o_{P_{b_0}}(1)$ as $T\rightarrow \infty$. This verifies \eqref{eq:norm_equiv} for $C_\gamma >2$, $0<c_\gamma<1/2$ and $D_\gamma>0$ large enough, so that applying Theorem \ref{thm:contraction_general} yields posterior contraction rate $\eps_T$ in the Hellinger distance $h_T$.

We have shown above that for $C_\gamma>2$ and $0<c_\gamma<1/2$,
$$ P_{b_0} ( 2c_\gamma \|b_1-b_2\|_{\mu_0} \leq  h_T(b_1,b_2) \leq \tfrac{1}{2} C_\gamma \|b_1-b_2\|_{\mu_0},  \forall b_1,b_2 \in V_J^{\otimes d})\to 1.$$
Using again the bias bounds $h_T(b_0,b_{0,J}) \leq r\varepsilon_T$ and $\|b_0-b_{0,J}\|_{\mu_0} \leq r\varepsilon_T$, with $P_{b_0}$-probability tending to one it holds that
$$\{ b\in V_J^{\otimes d} :\|b-b_0\|_{\mu_0} \geq \tilde{M}_T\eps_T  \} \subseteq \{ b\in V_J^{\otimes d}: h_T(b,b_0)\geq M_T \eps_T\}$$
for $\tilde{M}_T = (M_T + r)/(2c_\gamma) + r$. Since the posterior probability of the last set tends to zero in $P_{b_0}$-probability, this completes the proof of Theorem \ref{thm:contraction}.

\subsection{Proofs of Theorems \ref{super}-\ref{mapas}}

The proof of Theorem \ref{super} is based on combination of Theorem \ref{thm:contraction} -- which allows an initial localisation of the posterior distribution in a neighbourhood contracting about $b_0$ in $L^2$-norm via (\ref{tvlim}) below -- with the key Lemma \ref{maxwav}(i) (which itself follows from `quantitiative' semiparametric techniques developed in Section \ref{auxo}). Once Theorem \ref{super} is proved one can refine Lemma \ref{maxwav} (see its Part (ii)) and apply it to the $\|\cdot\|_\infty$-localised posterior, from which one can derive Theorems \ref{bvm1} and \ref{mapas}.

\subsubsection{Localisation and a key lemma}\label{sec:local}

We will repeatedly use the following basic fact that allows to `localise' the posterior distribution to sets $\mathcal D_T$ of high frequentist posterior probability: let $\mathcal D_T$ be any measurable set in the support of the prior satisfying $\Pi(\mathcal D_T|X^T) = 1-o_{P_{b_0}}(1)$ as $T\rightarrow \infty$, let $$\Pi^{\mathcal D_T}(\cdot)=\Pi(\cdot \cap \mathcal D_T)/\Pi(\mathcal D_T)$$ denote the prior conditioned to $\mathcal D_T$ and let $\Pi^{\mathcal D_T}(\cdot|X^T)$ denote the posterior distribution arising from prior $\Pi^{\mathcal D_T}$. By a standard inequality (\cite{vdvaart1998}, p.~142),
\begin{align} \label{tvlim}
\sup_{A \text{ measurable}} |\Pi(A|X^T)-\Pi^{\mathcal D_T}(A|X^T)| \leq 2\Pi (\mathcal D_T^c|X^T) \rightarrow_{T \to \infty}^{P_{b_0}} 0.
\end{align}

The relevant choices $D_T$ and $\bar D_T$ for $\mathcal D_T$, to be defined below, depend on a further choice $\Gamma_T \subset V_J^{\otimes d}$ of vector fields $\gamma$ admitting envelopes 
\begin{equation}\label{enve}
|\Gamma_T|_2 \geq \sup_{\gamma\in\Gamma_T} \|\gamma\|_{L^2}, \quad \quad ~\sigma_{\Gamma_T} \ge \sup_{\gamma \in \Gamma_T}\|\gamma\|_\mathbb H,
\end{equation}
where the RKHS norm $\|\cdot\|_\mathbb H$ arises from (\ref{eq:RKHS_inner_prod}) with $\sigma_l = 2^{-l(\alpha+d/2)}$. For any $M>0$ and $\varepsilon_T = T^{-\frac{a \wedge s}{2a+d}}(\log T)$ as in Theorem \ref{thm:contraction}, define
\begin{align}\label{D_T}
D_T  = \Big\{ b \in V_J^{\otimes d}: \|b-b_{0}\|_{L^2} \leq M_T \varepsilon_T,\, \sup_{\gamma \in \Gamma_T} |\langle b, \gamma \rangle_{\H}| \leq M \sqrt{T}\varepsilon_T \sigma_{\Gamma_T} \Big\},
\end{align}
where $M_T \to \infty$  arbitrarily slowly; and for $\bar M_T = (\log T)^{\delta-1}, \delta>5/2$, define
\begin{equation}\label{barD_T}
\bar D_T = \Big\{ b \in V_J^{\otimes d}: \|b-b_{0}\|_{\infty} \leq \bar M_T \varepsilon_T, ~\sup_{\gamma \in \Gamma_T} |\langle b, \gamma \rangle_{\H}| \leq M \sqrt{T}\varepsilon_T \sigma_{\Gamma_T} \Big\}.
\end{equation}

For $\lambda \leq J$, $1 \leq j \leq d$ and $a_\lambda > 0$ to be chosen, define the vector fields $\tilde{\Phi}_{\lambda,k,j} = (\tilde{\Phi}_{\lambda,k,j,1},\dots,\tilde{\Phi}_{\lambda,k,j,d}):\T^d \rightarrow \R^d$ with
\begin{align*}\label{eq:least_fav_dir}
\tilde{\Phi}_{\lambda,k,j,i}  =  \begin{cases}
a_\lambda P_{V_J}[\Phi_{\lambda,k}/\mu_0] & i=j,\\
0 & i\neq j.
\end{cases}
\end{align*}
Thus $\tilde{\Phi}_{\lambda,k,j}$ is the vector field which projects $a_\lambda \Phi_{\lambda,k}/\mu_0$ onto $V_J$ in the $j$-th coordinate and is uniformly zero on all other coordinates. Denote the collection of all such functions by 
\begin{equation}\label{gama_main}
\Gamma_T = \{ \tilde{\Phi}_{\lambda,k,j} : \lambda \leq J, k, 1\leq j \leq d\} \subset V_J^{\otimes d}.
\end{equation}
By Lemma \ref{lem:change_of_meas_L_inf} and for $\Gamma_T$ as above we can take the envelopes from (\ref{enve}) as
\begin{equation}\label{enve_choice}
|\Gamma_T|_2 = C(d,\mu_0) \max_{\lambda \le J} a_\lambda, \quad \sigma_{\Gamma_T} = C(d,\mu_0,\Phi) 2^{J(\alpha+d/2)} \max_{\lambda \le J} a_\lambda.
\end{equation}
The following central lemma will be proved in Section \ref{aux3} below. 

\begin{lemma}\label{maxwav}
(i) Assume the conditions of Theorem \ref{super} and let $D_T$ be the set from (\ref{D_T}) with $\Gamma_T$ as in (\ref{gama_main}), envelope $\sigma_{\Gamma_T}$ as in \eqref{enve_choice} and with 
$$a_\lambda = \begin{cases} 
      2^{\lambda d/2}2^{-Jd/2} (\log T)^{-\eta} & \text{ if } d\leq 4,\\
      2^{\lambda d/2} 2^{-J(d-2)} (\log T)^{-\eta} & \text{ if } d\geq  5,
   \end{cases}$$
for any $\eta>1$. Then $\Pi(D_T|X^T) = 1-o_{P_{b_0}}(1)$ as $T\to\infty$. Moreover if $b \sim \Pi^{D_T}(\cdot|X^T)$ then for all $\lambda \le J$, as $T \to \infty$,
\begin{equation}\label{nfull}
E^{\Pi^{D_T}}\big[ \max_{k,j} \sqrt T |\langle b_j-b_{0,j}, a_\lambda \Phi_{\lambda, k}\rangle_{L^2}||X^T\big] =O_{P_{b_0}}(\sqrt \lambda),
\end{equation}
\begin{equation} \label{full}
E^{\Pi^{D_T}}\big[ \max_{\lambda \le J,k,j} \sqrt T |\langle b_j-b_{0,j}, a_\lambda \Phi_{\lambda, k}\rangle_{L^2}||X^T\big] =O_{P_{b_0}}(\sqrt J).
\end{equation}
(ii) Assume the conditions of Theorem \ref{bvm1} and let $\bar D_T$ be the set from (\ref{barD_T}) with $\Gamma_T$ as in (\ref{gama_main}) with $a_\lambda=1$ for all $\lambda$ and envelope $\sigma_{\Gamma_T}$ as in \eqref{enve_choice}. Then $\Pi(\bar D_T|X^T) = 1-o_{P_{b_0}}(1)$ as $T \to \infty$ and if $b \sim \Pi^{\bar D_T}(\cdot|X^T)$, then for all $\lambda \le J$,
\begin{equation*}\label{nfull2}
E^{\Pi^{\bar D_T}}\big[ \max_{k,j} \sqrt T |\langle b_j-b_{0,j}, \Phi_{\lambda, k}\rangle_{L^2}||X^T\big] =O_{P_{b_0}}(\sqrt \lambda).
\end{equation*}
\end{lemma}

\subsubsection{Proof of Theorem \ref{super}}

Take the set $D_T$ from (\ref{D_T}) with $\Gamma_T$, $\sigma_{\Gamma_T}$ as in Lemma \ref{maxwav}(i). Then by that Lemma  and (\ref{tvlim}) with $\mathcal D_T = D_T$, it suffices to prove Theorem \ref{super} for $b$ drawn from the localised posterior distribution $\Pi^{D_T}(\cdot |X^T)$. Denote by $P_{V_J}$, $P_{V_J^{\otimes d}}$ the projection operators onto $V_J$, $V_J^{\otimes d}$, respectively. Setting $$\tilde{\varepsilon}_T = (\log T)^\delta T^{-\frac{s\wedge[a-(d/2-2)_+]}{2a+d}}$$ and applying Markov's inequality,
\begin{align}\label{nebd}
&   \Pi^{D_T} \big( b: \sum_{j=1}^d \|b_j-b_{0,j}\|_{\infty} \geq \tilde{\eps}_T \big| X^T \big)  \leq  \tilde \varepsilon_T^{-1} \sum_{j=1}^d E^{\Pi^{D_T}}[\|b_j-b_{0,j}\|_\infty |X^T] \notag \\
&  \leq \tilde \varepsilon_T^{-1} \sum_{j=1}^d  E^{\Pi^{D_T}}[\|b_j-P_{V_J} [b_{0,j}]\|_\infty |X^T]~+ \tilde \varepsilon_T^{-1}  \sum_{j=1}^d \|P_{V_J} [b_{0,j}]-b_{0,j}\|_\infty.
\end{align}
Since $b_0 \in C^s$, the second sum is of order $O(2^{-Js}) = O(T^{-\frac{s}{2a+d}}) = o(\tilde{\varepsilon}_T)$ by standard results in approximation theory (cf.~after (\ref{wavbas}) above). Suppose first that $d\leq 4$ and let $a_\lambda = 2^{\lambda d/2}2^{-Jd/2} (\log T)^{-\eta}$ for some $\eta > 1$. The standard inequality $\sup_x \sum_k |\Phi_{\lambda,k}(x)| \lesssim 2^{\lambda d/2}$ (Section 4.3 in \cite{ginenickl2016}) now gives
\begin{equation*}\label{eq:L_inf_main_bound}
\begin{split}
\sum_{j=1}^d \|b_j-P_{V_J}[b_{0,j}]\|_\infty & =  \sum_{j=1}^d \sup_x \Big| \sum_{\lambda\leq J} \sum_k \langle b_j-b_{0,j} ,\Phi_{\lambda,k} \rangle_{L^2} \Phi_{\lambda,k}(x) \Big| \\
& \lesssim  \sum_j \sum_{\lambda\leq J} \frac{2^{\lambda d/2}}{\sqrt{T}} \max_{k} \sqrt{T}|\langle b_j-b_{0,j},\Phi_{\lambda,k}\rangle_{L^2}| \\
& = \frac{1}{\sqrt{T}} \sum_j \sum_{\lambda\leq J} 2^{\lambda d/2} a_\lambda^{-1} \max_{k} \sqrt{T}|\langle b_j-b_{0,j},a_\lambda \Phi_{\lambda,k}\rangle_{L^2}| \\
& \lesssim  \frac{J2^{Jd/2}(\log T)^\eta}{\sqrt{T}} \max_{\lambda\leq J, k,j} \sqrt{T}|\langle b_j-b_{0,j},a_\lambda \Phi_{\lambda,k}\rangle_{L^2}|.
\end{split}
\end{equation*}
Taking posterior expectations in the last inequality, Lemma \ref{maxwav}(i) implies that the first term in (\ref{nebd}) is bounded by
\begin{align*}
& \tilde{\varepsilon}_T^{-1} \frac{J2^{Jd/2}(\log T)^\eta}{\sqrt{T}} E^{\Pi^{D_T}} \Big[  \max_{\lambda\leq J, k,j} \sqrt{T}|\langle b_j-b_{0,j},a_\lambda \Phi_{\lambda,k}\rangle_{L^2}| \big|X^T \Big]  \\
& =O_{P_{b_0}} \Big(\tilde{\varepsilon}_T^{-1} \frac{J^{3/2} 2^{Jd/2}(\log T)^\eta}{\sqrt{T}} \Big) = O_{P_{b_0}} \Big((\log T)^{3/2+\eta-\delta} \Big).
\end{align*}
Taking $\delta>3/2+\eta$ completes the proof when $d \le 4$ since $\eta>1$ was arbitrary. If $d> 4$, we set $a_\lambda = 2^{\lambda d/2} 2^{-J(d-2)} (\log T)^{-\eta}$ for $\eta>1$ and use again Lemma \ref{maxwav}(i) to obtain, as $T \to \infty$, the convergence to zero of
\begin{align*}
\tilde{\varepsilon}_T^{-1} \sum_j E^{\Pi^{D_T}} [\|b_j-P_{V_J} b_{0,j}\|_\infty |X^T]  =O_{P_{b_0}} \Big( \frac{J^{\frac{3}{2}} 2^{J(d-2)}(\log T)^\eta}{\tilde{\varepsilon}_T \sqrt{T}}\Big) = o_{P_{b_0}} (1).
\end{align*}

\subsubsection{Proof of Theorems \ref{bvm1} and \ref{mapas}}

Let $b \sim \Pi^{\bar D_T} (\cdot|X^T)$ conditionally on $X^T$, where $\bar D_T$ is the event from (\ref{barD_T}) with $\Gamma_T, \sigma_{\Gamma_T}$ chosen as in Lemma \ref{maxwav}(ii). Then by that lemma and (\ref{tvlim}) with $\mathcal D_T = \bar D_T$, it suffices to prove Theorem \ref{bvm1} for  $\Pi^{\bar D_T}(\cdot|X^T)$ in place of $\Pi(\cdot|X^T)$. 

Denote the centred ball of radius $r$ in $B^\rho_{1\infty}=B^{\rho, \otimes d}_{1\infty}$ by $\mathcal B^\rho(r)$ and let $\eta= (\eta_1, \dots, \eta_d) \in \mathcal B^\rho(1)$. For projections $$P_{V_J^{\otimes d}}[\eta/\mu_{0}]=(P_{V_J}[\eta_j/\mu_{0}]: j=1, \dots, d),$$ define the centring process $( \hat G_J(\eta) \equiv \langle \hat G_J, \eta \rangle_{L^2}  :  \eta \in \mathcal B^\rho(1))$ by
\begin{equation*}
 \langle \hat G_J, \eta \rangle_{L^2} = \sum_{j=1}^d \langle \hat G_{J,j}, \eta_j \rangle_{L^2} = \langle b_0, \eta \rangle_{L^2} + \frac{1}{T} \int_0^T P_{V_J^{\otimes d}}[\eta/\mu_{0}] (X_t).dW_t,
\end{equation*}
where the notation $\langle \hat G_J, \cdot \rangle_{L^2}$ is justified by linearity of the stochastic integral. Next define stochastic processes $$(Z_1(\eta) = \sqrt T(\langle b,\eta \rangle_{L^2} - \hat G_J(\eta)): \eta \in \mathcal B^\rho(1)), \quad (Z_2(\eta): \eta \in \mathcal B^\rho(1)),$$ where $Z_2$ has (cylindrical) law $\mathcal N_{b_0}$, and denote the (conditional) law of $Z_1$ by $\bar \Pi^{\bar D_T}=\bar \Pi^{\bar D_T}(\cdot|X^T)$. Both processes prescribe linear actions on $\mathcal B^\rho(1)$  -- this is clear for $Z_1$ and follows also for $Z_2$ as explained before (\ref{kC}) below. The estimates that follow imply moreover that the $Z_i$ define proper random variables in $(B^\rho_{1\infty})^*$. For $\kappa \in \mathbb N$ to be chosen, define probability measures $\bar \Pi^{\bar D_T}_\kappa, \mathcal N_{b_0,\kappa}$ as the laws of the stochastic processes
\begin{equation}\label{projection}
P_{(\kappa)}(Z_i)\equiv(Z_i(P_{V_\kappa^{\otimes d}}[\eta]): \eta \in \mathcal B^\rho(1)), \quad i=1,2,
\end{equation}
which, as projections, are defined on the same probability space as the $Z_i$'s. Using the triangle inequality for the metric $\beta=\beta_{(B^\rho_{1\infty})^*}$ we obtain
\begin{align}  \label{blconv}
\beta(\bar \Pi^{\bar D_T}, \mathcal N_{b_0}) & \le \beta (\bar \Pi^{\bar D_T}_\kappa, \mathcal N_{b_0,\kappa}) + \beta(\bar \Pi^{\bar D_T}, \bar \Pi^{\bar D_T}_\kappa) + \beta(\mathcal N_{b_0}, \mathcal N_{b_0, \kappa}) \notag \\
&\notag =  \beta_{V^{\otimes d}_\kappa} (\bar \Pi^{\bar D_T}_\kappa, \mathcal N_{b_0,\kappa})  + \sum_{i=1}^2\sup_{\|F\|_{Lip}\le 1}|E[F(Z_i) -F(P_{(\kappa)} (Z_i))]|  \\
&\le \beta_{V^{\otimes d}_\kappa} (\bar \Pi^{\bar D_T}_\kappa, \mathcal N_{b_0,\kappa}) +  \sum_{i=1}^2  E\|Z_i - P_{(\kappa)}(Z_i)\|_{(B^\rho_{1\infty})^*} = A+B+C.
\end{align}
For term $B$, we use Parseval's identity and the fact that (cf.~after (\ref{wavbas})) 
\begin{equation}\label{wavo}
\|\eta\|_{B^\rho_{1\infty}}\le 1 \Rightarrow \sum_j \sum_{r}|\langle \eta_j, \Phi_{l,r}\rangle_{L^2}| \lesssim 2^{-l(\rho-d/2)} \quad \forall l
\end{equation} to obtain, with $E=E^{\Pi^{\bar D_T}}[\cdot|X^T]$,
\begin{align} \label{kB}
&E\|Z_1-P_{(\kappa)}(Z_1)\|_{(B^\rho_{1\infty})^*}  = E \sup_{\|\eta\|_{B^\rho_{1\infty}} \le 1} \sqrt T \left | \langle b- \hat G_J, \eta - P_{V^{\otimes d}_\kappa}[\eta] \rangle_{L^2}\right| \notag \\
& \lesssim \sum_{\kappa<\lambda} 2^{-\lambda(\rho-d/2)} E\max_{k,j}\sqrt T \left | \langle b_j- \hat G_{J,j}, \Phi_{\lambda,k}\rangle_{L^2} \right| \notag \\
& \lesssim \sum_{\kappa<\lambda} 2^{-\lambda(\rho-d/2)} E\max_{k,j}\sqrt T \left | \langle b_j- b_{0,j}, \Phi_{\lambda,k}\rangle_{L^2} \right| \\
& \quad \quad +  \sum_{\kappa<\lambda} 2^{-\lambda(\rho-d/2)} \max_{k,j}\sqrt T \left | \langle \hat G_{J,j} - b_{0,j}, \Phi_{\lambda,k}\rangle_{L^2} \right|. \notag
\end{align}
By Lemma \ref{maxwav}(ii) and the usual decay bound for wavelet coefficients of $b_0 \in C^s$, the first sum is bounded in $P_{b_0}$-probability by $$\sum_{\kappa<\lambda \le J,j} 2^{-\lambda(\rho-d/2)} \sqrt \lambda + \sqrt T\sum_{\lambda >J} 2^{-\lambda(\rho+s)} = o(1)$$ as $T\to \infty$ and $\kappa \to \infty$, since $\rho >d/2$ and $s\geq a$. To deal with the second sum, note that by definition $$\sqrt T \langle \hat G_{J,j}-b_{0,j}, \Phi_{\lambda,k}\rangle_{L^2} = \frac{1}{\sqrt T} \int_0^T P_{V_J}[\Phi_{\lambda,k}/\mu_0](X_t)dW_t^j.$$ Arguing as after (\ref{eq:event_A}) below, Bernstein's inequality (\ref{eq:Bernstein}) implies that these variables are subgaussian under $P_{b_0}$ with variance proxy bounded by $$\|P_{V_J}[\Phi_{\lambda,k}/\mu_0]\|_{\mu_0}^2 + \frac{1}{T}\int_0^T |P_{V_J}[\Phi_{\lambda,k}/\mu_0](X_t)|^2 dt -   \|P_{V_J}[\Phi_{\lambda,k}/\mu_0]\|_{\mu_0}^2.$$ The first quantity is bounded by $\|\Phi_{\lambda,k}\|_{L^2} \|1/\mu_0\|_\infty \lesssim 1$ whereas Proposition \ref{prop:laplace_trans}(ii) implies that the second quantity is
$O_{P_{b_0}}(\tilde R_T)  =O_{P_{b_0}}(1)$ uniformly over $\lambda,k$ for our choice of $J,s$.  Thus by the usual subgaussian maximal inequality (Lemma 2.3.4 in \cite{ginenickl2016}), the last term in (\ref{kB}) is $O_{P_{b_0}}(\sum_{\lambda>J} 2^{-\lambda(\rho-d/2)} \sqrt \lambda)= o_{P_{b_0}}(1)$ for $\rho>d/2$, so that the last sum in (\ref{kB}) is $o_{P_{b_0}}(1)$ as $\kappa, T \to \infty$.

For term $C$, we first note that $\mathcal N_{b_0}$ defines a tight Gaussian probability measure in the space of bounded functions on $\mathcal B^\rho(1)$ (using Theorem 2.3.7, Proposition 2.1.5 and (4.184) in \cite{ginenickl2016}), and arguing as in Theorem 3.7.28 in \cite{ginenickl2016} one shows further that $\mathcal N_{b_0}$ extends to a Gaussian probability measure on $(B^\rho_{1\infty})^*$. In particular, a version of $Z_2$ exists that acts linearly on $\mathcal B^\rho(1)$. Define $\Phi_{\lambda,k,j} = (0,\dots,0,\Phi_{\lambda,k},0,\dots,0):\T^d \rightarrow \R^d$, where the non-zero coordinate occurs in the $j^{th}$ entry. Then, recalling the definition \eqref{projection} of $P_{(\kappa)}(Z_2)$ and using again (\ref{wavo}) and the standard subgaussian maximal inequality, now for the variables $(Z_2(\Phi_{\lambda,k,j}) \sim N(0, \|\Phi_{\lambda,k}\|_{1/\mu_0}^2))$,
\begin{align}\label{kC}
&E\|Z_2-P_{(\kappa)}(Z_2)\|_{(B^\rho_{1\infty})^*}  = E \sup_{\|\eta\|_{B^\rho_{1\infty}} \le 1}  \left | Z_2 (\eta - P_{V^{\otimes d}_\kappa}[\eta])\right| \notag \\
& \lesssim \sum_{\kappa<\lambda} 2^{-\lambda(\rho-d/2)} E\max_{k,j} \left | Z_2(\Phi_{\lambda,k, j}) \right|  \lesssim \sum_{\kappa<\lambda} 2^{-\lambda(\rho-d/2)} \sqrt \lambda =_{\kappa \to \infty} o(1).
\end{align}
For term $A$ we show convergence of the finite-dimensional distributions by the semi-parametric techniques from Section \ref{auxo}: consider the basis $(\Phi_{\lambda,k,j}: k, j;\lambda \le \kappa)$ of  $V^{\otimes d}_\kappa$ for $\kappa$ fixed. We apply Proposition  \ref{prop:laplace_trans}(iii) with $\gamma=P_{V_J^{\otimes d}}[\Phi_{\lambda,k,j}/\mu_{0}]$, then Lemma \ref{lem:change_of_meas}(ii) and the third part of Lemma \ref{lem:semipara_bias} to obtain
\begin{align*} \label{lapsg}
& E^{\Pi^{\bar D_T}}\left[e^{u \sqrt T (\langle b-b_0,\Phi_{\lambda,k,j}/\mu_{0} \rangle_{\mu_{0}}) - (u/\sqrt T) \int_0^T P_{V_J^{\otimes d}}[\Phi_{\lambda,k,j}/\mu_0](X_t).dW_t} \big|X^T \right] \\
& = C_T \exp\left\{\frac{u^2}{2}\int_{\T^d}\|P_{V_J^{\otimes d}}[\Phi_{\lambda,k,j}/\mu_{0}]\|^2d\mu_{0}\right\},
\end{align*}
where we can take the envelopes in Lemma \ref{lem:change_of_meas}(ii) to satisfy $|\Gamma_T|_2\lesssim 1, \varepsilon_T \sigma_{\Gamma_T}=o(1)$ as in the proof of Lemma \ref{maxwav}, and where  $C_T=1+o_{P_{b_0}}(1)$ as $T \to \infty$ for fixed $u \in \mathbb R$. We also have $\|P_{V_J^{\otimes d}}[\Phi_{\lambda,k,j}/\mu_{0}]\|_{\mu_0} \to \|\Phi_{\lambda,k}/\mu_{0}\|_{\mu_0}=\|\Phi_{\lambda,k}\|_{1/\mu_0}$ as $J \to \infty$ since $P_{V_J}$ are $L^2$-projections. The same is true if $\Phi_{\lambda,k,j}$ is replaced by arbitrary finite linear combinations $\sum_j \sum_{\lambda \le \kappa, k} a_{\lambda,k,j} \Phi_{\lambda,k,j}$, $\kappa$ fixed, and thus by Proposition 29 in \cite{N17} we deduce joint weak convergence of the finite-dimensional distributions. In particular, for every fixed $\kappa \in \mathbb N$,
\begin{equation} \label{fidiho}
\beta_{V_\kappa^{\otimes d}}(\bar \Pi_\kappa^{\bar D_T}, \mathcal N_{b_0, \kappa}) \to^{P_{b_0}} 0 \quad \text{as }T \to \infty.
\end{equation}
Combining the above bounds, given $\epsilon'>0$ we can choose $\kappa=\kappa(\epsilon')$ large enough so that  by virtue of the bounds following (\ref{kB}) and (\ref{kC}), the terms $B,C$ in (\ref{blconv}) are each less than $\epsilon'/3$ (for $B$ on an event of $P_{b_0}$-probability as close to one as desired). Then applying (\ref{fidiho}) for this choice of $\kappa$, we can also make the term $A$ less than $\epsilon'/3$ for $T$ large enough and with $P_{b_0}$-probability as close to one as desired, completing the proof of Theorem \ref{bvm1} with $\bar \Pi_T$ replacing $\tilde \Pi_T$, that is, with centring equal to $\hat G_J$.  

That $\hat G_J$ can be replaced by the posterior mean in Theorem \ref{bvm1} is the last step: since the laws $\bar \Pi_T$ form a sequence of (conditionally on $X^T$) Gaussian distributions on $(B^\rho_{1\infty})^*$ that converge weakly (in probability), we also have convergence of moments of that sequence (in probability) in $(B^\rho_{1\infty})^*$, using Exercise 2.1.4 in \cite{ginenickl2016} and arguing as in the proof of Theorem 2.7 in \cite{MNP17}. Since $\mathcal N_{b_0}$ has Bochner-mean zero we deduce that 
\begin{equation} \label{asylin}
\sqrt T(E^{\Pi_T}[b|X^T] - \hat G_J)= o_{P_{b_0}}(1)\text { in } (B^\rho_{1\infty})^*.
\end{equation}
This concludes the proof of Theorem \ref{bvm1}. Theorem \ref{mapas} now follows from (\ref{asylin}) and asymptotic normality of the $\sqrt T(\hat G_J (\eta) - \langle b_0, \eta \rangle)$ variables in the space $(B^\rho_{1\infty})^*$, proved as follows: if we denote by $\nu_T$ the law of the latter variables, then arguing just as in (\ref{blconv}) we have
\begin{equation} \label{betaagain}
\beta_{(B^\rho_{1\infty})^*}(\nu_T, \mathcal N_{b_0}) \le \beta_{V^{\otimes d}_\kappa} (\nu_{T,\kappa}, \mathcal N_{b_0,\kappa}) +  \sum_{i=1}^2  E\|\tilde Z_i - P_{(\kappa)}(\tilde Z_i)\|_{(B^\rho_{1\infty})^*}
\end{equation}
where $\tilde Z_2=^\mathcal LZ_2$ from above, $\tilde Z_1$ has law $\nu_T$ and $P_{(\kappa)}(\tilde Z_i)$ refers to the projected processes as in \eqref{projection}. The first term on the right hand side converges to zero, for every fixed $\kappa$, by applying the martingale central limit theorem as in (\ref{mclt}) to $(1/\sqrt T)\int_0^T (\Phi_{\lambda,r, j}/\mu_0)(X_t). dW_t$, $\lambda \le \kappa$ fixed, and using (\ref{eq:Bernstein}) to show that the term $$\frac{1}{\sqrt T}\int_0^T [P_{V_J^{\otimes d}}[\Phi_{\lambda,r, j}/\mu_0] -\Phi_{l,r, j}/\mu_0](X_t).dW_t =o_{P_{b_0}}(1)$$ in view of $\|P_{V_J^{\otimes d}}[\Phi_{\lambda,r, j}/\mu_0] -\Phi_{\lambda,r, j}/\mu_0\|_\infty \to 0$ as $J \to \infty$ for fixed $\lambda \le \kappa$. The third term in (\ref{betaagain}) was bounded as $o(1)$ for $\kappa \to \infty$ in (\ref{kC}), and the second term also converges to zero as $\kappa \to \infty$ by the arguments below (\ref{kB}). Thus choosing $\kappa$ large enough but fixed, and letting $T \to \infty$, Theorem \ref{mapas} follows since $\beta_{(B^\rho_{1\infty})^*}$ metrises weak convergence.

\section{Bayesian semi-parametric techniques for diffusions}\label{auxo}

We show here how techniques developed in \cite{CN14, castillo2014, CR15} for Gaussian white noise and i.i.d.~density models extend to the multi-dimensional diffusion case. 

\subsection{Asymptotic expansion of the posterior Laplace transform}\label{aux1}

We start with the following basic `LAN expansion' for $\ell_T$ from (\ref{eq:likelihood}).

\begin{lemma}\label{lem:LAN_expansion}
Suppose $b_0 \in C^{(d/2+\kappa)\vee 1}(\T^d), h \in H^{d/2+\kappa}(\T^d), \kappa>0$. Then
\begin{align*}
\ell_T(b_0 + h/\sqrt{T}) - \ell_T(b_0) = W_T(h) - \frac{1}{2T}\int_0^T \|h(X_t)\|^2 dt,
\end{align*}
where, as $T \to \infty$ and under $P_{b_0}$, $$W_T(h) \equiv \frac{1}{\sqrt T}\int_0^T h(X_t).dW_t \rightarrow^d N(0,\|h\|_{\mu_0}^2),~\frac{1}{T}\int_0^T \|h(X_t)\|^2 dt \rightarrow^{P}  \|h\|_{\mu_0}^2.$$
\end{lemma}

\begin{proof}
Using \eqref{eq:SDE} with $b=b_0$ and \eqref{eq:likelihood},
\begin{align*}
\ell_T(b_0 + h/\sqrt{T}) - \ell_T(b_0) & = \frac{1}{\sqrt{T}} \int_0^T h(X_t).dX_t - \frac{1}{\sqrt{T}} \int_0^T b_0(X_t).h(X_t)dt \\
& \quad - \frac{1}{2T} \int_0^T\| h(X_t)\|^2 dt \\
& = \frac{1}{\sqrt{T}}\int_0^T h(X_t).dW_t - \frac{1}{2T} \int_0^T\| h(X_t)\|^2 dt .
\end{align*}
Since $x\mapsto \|x\|^2$ is a smooth map, the function $f_h(x) = \|h(x)\|^2- \|h\|_{\mu_0}^2 \in L_{\mu_0}^2(\T^d) \cap H^{d/2+\kappa}(\T^d)\subset C(\T^d)$. In particular, $LL^{-1}[f_h]=f_h$ where $L^{-1}=L^{-1}_{b_0}$ is the inverse of the generator $L$  constructed in Lemma \ref{regest} below. Moreover, by that Lemma and the Sobolev embedding theorem, $L^{-1}[f_h]  \in H^{d/2+\kappa+2} \subset C^2$. By It\^o's formula (Theorem 39.3 in \cite{B11}),
\begin{align*}
\int_0^T f_h(X_t)dt &= \int_0^T LL^{-1}[f_h](X_t)dt \\
&= L^{-1}[f_h](X_T) - L^{-1}[f_h](X_0) - \int_0^T \nabla L^{-1}[f_h](X_t).dW_t.
\end{align*}
Since $L^{-1}[f_h] \in C^2$, the first term on the right-hand side is $O(1)$, while the second term satisfies
\begin{align*}
E_{b_0} \Big[ \int_0^T \nabla L^{-1}[f_h](X_t).dW_t \Big]^2 = E_{b_0} \int_0^T \|\nabla L^{-1}[f_h](X_t)\|^2 dt \lesssim T\|L^{-1}[f_h]\|_{C^1}^2
\end{align*}
so that $T^{-1}\int_0^T f_h(X_t)dt \rightarrow 0$ in $L^2(P_{b_0})$, and hence also in $P_{b_0}$-probability. Next, set $M_T^h  = \int_0^T h(X_t).dW_t$, so that $(M_T^h)_{T\geq 0}$ is a continuous local $L^2$-martingale with quadratic variation $[M^h]_T = \int_0^T \|  h(X_t)\|^2 dt$. By what precedes $T^{-1} [M^h]_T - \|h\|_{\mu_0}^2 = T^{-1}\int_0^T f_h(X_t)dt \rightarrow 0$ in $L^2(P_{b_0})$ as $T\rightarrow \infty$. Applying the martingale central limit theorem (p.338f. in \cite{KE86}),
\begin{align} \label{mclt}
T^{-1/2} M_T^h \rightarrow^d N ( 0, \|h\|_{\mu_0}^2 )
\end{align}
as $T\rightarrow \infty$, completing the proof.
\end{proof}

A key result for our proofs is the following expansion of the Laplace transform of the posterior distribution $\Pi^{\mathcal D_T}(\cdot|X^T)$ arising from a `localised' prior $\Pi^{\mathcal D_T}$ for the choices of $\mathcal D_T$ from Lemma \ref{maxwav}.

\begin{proposition}\label{prop:laplace_trans}
Suppose $b_0 \in C^s(\T^d) \cap H^s(\T^d)$, $s>\max(d/2,1)$, and consider the Gaussian prior $\Pi$ from \eqref{eq:Gauss_prior_wav} with $2^J \approx T^\frac{1}{2a+d}$ and $\sigma_l = 2^{-l(\alpha+d/2)}$ for $a>\max(d-1,1/2)$ and $0 \leq \alpha \leq a$. Let $\Gamma_T \subset V_J^{\otimes d}$ be a set of functions admitting envelopes as in (\ref{enve}) and let $D_T\subset V_J^{\otimes d}$ denote the set \eqref{D_T} for this choice of $\Gamma_T$ and arbitrary $M>0$. For $u\in \R$, $b\in V_J^{\otimes d}$ and fixed $\gamma \in \Gamma_T$, define the perturbations 
\begin{equation} \label{pertu}
b_u =b_u(T,\gamma)=b-\tfrac{u}{\sqrt{T}}\gamma \in V_J^{\otimes d}.
\end{equation}
For any measurable function $G:L^2(\T^d)\rightarrow\R$, write
\begin{align} \label{lambo}
E^{\Pi^{D_T}}[e^{u\sqrt{T}G(b)} |X^{T}] =e^{\Lambda_T(u)} \frac{\int_{D_T} e^{S_T(b) + \ell_T(b_u)} d\Pi(b)}{\int_{D_T} e^{\ell_T(b)} d\Pi(b)},
\end{align}
for some $\Lambda_T$ to be determined and where 
\begin{align*}
S_T(b) = u\sqrt{T} \left( G(b)-\langle b-b_0,\gamma\rangle_{\mu_0} \right).
\end{align*}

\noindent (i) If for some $\kappa>0$ (or $\kappa=0$ if $d=1$),
\begin{align*}
R_T := 2^{J[d+(d/2+\kappa-1)_+]} M_T \varepsilon_T |\Gamma_T|_2  \left( 1+ \sqrt{\log (1/(M_T\varepsilon_T))} + \sqrt{\log (1/|\Gamma_T|_2)}\right) 
\end{align*}
satisfies $R_T \to 0 $ as $T\rightarrow \infty$, then we can take
\begin{align*}
\Lambda_T(u) = \frac{u}{\sqrt{T}} \int_0^T \gamma(X_t).dW_t +  \frac{u^2}{2T}\int_0^T \|\gamma(X_t)\|^2 dt + ur_T,~u \in \mathbb R,
\end{align*}
in (\ref{lambo}) with $r_T= O_{P_{b_0}}(R_T ) =o_{P_{b_0}}(1)$ uniformly over $\gamma\in \Gamma_T$.

\noindent (ii) Furthermore,
\begin{equation*}\label{epbd}
E_{b_0} \sup_{\gamma\in\Gamma_T} \left| \frac{1}{T}\int_0^T \|\gamma(X_t)\|^2 dt -\|\gamma\|_{\mu_0}^2 \right|  \lesssim  \tilde{R}_T,
\end{equation*}
where
$$ \tilde{R}_T := T^{-1/2} 2^{J[d+(d/2+\kappa-1)_+]} |\Gamma_T|_2^2 \left( 1+\sqrt{\log (1/|\Gamma_T|_2)} \right)$$
for any $\kappa>0$ ($\kappa=0$ if $d=1$). In particular, if both $R_T,\tilde{R}_T \rightarrow 0$, then we can take
\begin{align*}
\Lambda_T(u) = \frac{u}{\sqrt{T}} \int_0^T \gamma(X_t).dW_t +  \frac{u^2}{2} \|\gamma\|_{\mu_0}^2 + u r_T + u^2 \tilde{r}_T, ~u \in \mathbb R,
\end{align*}
in (\ref{lambo}) with $r_T= O_{P_{b_0}}(R_T ) =o_{P_{b_0}}(1)$ and $\tilde{r}_T= O_{P_{b_0}}(\tilde{R}_T) =o_{P_{b_0}}(1)$ uniformly over $\gamma\in \Gamma_T$.


\noindent (iii) Parts (i) and (ii) remain true if $D_T$ is replaced by $\bar D_T$ from (\ref{barD_T}) and if $M_T$ is replaced by $\bar M_T$ in the definition of $R_T$.
\end{proposition}

\begin{proof}
$(i)$ For $\gamma = (\gamma_1,\dots,\gamma_d) \in \Gamma_T$,
\begin{align*}
E^{\Pi^{D_T}}[e^{u\sqrt{T}G(b)} |X^{T}] & =  E^{\Pi^{D_T}}[e^{S_T(b) + u\sqrt{T} \langle b-b_0,\gamma\rangle_{\mu_0} } |X^{T}]  \\
& = Z_T^{-1} \int_{D_T} e^{S_T(b) + u\sqrt{T}\langle b-b_{0},\gamma \rangle_{\mu_0} + \ell_T(b_u) + \ell_T(b)-\ell_T(b_u)}  d\Pi(b),
\end{align*}
with $Z_T = \int_{D_T} e^{\ell_T(b)} d\Pi(b)$. Define the empirical process $$\G_T [h] = \sqrt{T} \left( \frac{1}{T} \int_0^T h(X_t)dt - \int_{\T^d} h d\mu_0\right),~~h\in L^2(\T^d).$$ Using the LAN expansion from Lemma \ref{lem:LAN_expansion},
\begin{align*}
\ell_T(b)-\ell_T(b_u) &= \frac{u}{\sqrt{T}} \int_0^T \gamma(X_t).dW_t  - \frac{u}{\sqrt{T}} \int_0^T [b(X_t)-b_{0}(X_t)].\gamma(X_t)dt \\
& \quad + \frac{u^2}{2T} \int_0^T \|\gamma(X_t)\|^2 dt\\
& =  \frac{u}{\sqrt{T}} \int_0^T \gamma(X_t).dW_t  -u\G_T[(b-b_{0}).\gamma] - u\sqrt{T} \langle b-b_{0}, \gamma \rangle_{\mu_0} \\
&\quad  + \frac{u^2}{2T} \int_0^T \|\gamma(X_t)\|^2 dt.
\end{align*}
The Laplace transform from the first equation of this proof therefore equals
\begin{align*}
e^{\frac{u}{\sqrt{T}} \int_0^T \gamma(X_t).dW_t  + \frac{u^2}{2T} \int_0^T \|\gamma(X_t)\|^2 dt} Z_T^{-1} \int_{D_T} e^{-u\G_T[(b-b_{0}).\gamma]} e^{S_T(b) + \ell_T(b_u)} d\Pi(b).
\end{align*}
We use Lemma \ref{lem:maxim} to control the empirical process term uniformly over $b\in D_T$, $\gamma \in \Gamma_T$. Set
\begin{align*}
\mF_{T} = \left\{f_{b,\gamma}:= (b-b_{0}).\gamma - \int_{\T^d} (b-b_{0}).\gamma d\mu_0:  b\in D_T , \gamma\in \Gamma_T \right\},
\end{align*}
which is a subset of $L_{\mu_0}^2(\T^d) \cap H^{d/2+\kappa}(\T^d)$ for $0<\kappa<s-d/2$ since $\Gamma_T \subset V_J^{\otimes d} \subset H^p$ for any $p\le S$. Suppose $d\geq 2$. Lemma \ref{lem:metric} with $p=d/2+\kappa-1$ gives that for any $0<\kappa < s-d/2+1$, $b,\bar{b}\in D_T$ and $\gamma,\bar{\gamma}\in \Gamma_T$,
\begin{align*}
&d_L(f_{b,\gamma},f_{\bar{b},\bar{\gamma}})  \lesssim \|f_{b,\gamma}-f_{\bar{b},\bar{\gamma}} \|_{H^{d/2+\kappa-1}}\\
& \leq  \sum_{j=1}^d \Big\|(b_j-\bar{b}_j)\gamma_j  +  (\bar{b}_j-b_{0,j})(\gamma_j-\bar{\gamma}_j) \\
&~~~~~~~~~~~~~~~ -  \langle b_j-\bar{b}_j,\gamma_j\rangle_{\mu_0}  -  \langle \bar{b}_j-b_{0,j},\gamma_j -\bar{\gamma}_j\rangle_{\mu_0}\Big\|_{H^{d/2+\kappa-1}} \\
& \lesssim \sum_{j=1}^d 2^{J(d+\kappa-1)} \|b_j-\bar{b}_j\|_{L^2} \|\gamma_j\|_{L^2}   \\
& \quad +\sum_{j=1}^d \|(\bar{b}_j-P_{V_J}b_{0,j})(\gamma_j-\bar{\gamma}_j)  - \langle \bar{b}_j-P_{V_J}b_{0,j},\gamma_j -\bar{\gamma}_j\rangle_{\mu_0}\|_{H^{d/2+\kappa-1}} \\
& \quad + \sum_{j=1}^d \|(P_{V_J}b_{0,j}-b_{0,j})(\gamma_j-\bar{\gamma}_j)  - \langle P_{V_J}b_{0,j}-b_{0,j},\gamma_j -\bar{\gamma}_j\rangle_{\mu_0}\|_{H^{d/2+\kappa-1}}.
\end{align*}
The first sum above is bounded by $C2^{J(d+\kappa-1)} |\Gamma_T|_2  \|b-\bar{b}\|_{L^2}$, while by Lemma \ref{lem:metric} the second sum is bounded by $C \sum_{j=1}^d 2^{J(d+\kappa-1)} M_T\varepsilon_T \|\gamma-\bar{\gamma}\|_{L^2}$. Using Lemma \ref{lem:runst}, that $b_0 \in C^s\cap H^s$ for $s>d/2$ and \eqref{eq:wav_bd_1}-\eqref{eq:wav_bd_2}, the third sum is bounded by
\begin{align*}
& \sum_{j=1}^d \|P_{V_J}b_{0,j}-b_{0,j}\|_{L^\infty} \|\gamma_j - \bar{\gamma}_j\|_{H^{d/2+\kappa-1}} +  \|P_{V_J}b_{0,j}-b_{0,j}\|_{H^{d/2+\kappa-1}}  \|\gamma_j - \bar{\gamma}_j\|_{L^\infty} \\
& \lesssim \sum_{j=1}^d  \left( 2^{-Js} 2^{J(d/2+\kappa-1)} \|\gamma_j - \bar{\gamma}_j\|_{L^2} + 2^{J(d+\kappa-1-s)}  \|\gamma_j - \bar{\gamma}_j\|_{L^2} \right)\\
& \lesssim 2^{J(d+\kappa-1)} M_T\varepsilon_T \|\gamma-\bar{\gamma}\|_{L^2},
\end{align*}
using again that $\Gamma_T \subset V_J^{\otimes d}$. Summarizing,
\begin{align*}
d_L(f_{b,\gamma},f_{\bar{b},\bar{\gamma}})  \lesssim 2^{J(d+\kappa-1)} ( |\Gamma_T|_2  \|b-\bar{b}\|_{L^2} +M_T\varepsilon_T \|\gamma-\bar{\gamma}\|_{L^2}).
\end{align*}
In particular, $\mF_{T}$ has $d_L$-diameter $D_{\mF_{T}} \lesssim 2^{J(d+\kappa-1)}M_T\varepsilon_T |\Gamma_T|_2 = o(R_T) = o(1)$. Since $D_T,\Gamma_T \subset (V_J^{\otimes d}, \|\cdot\|_{L^2})$ where $v_J=dim(V_J)=O(2^{Jd})$, applying Proposition 4.3.34 of \cite{ginenickl2016} yields
\begin{align*}
&N(\mF_{T,},d_L,\tau) \\
& \leq  N(D_T,c2^{J(d+\kappa-1)}|\Gamma_T|_2 \|\cdot\|_{L^2},\tau/2) N(\Gamma_T,c2^{J(d+\kappa-1)}M_T\varepsilon_T \|\cdot\|_{L^2},\tau/2) \\
& \leq ( C 2^{J(d+\kappa-1)}|\Gamma_T|_2/\tau)^{dv_J} (C2^{J(d+\kappa-1)}M_T\varepsilon_T /\tau )^{dv_J}
\end{align*}
for some $c,C>0$. Recall the inequality
\begin{align*}
\int_0^a \sqrt{\log (A/x)}dx \leq \frac{2\log A}{2\log A-1} a\sqrt{\log (A/a)} \leq 4a\sqrt{\log (A/a)}
\end{align*}
for any $A\geq 2$ and $0<a\leq 1$ (p. 190 of \cite{ginenickl2016}). Using the last two displays and that $D_{\mF_T} \rightarrow 0$, $\int_0^{D_{\mF_T}} \sqrt{\log 2N(\mF_{T},d_L,\tau)}d\tau$ is bounded by a multiple of
\begin{align*}
&v^{1/2}_J \int_0^{D_{\mF_T}} \Big[\sqrt{\frac{\log ([C2^{J(d+\kappa-1)} |\Gamma_T|_2]\vee 2 }{\tau}} + \sqrt{\frac{\log ([C2^{J(d+\kappa-1)} M_T\varepsilon_T]\vee 2}{\tau}} \Big] d\tau\\
& \lesssim 2^{Jd/2} D_{\mF_T}  \sqrt{\log ([C2^{J(d+\kappa-1)} |\Gamma_T|_2]\vee 2 /D_{\mF_T})}  \\
& ~~~~+ 2^{Jd/2} D_{\mF_T}  \sqrt{\log ([C2^{J(d+\kappa-1)} M_T\varepsilon_T]\vee 2 /D_{\mF_T})}.
\end{align*}
Taking $D_{\mF_{T}} \approx 2^{J(d+\kappa-1)}M_T\varepsilon_T |\Gamma_T|_2$ for $\kappa>0$ arbitrarily small, one can therefore bound the quantity $J(\mF_T,d_L,D_{\mF_T})$ in Lemma \ref{lem:maxim} by
$$ 2^{J(3d/2+\kappa-1)}M_T\varepsilon_T  |\Gamma_T|_2 (1+ \sqrt{\log (1/(M_T\varepsilon_T))} + \sqrt{\log (1/|\Gamma_T|_2)}) = R_T.$$
Using the Sobolev embedding theorem, Lemma \ref{regest}, Lemma \ref{lem:metric} and similar computations to the above, we see that
$$\sup_{f_{b,\gamma} \in \mF_T} \|L^{-1}[f_{b,\gamma}]\|_\infty \lesssim \sup_{f_{b,\gamma} \in \mF_T} \|f_{b,\gamma}\|_{H^{(d/2+\kappa-2)_+}} \lesssim 2^{J[d/2 + (d/2+\kappa-2)_+]}M_T\varepsilon_T |\Gamma_T|_2$$ is also $o(R_T)$. Substituting these bounds into Lemma \ref{lem:maxim} yields $$E_{b_0} \sup_{b\in D_T,\gamma\in\Gamma_T} |\G_T[(b-b_{0}).\gamma]| \lesssim R_T \rightarrow 0,$$ proving the first statement. The case $d=1$ is proved similarly, using instead the simpler bound $$d_L(f_{b,\gamma},f_{\bar{b},\bar{\gamma}}) \lesssim 2^{J/2}  |\Gamma_T|_2  \|b-\bar{b}\|_{L^2} +  2^{J/2} M_T\varepsilon_T \|\gamma-\bar{\gamma}\|_{L^2}.$$

$(ii)$ Since $x\mapsto \|x\|^2$ is a smooth map, the function $g_\gamma(x) = \|\gamma(x)\|^2 - \int_{\T^d} \|\gamma(y)\|^2 d\mu_0(y) \in L_{\mu_0}^2(\T^d)\cap H^{d/2+\kappa}$ for $\kappa >0$. Since $\gamma \in V_J^{\otimes d}$, Lemma \ref{lem:metric} with $p=(d/2+\kappa-1)_+$ gives that for any $\kappa>0$ small enough and $\gamma,\bar{\gamma}\in \Gamma_T$,
\begin{align*}
 d_L(g_\gamma,g_{\bar{\gamma}}) & \lesssim \|g_\gamma-g_{\bar{\gamma}}\|_{H^{(d/2+\kappa-1)_+}} \\
 & \lesssim \sum_{j=1}^d \|\gamma_j^2 -\bar{\gamma}_j^2- \int_{\T^d} (\gamma_j^2 - \bar{\gamma}_j^2) d\mu_0 \|_{H^{(d/2+\kappa-1)_+}} \\
& \lesssim \sum_{j=1}^d 2^{J[d/2 + (d/2+\kappa-1)_+]}  \|\gamma_j\ - \bar{\gamma}_j\|_{L^2} \|\gamma_j\ + \bar{\gamma}_j\|_{L^2} \\
& \lesssim  2^{J[d/2+(d/2+\kappa-1)_+]} |\Gamma_T|_2   \|\gamma - \bar{\gamma}\|_{L^2}.
\end{align*}
In particular, $\mathcal{G}_T=\{ g_\gamma: \gamma\in \Gamma_T\} \cup \{ 0\}$ has $d_L$-diameter $$D_{\mathcal{G}_T} \lesssim 2^{J[d/2+(d/2+\kappa-1)_+]} |\Gamma_T|_2^2.$$ Using the same arguments as above, one deduces
\begin{align*}
N(\mathcal{G}_T,d_L,\tau) &\leq N(\Gamma_T,2^{J[d/2+(d/2+\kappa-1)_+]}|\Gamma_T|_2 \|\cdot\|_2,\tau) \\
& \leq (C2^{J[d/2+(d/2+\kappa-1)_+]} |\Gamma_T|_2 /\tau)^{dv_J}
\end{align*}
and hence
$$J(\mathcal{G}_T,d_L,D_{\mathcal{G}_T}) \lesssim 2^{J[d+(d/2+\kappa-1)_+]} |\Gamma_T|_2^2 \left( 1+ \sqrt{\log (1/|\Gamma_T|_2)}\right) = \sqrt{T}\tilde{R}_T.$$
In exactly the same way, $$\sup_{g_\gamma \in \mathcal{G}_T} \|L^{-1}[g_\gamma]\|_\infty \lesssim 2^{J[d/2+(d/2+\kappa-2)_+]} \|\gamma\|_{L^2}^2 \leq 2^{J[d/2+(d/2+\kappa-2)_+]} |\Gamma_T|_2^2.$$ Applying Lemma \ref{lem:maxim} thus gives
\begin{align*}
& E_{b_0} \sup_{\gamma\in\Gamma_T} \left| \frac{1}{T}\int_0^T \|\gamma(X_t)\|^2 dt -\int_{\T^d}\|\gamma(x)\|^2 d\mu_0(x) \right| \\
& = \frac{1}{\sqrt{T}} E_{b_0} \sup_{g_\gamma \in \mathcal{G}_T} \left| \G_T(g_\gamma) \right| \lesssim \frac{1}{\sqrt{T}}2^{J[d+(d/2+\kappa-1)_+]} |\Gamma_T|_2^2 \left( 1 +  \sqrt{\log (1/|\Gamma_T|_2)} \right) 
\end{align*}
which equals $\tilde{R}_T.$ Finally, the proof of Part (iii) follows in the same way, using that $\|\cdot\|_{L^2}\le \|\cdot\|_\infty$ and replacing $M_T$ by $\bar M_T$.
\end{proof}

\subsection{Change of measure}\label{aux2}

In this section, let $\Pi=\Pi_T$ be the prior from (\ref{eq:Gauss_prior_wav}). Using the lower bound for the small-ball probability \eqref{eq:thm_small_ball} in the proof of Theorem \ref{thm:contraction}, the proof of the following lemma is similar to the one of Theorem 8.20 in \cite{GvdV17}, and hence omitted.

\begin{lemma}\label{lem:post_prob_vanish}
Suppose $b_0\in C^s$ for $s>0$. Then there exists a finite constant $C=C(b_0)>0$ such that if $B_T$ are measurable sets satisfying $\Pi_T(B_T) = o(e^{-CT\varepsilon_T^2})$ for $\varepsilon_T=T^{-\frac{a\wedge s}{2a+d}}(\log T)$, then $E_{b_0} \Pi_T(B_T|X^T) \rightarrow 0$.
\end{lemma}

We now bound the ratio of Gaussian integrals from (\ref{lambo}) in Proposition \ref{prop:laplace_trans}.

\begin{lemma}\label{lem:change_of_meas}
(i) Suppose $b_0\in C^s(\T^d)\cap H^s(\T^d)$ for some $s>\max(d/2,1)$. Let $2^J \approx T^{\frac{1}{2a+d}}$ for $a>\max(d-1,1/2)$ and $\varepsilon_T=T^{-\frac{a\wedge s}{2a+d}}(\log T)$.  Let $D_T$ be as in (\ref{D_T}) for a choice of $\Gamma_T \subset V_J^{\otimes d}$ whose envelopes from (\ref{enve}) satisfy $|\Gamma_T|_2 = O(\sqrt{T}\varepsilon_T)$ and $\varepsilon_T \sigma_{\Gamma_T} \to 0$ as $T \to \infty$. Then for all $M>0$ large enough, $\Pi(D_T|X^T) = 1-o_{P_{b_0}}(1)$. Moreover for $b_u$ as in (\ref{pertu}) and all $u\in \R$,
\begin{align}\label{cammart}
\frac{\int_{D_T} e^{\ell_T(b_u)} d\Pi(b)}{\int_{D_T} e^{\ell_T(b)} d\Pi(b)} = 1+\zeta_T(u) \leq C_T e^{r_Tu^2},
\end{align}
where $\zeta_T(u)=o_{P_{b_0}}(1)$ for every fixed $u$, where both $C_T = O_{P_{b_0}}(1)$ and non-random $r_T = o(1)$ are independent of $u$, and all terms are uniform over $\gamma \in \Gamma_T$.

(ii) The conclusion of Part (i) remains true for $d \le 4$ and under the conditions of Theorem \ref{super} if $D_T$ is replaced by the set $\bar D_T$ from (\ref{barD_T}) with $\bar M_T = (\log T)^{\delta-1}, \delta>5/2$, and if in addition $|\Gamma_T|_{2}=O(1)$ as $T \to \infty$.
\end{lemma}

\begin{proof}
(i) The set of $b$'s satisfying the $L^2$-constraint in the definition \eqref{D_T} of $D_T$ has posterior probability tending to one by Theorem \ref{thm:contraction}. Recall that by definition of the RKHS, $\langle b,\gamma \rangle_\H \sim N(0,\|\gamma\|_\H^2)$ for $b\sim \Pi$ and $\gamma \in \H = V_J^{\otimes d}$. By Dudley's metric entropy inequality (Section 2.3 in \cite{ginenickl2016}) applied to the Gaussian process $(\langle b, \gamma \rangle_{\mathbb H}: \gamma \in \Gamma_T)$ indexed by bounded subsets of the finite-dimensional space $V_J^{\otimes d}$ (with covering numbers bounded in Proposition 4.3.34 in \cite{ginenickl2016}), we have
\begin{align}\label{eq:GP_expect}
E^{\Pi} \sup_{\gamma\in\Gamma_T} |\langle b,\gamma \rangle_{\H}|\lesssim 2^{Jd/2} \sigma_{\Gamma_T} \sqrt{\log(1/\sigma_{\Gamma_T})} \le M_0 \sqrt{T}\varepsilon_T \sigma_{\Gamma_T}
\end{align}
for some $M_0>0$, since we may always take $\sigma_{\Gamma_T} \geq 1$. By the Borell-Sudakov-Tsirelson inequality (Theorem 2.5.8 of \cite{ginenickl2016}), for $M>M_0$,
\begin{align*}
& \Pi  \big(  \sup_{\gamma \in \Gamma_T} |\langle b, \gamma \rangle_{\H}| > M \sqrt{T}\varepsilon_T\sigma_{\Gamma_T} \big) \\
&\quad \leq \Pi \big( \sup_{\gamma\in\Gamma_T} |\langle b,\gamma \rangle_{\H}| > E^{\Pi} \sup_{\gamma \in \Gamma_T} |\langle b,\gamma \rangle_{\H}| +  (M-M_0) \sqrt{T}\varepsilon_T\sigma_{\Gamma_T} \big) \\
& \quad \leq e^{-\tfrac{1}{2} (M-M_0)^2 T\varepsilon_T^2}.
\end{align*}
Taking $M>0$ large enough, the posterior probability of the set in the last display is then $o_{P_{b_0}}(1)$ by Lemma \ref{lem:post_prob_vanish}. In summary this establishes that $\Pi(D_T^c |X_T) = o_{P_{b_0}}(1)$.

We now establish (\ref{cammart}). Letting $\Pi_u$ denote the law of $b_u$ under the prior and applying the Cameron-Martin theorem (Theorem 2.6.13 of \cite{ginenickl2016}), the desired ratio equals
\begin{align}\label{ratio}
\frac{\int_{D_{T,u}} e^{\ell_T(g)} \frac{d\Pi_u}{d\Pi}(g) d\Pi(g) }{\int_{D_T} e^{\ell_T(g)} d\Pi(g)} = \frac{\int_{D_{T,u}} e^{\ell_T(g)} e^{-\frac{u}{\sqrt{T}} \langle \gamma,g\rangle_{\H} - \frac{u^2}{2T}\|\gamma\|_{\H}^2} d\Pi(g) }{\int_{D_T} e^{\ell_T(g)} d\Pi(g)},
\end{align}
where $D_{T,u} = \{ g=b_u: b\in D_T\}$. By the definition of $D_T$,
\begin{align*}
\sup_{g\in D_{T,u},\gamma\in \Gamma_T} \left| \frac{u}{\sqrt{T}} \langle \gamma,g\rangle_\H + \frac{u^2}{2T}\|\gamma\|_{\H}^2 \right| & \leq \frac{|u|}{\sqrt{T}} \sup_{b\in D_T,\gamma\in\Gamma_T} |\langle\gamma,b-\frac{u\gamma}{\sqrt T} \rangle_\H| + \frac{u^2\sigma_{\Gamma_T}^2}{2T} \\
& \leq |u| M\varepsilon_T \sigma_{\Gamma_T} + \frac{3u^2\sigma_{\Gamma_T}^2}{2T}.
\end{align*}
We thus upper bound \eqref{ratio} by
\begin{align}\label{ratio2}
e^{\tilde{r}_Tu^2 + \tilde{r}_T'|u|} \frac{\int_{D_{T,u}} e^{\ell_T(g)} d\Pi(g)}{\int_{D_T} e^{\ell_T(g)} d\Pi(g)} = e^{\tilde{r}_Tu^2 + \tilde{r}_T'|u|}  \frac{\Pi(D_{T,u}|X^T)}{\Pi(D_T|X^T)},
\end{align}
where $\tilde{r}_T,\tilde{r}_T'\rightarrow 0$ are non-random and uniform over $\gamma\in \Gamma_T$. Since $\alpha |u| \leq \alpha^2 u^2 +1$ for all $\alpha \geq 0$ and $u\in \R$, the exponential in the last display is bounded by $e^{r_T u^2 + 1}$ for all $u\in \R$, where $r_T = \tilde{r}_T + (\tilde{r}_T')^2 = 3\sigma_{\Gamma_T}^2 /(2T) + M^2 \varepsilon_T^2\sigma_{\Gamma_T}^2\rightarrow 0$. Since we have already shown that $\Pi(D_T|X^T) = 1-o_{P_{b_0}}(1)$ and the posterior probability $\Pi(D_{T,u}|X^T)$ is bounded by one, the inequality in (\ref{cammart}) follows.

Turning to the exact asymptotics for fixed $u\in \R$, the right hand side in \eqref{ratio2} equals $\Pi(D_{T,u}|X^T)(1+o_{P_{b_0}}(1))$, and \eqref{ratio} can be lower bounded by \eqref{ratio2} with $e^{\tilde{r}_Tu^2 + \tilde{r}_T'|u|} $ replaced by $e^{-\tilde{r}_Tu^2 - \tilde{r}_T'|u|}$. It consequently suffices to prove $\Pi(D_{T,u}|X^T) = 1-o_{P_{b_0}}(1)$. Now
\begin{align*}
\Pi(D_{T,u}^c|X^T) & \leq \Pi( g\in V_J^{\otimes d}: \|g+\tfrac{u}{\sqrt{T}}\gamma-b_{0}\|_{\mu_0} > M_T \varepsilon_T|X^T) \\
& \quad + \Pi \Big(  g\in V_J^{\otimes d}: \sup_{\gamma \in \Gamma_T} |\langle g+ \tfrac{u}{\sqrt{T}}\gamma, \gamma \rangle_{\H}| > M \sqrt{T}\varepsilon_T\sigma_{\Gamma_T} \big| X^T \Big).
\end{align*}
By Proposition \ref{as:true_drift}, $\|\tfrac{u}{\sqrt{T}} \gamma\|_{\mu_0}\lesssim \tfrac{|u|}{\sqrt{T}} |\Gamma_T|_2 = O(\varepsilon_T)= o(M_T\varepsilon_T)$, so that the first posterior probability tends to zero by Theorem \ref{thm:contraction}. Using \eqref{eq:GP_expect}, that $\sigma_{\Gamma_T}^2/\sqrt{T} = o(\sqrt{T}\varepsilon_T \sigma_{\Gamma_T})$ and the Borell-Sudakov-Tsirelson inequality (Theorem 2.5.8 of \cite{ginenickl2016}), the prior probability of the second event is bounded by
\begin{align*}
& \Pi \Big(g: \sup_{\gamma\in\Gamma_T} |\langle g,\gamma \rangle_{\H}| + \frac{|u|\sigma_{\Gamma_T}^2}{\sqrt{T}} > E^{\Pi} \sup_{\gamma \in \Gamma_T} |\langle b,\gamma \rangle_{\H}| +  (M-M_0) \sqrt{T}\varepsilon_T\sigma_{\Gamma_T} \Big) \\
& \leq e^{-\tfrac{1}{4} (M-M_0)^2 T\varepsilon_T^2}
\end{align*}
for $T$ large enough depending on $u$. For $M>0$ large enough, Lemma \ref{lem:post_prob_vanish} then yields that the posterior probability of this last set is $o_{P_{b_0}}(1)$, which shows $\Pi(D_{T,u}|X^T) = 1-o_{P_{b_0}}(1)$ as desired. 

Part (ii) is proved in the same way using Theorem \ref{super} (whose proof only relies on Part (i) of the present lemma) to ensure that $\Pi(\bar D_T|X^T) \to^{P_{b_0}} 1$ as $ T \to \infty$, and upon noting that $$\|\tfrac{u}{\sqrt{T}} \gamma\|_{\infty}\lesssim \tfrac{2^{Jd/2}|u|}{\sqrt{T}} |\Gamma_T|_2 = O(\varepsilon_T)= o(\bar M_T\varepsilon_T)$$ since $|\Gamma_T|_{2}=O(1)$ as $T \to \infty$.
\end{proof}

\subsection{An approximation lemma} \label{aux4}

\begin{lemma}\label{lem:semipara_bias}
Suppose $b_{0}\in C^s(\T^d)$ for some $s\geq1$ and let $\lambda \leq J$, $1 \leq j \leq d$ and $b \in V_J^{\otimes d}$. If $a_\lambda = 2^{\lambda d/2}2^{-Jd/2} (\log T)^{-\eta}$ for some $\eta\ge 0$, then
\begin{align*}
&a_\lambda  |\langle \mu_0(b_j-b_{0,j}), \Phi_{\lambda,k}/\mu_0 - P_{V_J} [\Phi_{\lambda,k}/\mu_0] \rangle_{L^2}| \\
&~~~~\leq  C(2^{-2J} \|b_j-b_{0,j}\|_{L^2} +  2^{-J(s+d/2+1)})(\log T)^{-\eta}.
\end{align*}
If instead $a_\lambda = 2^{\lambda d/2} 2^{-J(d-2)} (\log T)^{-\eta}$ for some $\eta\ge 0$, then
\begin{align*}
&a_\lambda  |\langle \mu_0(b_j-b_{0,j}), \Phi_{\lambda,k}/\mu_0 - P_{V_J} [\Phi_{\lambda,k}/\mu_0] \rangle_{L^2}| \\
&~~~~\leq C (2^{-Jd/2} \|b_j-b_{0,j}\|_{L^2} +  2^{-J(s+d-1)}) (\log T)^{-\eta}.
\end{align*}
Finally,
$$\big|\big \langle \mu_0(b_j-b_{0,j}), \frac{\Phi_{\lambda,k}}{\mu_0} - P_{V_J} \big[\frac{\Phi_{\lambda,k}}{\mu_0}\big] \big\rangle_{L^2}\big| \leq  C2^{-\lambda d/2} \big(2^{-2J} \|b_j-b_{0,j}\|_{\infty} +  2^{-J(s+1)} \big).$$
In all cases, the constant $C$ depends only on $b_0$, $\Phi$ and $d$.
\end{lemma}

\begin{proof}
By the triangle inequality, the desired quantity is bounded by
\begin{align*}
& a_\lambda  |\langle \mu_0(b_j-P_{V_J} b_{0,j}), \Phi_{\lambda,k}/\mu_0 - P_{V_J} [\Phi_{\lambda,k}/\mu_0] \rangle_{L^2}| \\
& \quad  + a_\lambda  |\langle \mu_0(b_{0,j}-P_{V_J} b_{0,j}), \Phi_{\lambda,k}/\mu_0 - P_{V_J} [\Phi_{\lambda,k}/\mu_0] \rangle_{L^2}| =: (I) + (II).
\end{align*}
By Parseval's identity,
\begin{equation}\label{eq:bias_lem}
\begin{split}
(I) & = a_\lambda \left| \sum_{l >J} \sum_r  \langle \mu_0(b_j-P_{V_J}b_{0,j}), \Phi_{l,r} \rangle_{L^2} \langle \Phi_{\lambda,k}/\mu_0 ,\Phi_{l,r} \rangle_{L^2} \right| \\
& \quad \leq a_\lambda \sum_{l>J} \max_r |\langle \mu_0(b_j-P_{V_J}b_{0,j}), \Phi_{l,r} \rangle_{L^2}| \sum_r |\langle \Phi_{\lambda,k}/\mu_0 ,\Phi_{l,r} \rangle_{L^2}|.
\end{split}
\end{equation}
By Proposition \ref{as:true_drift} we know that $\mu_0$ has finite Lipschitz norm $\|\mu_0\|_\text{Lip}$. Let $x_{l,r}\in I_{l,r} := \supp (\Phi_{l,r})$ and note that $\text{diam}(I_{l,r}) = O(2^{-l})$ by construction of the wavelets. Using that $b_j - P_{V_J}b_{0,j} \in V_J$ is orthogonal to $\Phi_{l,r}$ for any $l>J$, $\|\Phi_{l,r}\|_{L^1}\lesssim 2^{-ld/2}$ and \eqref{eq:wav_bd_2},
\begin{align}\label{holder}
&|\langle \mu_0(b_j-P_{V_J} b_{0,j}), \Phi_{l,r} \rangle_{L^2}| \notag \\
&= \left| \int_{\T^d} (\mu_0(x)-\mu_0(x_{l,r})) (b_j(x)-P_{V_J}b_{0,j}(x)) \Phi_{l,r}(x)  dx  \right| \notag \\
& \leq \|\mu_0\|_\text{Lip} \text{diam}(I_{l,r})  \int_{\T^d} |b_j(x)-P_{V_J}b_{0,j}(x)| |\Phi_{l,r}(x)| dx \notag \\
& \leq C(b_0,\Phi) 2^{-l} \|b_j-P_{V_J}b_{0,j}\|_{\infty} \|\Phi_{l,r}\|_{L^1} \\ \notag
&\leq C(b_0,\Phi) 2^{Jd/2} \|b_j-P_{V_J}b_{0,j}\|_{L^2} 2^{-l(d/2+1)}. 
\end{align}
Moreover, using the following standard properties of wavelet bases, $$\sup_x \sum_r |\Phi_{l,r}(x)| \lesssim 2^{ld/2}, l\geq 0,~~ \langle \Phi_{\lambda,k},\Phi_{l,r}\rangle_{L^2} = 0,~~\lambda\leq J<l,$$ we deduce
\begin{align}\label{locali}
\sum_r |\langle \Phi_{\lambda,k}/\mu_0,\Phi_{l,r}\rangle_{L^2}| & =  \sum_r \left| \int_{\T^d} \left( \frac{1}{\mu_0(x)} - \frac{1}{\mu_0(x_{l,r})} \right) \Phi_{\lambda,k}(x) \Phi_{l,r}(x) dx \right| \notag \\
& \leq  \|1/\mu_0\|_\text{Lip} \text{diam}(I_{l,r}) \int_{\T^d} |\Phi_{\lambda,k}(x)| \sum_r |\Phi_{l,r}(x)| dx \notag \\
& \leq C(b_0,\Phi) 2^{l(d/2-1)} 2^{-\lambda d/2}.
\end{align}
Substituting \eqref{holder} and \eqref{locali} into \eqref{eq:bias_lem} yields
\begin{align*}
(I) & \lesssim a_\lambda 2^{Jd/2} 2^{-\lambda d/2} \|b_j-P_{V_J}b_{0,j}\|_{L^2} \sum_{l>J} 2^{-2l} \lesssim  2^{-2J} (\log T)^{-\eta} \|b_j-b_{0,j}\|_{L^2}
\end{align*}
as desired. Next expanding $(II)$ as in \eqref{eq:bias_lem} and using \eqref{locali},
\begin{align*}
(II) &  \leq C(b_0,\Phi) a_\lambda \sum_{l>J} \max_r |\langle \mu_0(b_{0,j}-P_{V_J}[b_{0,j}]), \Phi_{l,r} \rangle_{L^2}| 2^{l(d/2-1)} 2^{-\lambda d/2} \\
&\lesssim a_\lambda 2^{-Js} 2^{-\lambda d/2} \sum_{l>J} 2^{-l} \lesssim 2^{-J(s+1+d/2)} (\log T)^{-\eta},
\end{align*}
where we have used H\"older's inequality and $\|b_{0,j}-P_{V_J}[b_0,j]\|_\infty \lesssim 2^{-Js}$. The last two displays imply the first inequality in the lemma.  If instead $a_\lambda = 2^{\lambda d/2} 2^{-J(d-2)} (\log T)^{-\eta}$, then substituting this value into the final bounds for $(I)$ and $(II)$ gives the required result. The final inequality of the lemma is proved in the same way, but by using the penultimate $\|\cdot\|_\infty$ bound in (\ref{holder}) rather than the $\|\cdot\|_{L^2}$ bound in the last line. Taking $a_\lambda=1$ in the rest of the argument gives the result.
\end{proof}

\subsection{Proof of Lemma \ref{maxwav}}\label{aux3}

\begin{lemma}\label{lem:change_of_meas_L_inf}
Suppose $b_0\in C^s(\T^d)$ for some $s>d/2$. Then for $\Gamma_T$ as in (\ref{gama_main}) and the RKHS norm $\|\cdot\|_{\H}$ defined in \eqref{eq:RKHS_inner_prod} with $\sigma_l = 2^{-l(\alpha+d/2)}$ for $\alpha\geq 0$, we can take the envelopes from (\ref{enve}) as $$|\Gamma_T|_2 = C(d,\mu_0) \max_{\lambda \le J} a_\lambda, \quad \sigma_{\Gamma_T} = C(d,\mu_0,\Phi) 2^{J(\alpha+d/2)} \max_{\lambda \le J} a_\lambda.$$
\end{lemma}
\begin{proof}
Using Proposition \ref{as:true_drift}, $$\|\tilde{\Phi}_{\lambda,k,j}\|_{L^2}^2 = \|P_{V_J}[a_\lambda \Phi_{\lambda,k}/\mu_0]\|_{L^2}^2 \leq a_\lambda^2 \|1/\mu_0\|_{\infty}^2 \|\Phi_{\lambda,k}\|_{L^2}^2 \lesssim a_\lambda^2.$$ To prove the second bound, note that Proposition \ref{as:true_drift} implies that $1/\mu_0$ has finite Lipschitz norm $\|1/\mu_0\|_\text{Lip}$ on $\T^d$. Let $x_{l,r}\in \supp (\Phi_{l,r})$ and note that $\text{diam}(\supp (\Phi_{l,r})) = O(2^{-l})$ by construction of the wavelets. Using the orthogonality of the wavelets and H\"older's inequality, for $(l,r) \neq (\lambda,k)$ with $\lambda \leq J$,
\begin{equation*}
\begin{split}\label{wav_lip}
|\langle P_{V_J}[\Phi_{\lambda,k}/\mu_0], \Phi_{l,r} \rangle_{L^2} | &=  \left| \int_{\T^d} \Phi_{\lambda,k}(x)\Phi_{l,r}(x) \left( \frac{1}{\mu_0(x)}-\frac{1}{\mu_0(x_{l,r})} \right) dx \right| \\
& \lesssim \|1/\mu_0\|_\text{Lip} 2^{-\max (l,\lambda)} \int_{\T^d} |\Phi_{\lambda,k}(x)||\Phi_{l,r}(x)| dx\\
& \lesssim 2^{-\max (l,\lambda)} 2^{-|l-\lambda|d/2},
\end{split}
\end{equation*}
while $|\langle P_{V_J}[\Phi_{\lambda,k}/\mu_0], \Phi_{\lambda,k} \rangle_{L^2}| \leq \|1/\mu_0\|_\infty \|\Phi_{\lambda,k}\|_{L^2}^2 \lesssim 1$. Note that for $l\leq \lambda$, there are a constant number of wavelets $\Phi_{l,r}$ intersecting $\supp (P_{V_J}[\Phi_{\lambda,k}/\mu_0])$, while for $l\geq \lambda$, there are $O(2^{(l-\lambda)d})$ such wavelets. Splitting the following sum into these two cases, while separately keeping track of the term $(l,r) = (\lambda,k)$, and using the above bounds gives that for $\lambda \leq J$,
\begin{align*}
\|\tilde{\Phi}_{\lambda,k,j}\|_{\H}^2 & = \sum_{l\leq J} \sum_r \sigma_l^{-2} |\langle P_{V_J}[a_\lambda \Phi_{\lambda,k}/\mu_0], \Phi_{l,r} \rangle_{L^2} |^2 \\
& \lesssim a_\lambda^2 \left(  \sum_{l=0}^{\lambda} \sigma_l^{-2} 2^{-2\lambda-(\lambda-l)d} + \sum_{l=\lambda+1}^J \sigma_l^{-2} 2^{(l-\lambda)d} 2^{-2l -(l-\lambda)d} + \sigma_\lambda^{-2} \right) \\
& \lesssim a_\lambda^2 \left( 2^{-(d+2)\lambda} \sum_{l=0}^\lambda 2^{2l(\alpha+d)} + \sum_{l=\lambda+1}^J 2^{(2\alpha+d-2)l} + 2^{(2\alpha+d)\lambda} \right) \\
& \lesssim a_\lambda^2 2^{(2\alpha+d)J} .
\end{align*}
This yields $\sigma_{\Gamma_T}^2 \lesssim a_\lambda^2 2^{(2\alpha+d)J}$.
\end{proof}

We now prove Lemma \ref{maxwav} from above. The first assertion in Part (i) follows from Lemmas \ref{lem:change_of_meas} and \ref{lem:change_of_meas_L_inf} since the envelopes satisfy $|\Gamma_T|_2 =O(a_J)=O((\log T)^{-\eta})$, $\sqrt T \varepsilon_T \to \infty$ and $\varepsilon_T \sigma_{\Gamma_T} =  o(1)$ for $0 \leq \alpha < a\wedge s-d/2$ and the specified choice $2^J \approx T^{1/(2a+d)}$. To prove the first two maximal inequalities, we start by verifying the conditions of Proposition \ref{prop:laplace_trans} for the case $d \le 4$. Using the envelopes from Lemma \ref{lem:change_of_meas_L_inf}, we can bound $R_T$ and $\tilde{R}_T$ in Proposition \ref{prop:laplace_trans} by
\begin{align*}
R_T \lesssim  M_T T^{\frac{d-a\wedge s+(d/2+\kappa-1)_+}{2a+d}} (\log T)^{3/2-\eta} \rightarrow 0
\end{align*}
for $a\wedge s>\max(3d/2-1,1)$, $0<\kappa <a\wedge s-3d/2+1$ ($\kappa=0$ if $d=1$) and $M_T \rightarrow \infty$ slowly enough, while
\begin{align}\label{eq:tilde(R)_T}
\tilde{R}_T \lesssim  T^{-\frac{1}{2} + \frac{d+(d/2+\kappa-1)_+}{2a+d}} (\log T)^{-2\eta} \sqrt{\log \log T} \rightarrow 0
\end{align}
for $a> \max(d-1,1/2)$ and $0<\kappa <s-d+1$ ($\kappa=0$ if $d=1$). We may thus apply Proposition \ref{prop:laplace_trans} to $G(b) = \langle b_j-b_{0,j},a_\lambda \Phi_{\lambda,k}\rangle_{L^2}$ with $\gamma=\tilde{\Phi}_{\lambda,k,j}\in V_J^{\otimes d}$, so that for $b_u = b-\tfrac{u}{\sqrt{T}}\tilde{\Phi}_{\lambda,k,j}$, $u\in \R$,
\begin{align*}
E^{\Pi^{D_T}}[e^{u\sqrt{T}\langle b_j-b_{0,j},a_\lambda \Phi_{\lambda,k}\rangle_{L^2}} |X^{T}] & = e^{\frac{u}{\sqrt{T}} \int_0^T \tilde{\Phi}_{\lambda,k,j} (X_t).dW_t +  \frac{u^2}{2}\|\tilde{\Phi}_{\lambda,k,j}\|_{\mu_0}^2 + ur_T + u^2\tilde{r}_T} \\
& \quad \times  \frac{\int_{D_T} e^{S_T(b) + \ell_T(b_u)} d\Pi(b)}{\int_{D_T} e^{\ell_T(b)} d\Pi(b)},
\end{align*}
where $r_T = O_{P_{b_0}}(R_T)$, $\tilde{r}_T = O_{P_{b_0}}(\tilde{R}_T)$ uniformly over $\lambda \leq J$ and $k,j$, and
\begin{align*}
S_T(b) & = u\sqrt{T}   \langle b_j-b_{0,j},a_\lambda \Phi_{\lambda,k}\rangle_{L^2} -u\sqrt{T}  \langle b_j-b_{0,j},a_\lambda P_{V_J} [\Phi_{\lambda,k}/\mu_0] \rangle_{\mu_0}  \\
& =  u\sqrt{T} a_\lambda  \langle \mu_0(b_j- b_{0,j}), \Phi_{\lambda,k}/\mu_0 - P_{V_J} [\Phi_{\lambda,k}/\mu_0] \rangle_{L^2}.
\end{align*}
Applying the first bound from Lemma \ref{lem:semipara_bias} with $d\leq 4$, $\eta>1$,
\begin{align} \label{biao}
\sup_{b\in D_T} |S_T(b)| \lesssim |u| \sqrt{T} (\log T)^{-\eta} (2^{-2J} M_T \varepsilon_T + 2^{-J(s+1+d/2)}) =|u| \times o(1),
\end{align}
for $M_T \rightarrow \infty$ slowly enough and $s \geq a-1$. Applying $\alpha |u| \leq \alpha^2 u^2 +1$ for all $\alpha \geq 0$ to \eqref{biao}, and using Lemmas \ref{lem:change_of_meas} and \ref{lem:change_of_meas_L_inf} gives for all $u\in \R$,
\begin{align*}
\frac{\int_{D_T} e^{S_T(b) + \ell_T(b_u)} d\Pi(b)}{\int_{D_T} e^{\ell_T(b)} d\Pi(b)} \leq C_T e^{c_T u^2},
\end{align*}
where $C_T = O_{P_{b_0}}(1)$ and non-random $c_T = o(1)$ are independent of $u$ and uniform over $\lambda \leq J$ and $k,j$.

Setting $Z_{\lambda,k,j}  = \langle b_j - b_{0,j},a_\lambda \Phi_{\lambda,k}\rangle_{L^2} - \tfrac{1}{T} \int_0^T \tilde{\Phi}_{\lambda,k,\cdot} (X_t).dW_t$ and using again that $\alpha |u| \leq \alpha^2 u^2 +1$ for all $\alpha \geq 0$, the Laplace transform satisfies the conditional subgaussian bound
\begin{align*}
E^{\Pi^{D_T}}[e^{u\sqrt{T}Z_{\lambda,k,j}} |X^{T}] \leq C_T' e^{\frac{u^2}{2}(\|\tilde{\Phi}_{\lambda,k,j}\|_{\mu_0}^2+c_T')}
\end{align*}
for sequences $C_T' =O_{P_{b_0}}(1)$ and $c_T' = 2(\tilde{r}_T + r_T^2 + c_T) = o_{P_{b_0}}(1)$, which are independent of $u$ and uniform over $\lambda \leq J$ and $k,j$. Since $\|\tilde{\Phi}_{\lambda,k,j}\|_{\mu_0}^2 = a_\lambda^2 \|P_{V_J}[\Phi_{\lambda,k}/\mu_0]\|_{\mu_0}^2 \leq a_\lambda^2 \|\mu_0\|_{\infty}  \|1/\mu_0\|_\infty^2 \leq C(b_0)$, standard subgaussian inequalities (Lemmas 2.3.2 and 2.3.4 of \cite{ginenickl2016}) give
\begin{align}\label{miracle}
&\sqrt{T} E^{\Pi^{D_T}} [\max_{\lambda \leq J,k,j} |Z_{\lambda,k,j}| |X^T]  \\
&\lesssim \sqrt{2 \log 2 \text{dim} (V_J^{\otimes d})} (C_T'+1) \max_{\lambda \leq J,k,j} (\|\tilde{\Phi}_{\lambda,k,j}\|_{\mu_0}^2 + c_T') =O_{P_{b_0}}( \sqrt{J}) \notag
\end{align}
since $\text{dim}(V_J^{\otimes d}) = O(d2^{Jd})$. Similarly, for $\lambda \le J$,
\begin{align}\label{nmiracle}
&\sqrt{T} E^{\Pi^{D_T}} [\max_{k,j} |Z_{\lambda,k,j}| |X^T] \\
& \lesssim \sqrt{2 \log 2 \text{dim} (V_\lambda^{\otimes d})} (C_T'+1) \max_{k,j} (\|\tilde{\Phi}_{\lambda,k,j}\|_{\mu_0}^2 + c_T') =O_{P_{b_0}}( \sqrt{\lambda}). \notag
\end{align}
We now deduce (\ref{full}) from (\ref{miracle}), the same arguments then also show that (\ref{nfull}) follows from (\ref{nmiracle}). Decompose
\begin{align*}
& E^{\Pi^{D_T}} \left[ \left. \max_{\lambda\leq J, k,j} \sqrt{T}|\langle b_j-b_{0,j},a_\lambda \Phi_{\lambda,k}\rangle_{L^2}| \right| X^T \right]  \\
& \quad  \leq E^{\Pi^{D_T}} \left[ \left. \max_{\lambda\leq J, k,j} \sqrt{T}|Z_{\lambda,k,j}| \right| X^T \right] +  \max_{\lambda\leq J, k,j} \left| \frac{1}{\sqrt{T}} \int_0^T \tilde{\Phi}_{\lambda,k,j} (X_t).dW_t \right|
\end{align*}
and we have shown the first term is $O_{P_{b_0}}(\sqrt{J})$. We now control the $P_{b_0}$-expectation of the second term by showing that $M_T^{\lambda,k,j} =  \int_0^T \tilde{\Phi}_{\lambda,k,j} (X_t).dW_t$ are subgaussian with uniform constants on a suitable event $A_T$. For $\epsilon>0$ fixed, set
\begin{align}\label{eq:event_A}
A_T = \left\{ \max_{\lambda,k,j} \left| \frac{1}{T}\int_0^T \|\tilde{\Phi}_{\lambda,k,j}(X_t)\|^2 dt -\|\tilde{\Phi}_{\lambda,k,j}\|_{\mu_0}^2 \right| \leq \epsilon \right\} .
\end{align}
Applying Markov's inequality, Proposition \ref{prop:laplace_trans}(ii) and \eqref{eq:tilde(R)_T}, $P_{b_0}(A_T^c) \lesssim \epsilon^{-1} \tilde{R}_T \rightarrow 0$. On $A_T$ we have $$T^{-1}[M^{\lambda,k,j}]_T = T^{-1} \int_0^T \|\tilde{\Phi}_{\lambda,k,j}(X_T)\|^2 dt \leq \|\tilde{\Phi}_{\lambda,k,j}\|_{\mu_0}^2 + \epsilon \leq C_0(b_0) + \epsilon,$$ so that applying Bernstein's inequality \eqref{eq:Bernstein}, for any $x >0$,
\begin{align*}
P_{b_0} (T^{-1/2} |M_T^{\lambda,k,j}| 1_{A_T} \geq x) &\leq P_{b_0} (|M_T^{\lambda,k,j}| \geq x\sqrt{T}, [M^{\lambda,k,j}]_T \leq (C_0+\epsilon)T ) \\
& \leq 2e^{-\frac{x^2}{2(C_0+\epsilon)}}.
\end{align*}
Consequently, $(T^{-1/2} M^{\lambda,k,j}1_{A_T}: \lambda,k)$ are subgaussian random variables with uniformly bounded constants, so that $E_{b_0} \max_{\lambda\leq J,k,j}T^{-1/2}|M_T^{\lambda,k,j}|1_{A_T} = O(\sqrt{J})$ by Lemma 2.3.4 of \cite{ginenickl2016}. 
When $d \geq 5$, one proceeds exactly as above with the only difference to the case $d\leq 4$ being that we use the second bound in Lemma \ref{lem:semipara_bias} with $a \leq s+d/2-1$ rather than the first bound, which is needed to ensure $\sup_{b\in D_T} |S_T(b)| = o(|u|)$.

For Part (ii), we can invoke Lemma \ref{lem:change_of_meas}(ii) to obtain $\Pi(\bar D_T|X^T) \to 1$ in $P_{b_0}$-probability. The maximal inequality then follows from the same proof as in (i), using Proposition \ref{prop:laplace_trans}(iii), that the $\|\cdot\|_\infty$-contraction rate implies the same rate in $L^2$-norm, and replacing the bias bound (\ref{biao}) by the third inequality of Lemma \ref{lem:semipara_bias} so that scaling by $a_\lambda$ is not necessary. The conditions $s>a-1+d/2$ and $d<4$ ensure that the terms $\sqrt T 2^{-J(s+1)}$ and $\sqrt T2^{-2J}\|b-b_0\|_\infty$ are both $o(1)$ and hence asymptotically negligible.

\section{Proofs for Section \ref{bvmsec}}\label{prfl}

The proofs follow from Theorems \ref{bvm1} and \ref{mapas} and a version of the `Delta'-method for weak convergence applied to the map $b \mapsto \mu_b$. We will represent $\mu_{b}-\mu_{b+h}$ by a linear transformation of the vector field $h$ plus a remainder term that will be seen to be quadratic in (suitable norms of) $h$. The identity (\ref{deltam}) below is the key to these proofs and can be derived from perturbation arguments for the PDE (\ref{cpde}) as follows: let $\mu_b$ and $\mu_{b+h}$ correspond to vector fields $b,b+h \in C^1(\T^d)$ (cf.~Proposition \ref{as:true_drift}). Then necessarily $L^*_b \mu_b = L^*_{b+h} \mu_{b+h}$ or in other words 
$$\frac{\Delta}{2} \mu_b - b . \nabla \mu_b - div(b) \mu_b = \frac{\Delta}{2} \mu_{b+h} - (b+h) . \nabla \mu_{b+h} - div(b+h) \mu_{b+h},$$
which is the same as
$$\frac{\Delta}{2} (\mu_b - \mu_{b+h}) - b . \nabla (\mu_b - \mu_{b+h}) - div(b)(\mu_b-\mu_{b+h}) = -h . \nabla \mu_{b+h} - div(h) \mu_{b+h}.$$
Thus $u=\mu_b - \mu_{b+h}$ solves the equation
\begin{equation}\label{pertu2}
L^*_b u = -h. \nabla \mu_{b+h} - div(h) \mu_{b+h}.
\end{equation}
Next denote by $v_h=v_{b,h}$ the unique periodic solution of the PDE
\begin{equation}\label{pertu3} L_b^*v_h = -h . \nabla \mu_b - div(h) \mu_b = -\sum_{j=1}^d \frac{\partial}{\partial x_j} (h_j \mu_b) \equiv f_h
\end{equation}
satisfying $\int v_h=0$. In view of the results in Section \ref{pdeeee} and since, with $dx^{(j)}=\prod_{i\neq j} dx_i$,
\begin{align*} 
 \int_{\T^d} f_h(x)dx &= \sum_{j=1}^d \int_0^1 \dots \int_0^1\frac{\partial}{\partial x_j} (h_j(x) \mu_{b}(x))dx \\
& = \sum_j \int_{\T^{d-1}} \Big[(h_j \mu_{b})(x_1, \dots, x_{j-1},1,x_{j+1}, \dots, x_d)  \\
&~~~~~~~~~~~~~~- (h_j \mu_{b})(x_1, \dots,x_{j-1}, 0, x_{j+1} \dots, x_d) \Big] dx^{(j)} =0, \notag
\end{align*}
such a solution exists and can be represented as $v_h = (L_b^*)^{-1}[f_h]$, a map that is linear in $h$. Now since $\int \mu_{b+h}-\int \mu_b = 1-1=0$, we can use (\ref{pertu2}) and (\ref{pertu3}) to see that the differences $w_{b,h} = \mu_b - \mu_{b+h} -v_{h}$ are the unique (periodic) integral-zero solutions of $$L_{b+h}^* w_{b,h} = L_b^*w_{b,h} - h . \nabla w_{b,h} - div(h) w_{b,h}  = \sum_{j=1}^d  \frac{\partial}{\partial x_j} [h_j v_h] \equiv \bar f_h,$$ where again $\int \bar f_h =0$ as in the penultimate display, so that we can write $w_{b,h}= (L_{b+h}^*)^{-1}[\bar f_h]$. Thus we have, for any $h \in C^1(\T^d)$, the decomposition
\begin{equation} \label{deltam}
\mu_b-\mu_{b+h} =  v_{b,h} + w_{b,h} = (L_b^*)^{-1}[f_h]+  (L_{b+h}^*)^{-1}[\bar f_h],
\end{equation}
which for sufficiently smooth $b,h$ (such that also $\mu_b, \mu_{b+h} \in C^r, r>2$, see after (\ref{chain})) holds classically (pointwise on $\T^d$).

\subsection{Proof of Theorem \ref{invabvm}}

It suffices to prove the theorem for $$\sqrt T (\mu_b - \mu_{\hat b_T})|X^T, \quad b \sim \Pi^{\bar D_T}(\cdot|X^T),$$ where $\Pi^{\bar D_T}(\cdot|X^T)$ was introduced at the beginning of the proof of Theorem \ref{bvm1} and $\bar D_T$ is given by \eqref{barD_T}. On the set $\bar D_T$ we have the estimate $$\|b\|_{B^1_{\infty \infty}} \lesssim 2^{J} \|b -b_0\|_\infty + \max_j\|b_{0,j} - P_{V_J}(b_{0,j})\|_{B^1_{\infty \infty}} = O(1)$$ as $T\to \infty$, and the same argument shows $\|\hat b_T\|_{B^1_{\infty \infty}}=O_{P_{b_0}}(1)$ by virtue of Corollary \ref{MAP}. Proposition \ref{as:true_drift} then further implies that $\|\mu_b\|_{Lip}, \|\mu_{\hat b_T}\|_{Lip}$ are also $O(1)$ and $O_{P_{b_0}}(1)$, respectively -- these bounds will be used repeatedly in the proof without further mention. Recall that $v_h = (L_b^*)^{-1}[f_h]$ and $w_{b,h} = (L_{b+h}^*)^{-1}[\bar f_h]$ for $f_h = -\sum_{j=1}^d \tfrac{\partial}{\partial x_j} (h_j \mu_b)$ and $\bar f_h = \sum_{j=1}^d  \tfrac{\partial}{\partial x_j} [h_j v_h]$. We will use the decomposition (\ref{deltam}) with $h=  \hat b_T -b$. First, for the `remainder' term, we can use (\ref{negtreg}) below and (\ref{mult}) to deduce that, uniformly in $\|g\|_{\mathbb B_r} \le 1$,
\begin{align} \label{rembd}
\Big | \int_{\T^d} w_{b,h} g \Big | &\le \|g\|_{L^2} \|(L_{b+h}^*)^{-1}[\bar f_h]\|_{L^2} \notag \\
& \lesssim \|\bar f_h\|_{H^{-2}}=\sup_{\|\phi\|_{H^2} \le 1}\Big|\sum_{j=1}^d \int_{\T^d} \phi \frac{\partial}{\partial x_j} [h_j v_h]  \notag \Big|  \\
&\le \sup_{\|\phi\|_{H^2} \le 1}\Big|\sum_{j=1}^d \int_{\T^d} h_j v_h \frac{\partial}{\partial x_j}\phi   \Big|  \lesssim \|h\|_{\infty} \|(L_b^*)^{-1}[f_h]\|_{L^2}   \\
&\lesssim \|h\|_{\infty} \|f_h\|_{H^{-2}} \lesssim \|h\|_\infty \|h \mu_b\|_{L^2}  \lesssim \|h\|_\infty^2  \notag
\end{align}
is $O_{P_{b_0}}((\|\hat b_T-b_0\|_{\infty} + \|b- b_0\|_\infty)^2) = o_{P_{b_0}}(1/\sqrt T))$ on $\bar D_T$ and by Corollary \ref{MAP}. For the `linear' term we may write, noting the dependence $f_h=f_{h,b}$ on $b$,
$$\int v_{b,h} g = \int (L_{b_0}^*)^{-1}[f_{h,b_0}] g + \int (L_{b_0}^*)^{-1}[f_{h,b}-f_{h,b_0}] g + \int [(L_{b}^*)^{-1}-(L_{b_0}^*)^{-1}][f_{h,b}] g$$ and we denote the right hand side as $A_0+A_1+A_2$. The last term $A_2$ is $o_{P_{b_0}}(1/\sqrt T)$ in $\mathbb B^*_r$ since $[(L_{b}^*)^{-1}-(L_{b_0}^*)^{-1}][f_{h,b}]$ can be written as $-(L_b^*)^{-1}[(b-b_0).\nabla \omega + div(b-b_0) \omega]$ for $\omega= (L_{b_0}^*)^{-1}[f_{h,b}]$ (arguing just as in (\ref{pertu2})), so that using (\ref{negtreg}) gives (as in (\ref{rembd})) the inequality 
\begin{equation} \label{a1}
\|[(L_{b}^*)^{-1}-(L_{b_0}^*)^{-1}][f_h]\|_{L^2} \lesssim \|b-b_0\|_\infty (\|\hat b_T - b_0\|_{\infty} + \|b-b_0\|_{\infty})=o_{P_{b_0}}(1/\sqrt T).
\end{equation}
Similarly the term $A_1$ can be bounded in $\mathbb B^*_r$ by $$\|(L_{b_0}^*)^{-1}[f_{h,b}-f_{h,b_0}]\|_{L^2} \lesssim \|f_{h,b}-f_{h,b_0}\|_{H^{-2}} \lesssim \|h\|_\infty \|b-b_0\|_{\infty} =o_{P_{b_0}}(1/\sqrt T).$$
Finally, for the term $A_0$, we show that the linear operator $$h \mapsto v_{b_0,h} = (L_{b_0}^*)^{-1}[f_{h,b_0}]$$ is Lipschitz on $C^1(\T^d)$ for the norms $(B^{r+1}_{1\infty})^*$ and $\mathbb B_r^*$ for any $d/2-1<r <1$. Using that $v_{b_0,h} \in L^2_0(\T^d)$ and writing $\bar g = g-\int g d\mu_0$,
\begin{align} \label{chain}
&\|v_{b_0,h}\|_{\mathbb B_r^*} \notag \\
&= \sup_{\|g\|_{L^2}+ \|g\|_{B^{r}_{1\infty}} \le 1} \left|\int_{\T^d} g v_{b_0,h} \right| =\sup_{\|g\|_{L^2}+ \|g\|_{B^{r}_{1\infty}} \le 1} \left|\int_{\T^d} L_{b_0}L_{b_0}^{-1}[\bar g] v_{b_0,h} \right| \notag \\
&= \sup_{\|g\|_{L^2}+\|g\|_{B^{r}_{1 \infty}} \le 1} \left| \int_{\T^d} L_{b_0}^*v_{b_0,h} L_{b_0}^{-1}[\bar g] \right| \notag \\
&=\sup_{\|g\|_{L^2}+\|g\|_{B^{r}_{1 \infty}} \le 1} \left|\sum_{j=1}^d \int_{\T^d} h_j \big(\frac{\partial}{\partial x_j}  L_{b_0}^{-1}[\bar g]\big) \mu_{0}\right| \notag \\
& =\sup_{\|g\|_{L^2}+\|g\|_{B^{r}_{1 \infty}} \le 1}  \left| \langle h, \mu_0 \nabla L_{b_0}^{-1}[\bar g] \rangle_{L^2} \right|  \\
&\lesssim \|\mu_0\|_{C^{r+1}} \sup_{\|g\|_{L^2}+\|g\|_{B^{r}_{1 \infty}} \le 1}\|\nabla L_{b_0}^{-1}[\bar g]\|_{B^{r+1}_{1 \infty}} \sup_{\tilde g:\|\tilde g\|_{B^{r+1}_{1 \infty}} \le 1}  \left| \langle h, \tilde g\rangle_{L^2}\right|  \lesssim \|h\|_{(B^{r+1}_{1\infty})^*}, \notag
\end{align}
where we have used (\ref{mult}), (\ref{besovmap}) below and that $\nabla$ maps ${B^{r+2}_{1\infty}}$ continuously into $B^{r+1, \otimes d}_{1\infty}$. We also used that $\mu_0 \in C^{t}(\T^d), t \ge 1,$ whenever $b_0 \in C^s \cap H^s, s>t+d/2-1$. Indeed, $\|\mu_0\|_{H^1} \lesssim \|\mu_0\|_{Lip} \le C(d,\|b_0\|_\infty)$ (by Proposition \ref{as:true_drift} and Rademacher's theorem) allows an iterated application of the inequality (\ref{initiostar}) below with $u=\mu_0$ to bound $\|\mu_0\|_{H^{s+1}}$ by a constant $C(d, s, \|b_0\|_{C^s})$, which in turn bounds $\|\mu_0\|_{C^t}$ by the Sobolev embedding theorem. 

Summarizing, with $h=b-\hat{b}_T$ we have proved uniformly in $\|g\|_{\mathbb{B}_r} \leq 1$,
$$\sqrt{T} \int_{\T^d} (\mu_b - \mu_{\hat{b}_T}) g = \int_{\T^d} v_{b_0,\sqrt{T}(b-\hat{b}_T)} g + o_{P_{b_0}}(1)$$
and that the linear operator $h\mapsto v_{b_0,h}$ is continuous from $(C^1(\T^d),\|\cdot\|_{(B_{1\infty}^{r+1})^*})$ to $\mathbb{B}_r^*.$ Theorem \ref{invabvm} now follows from Theorem \ref{bvm1} with $r+1=\rho$ and the continuous mapping theorem for weak convergence applied to $\sqrt{T}(b-\hat b_T)$. We note that the calculation leading to (\ref{chain}) shows that the covariance of the limiting Gaussian process is the one of the Gaussian process $g \mapsto \mathbb W_0(\mu_0 \nabla L_{b_0}^{-1}[\bar g])$, $\mathbb W_0 \sim \mathcal N_{b_0}$, of the required form. In particular, $\mathcal N_{\mu_0}$ exists as a tight Gaussian probability measure in $\mathbb B_r^*$ as the image of $\mathcal N_{b_0}$ under the continuous map $v_{b_0, \cdot}$. The limit of the MAP-estimate follows from similar (in fact simpler) arguments and Theorem \ref{mapas}, and is left to the reader.

\subsection{Proof of Theorem \ref{invaclt}}

We finally prove Theorem \ref{invaclt} and explain the necessary modifications to the arguments from the proof of Theorem \ref{invabvm}. As in the proof of Theorem \ref{invabvm}, one shows that $\hat b_T, b$ are (in the former case, stochastically) bounded in $B^1_{\infty \infty}$ on the set $\bar D_T$, and so are then $\mu_{\hat b_T}, \mu_b$ by Proposition \ref{as:true_drift}. Using the Sobolev-embedding $H^1(\T) \subset C(\T)$ and then repeatedly Lemma \ref{regestc}, (\ref{negtreg}) and the basic interpolation inequality $\|g\|_{H^1} \lesssim \|g\|_{H^2}^{1/2}\|g\|_{L^2}^{1/2}$, the second term in the decomposition (\ref{deltam}) can be bounded by 
$$\|(L_{b+h}^*)^{-1}[\bar f_h]\|_{H^1} \lesssim \|\bar f_h\|_{L^2}^{1/2} \|\bar f_h\|_{H^{-2}}^{1/2} \lesssim \|h\|_{H^1} \|h\|_{L^2},$$ which for $a>3/2$ and $h=b-\hat b_T=b-b_0 -(\hat b_T - b_0), b \sim \Pi^{\bar D_T}(\cdot|X^T),$ is of order $\|h\|_{H^1} \|h\|_{L^2} =  o_P(1/\sqrt T)$ since $\|h\|_{H^1} \lesssim 2^J \|h\|_{L^2}$ for $h \in V_J$. The linear term in (\ref{deltam}) can be decomposed as
$$(L_{b_0}^*)^{-1}[f_h](x) -[(L_{b_0}^*)^{-1}-(L_{b}^*)^{-1}][f_h](x).$$ Then arguing as before (\ref{a1}) and using the Sobolev embedding $H^1 \subset C(\T)$ as well as Lemma \ref{regestc}, the second term is bounded, for $a>3/2$, by
$$\|[(L_{b}^*)^{-1}-(L_{b_0}^*)^{-1}][f_h]\|_{\infty} \lesssim  \|b-b_0\|_\infty (\|\hat b_T - b_0\|_{H^1} + \|b-b_0\|_{H^1}) =o_{P_{b_0}}(1/\sqrt T).$$
Similarly, noting the dependence $f_h=f_{h,b}$ on $b$, the term $(L_{b_0}^*)^{-1}[f_{h,b}] - (L_{b_0}^*)^{-1}[f_{h,b_0}]$ can be shown to be $o_{P_{b_0}}(1/\sqrt T)$ in $C(\T)$. 

We next establish continuity of the linear operator $$h \mapsto v_{b_0, h} = (L_{b_0}^*)^{-1}[f_{h,b_0}]$$ on $C^1(\T)$ for the norms of $(B^1_{1\infty}(\T))^*$ and $C(\T)$, so that the theorem follows from Theorem \ref{bvm1} and the continuous mapping theorem for weak convergence, just as in the proof of Theorem \ref{invabvm}. We use a dual representation for the weighted wavelet sequence norms characterising Besov spaces -- more precisely, that the classical identities $(c_0)^*=\ell_1, (\ell_1)^*=\ell_\infty$, where $c_0=\{(a_k): \lim_{k \to \infty}a_k=0\}$ is equipped with the supremum-norm on sequences, imply $$\|g\|_{B^0_{\infty 1}} \lesssim \sup_{\phi \in C(\T):\|\phi\|_{B^0_{1\infty}}\le 1} |\langle g, \phi\rangle_{L^2}|, \quad g \in C(\T).$$ Moreover (\ref{mult}) and (\ref{besovmap}) imply
$$\sup_{\|\phi\|_{B^0_{1\infty}}\le 1} \| \mu_0 \frac{d}{dy} L_{b_0}^{-1}[\bar \phi]\|_{B^1_{1\infty}} \lesssim \|\mu_0\|_{Lip} \sup_{\|\phi\|_{B^0_{1\infty}}\le 1} \| L_{b_0}^{-1}[\bar \phi]\|_{B^2_{1\infty}} <\infty,$$ which will be used in the following estimate. For $\bar \phi = \phi- \int \phi d\mu_{b_0}$, and since $v_{b_0, h} \in L^2_0(\T) \cap H^2 \subset C(\T)$ in view of Lemma \ref{regestc} below,
\begin{align*}
\|v_{b_0,h}\|_\infty &\lesssim \|v_{b_0,h}\|_{B^0_{\infty 1}} \lesssim \sup_{\phi \in C(\T): \|\phi\|_{B^0_{1\infty}}\le 1} \left|\int v_{b_0,h} \phi \right| \\
&= \sup_{\phi \in C(\T):\|\phi\|_{B^0_{1\infty}}\le 1} \left|\int (L_{b_0}^*)^{-1}[f_{h,b_0}] L_{b_0}L_{b_0}^{-1}[\bar \phi] \right| \\
&= \sup_{\phi \in C(\T):\|\phi\|_{B^0_{1\infty}}\le 1} \left|\int f_{h, b_0} L_{b_0}^{-1}[\bar \phi] \right| \\
& \le \sup_{\|\phi\|_{B^0_{1\infty}}\le 1} \| \mu_0 \frac{d}{dy} L_{b_0}^{-1}[\bar \phi]\|_{B^1_{1\infty}} \|h\|_{(B^1_{1\infty})^*} \lesssim \|h\|_{(B^1_{1\infty})^*}.
\end{align*}
The covariance of the limiting Gaussian process is obtained as follows: since $G_{b_0}(x,y)$ is the periodic Green kernel of $L_{b_0}^{-1}$, the Green kernel of $(L_{b_0}^{*})^{-1}$ is $G_{b_0}(y,x)$, and thus by the definitions and integration by parts,
$$(L_{b_0}^*)^{-1}[f_{h,b_0}] = -\int_\T G_{b_0}(y,\cdot) \frac{d}{dy}[h \mu_0](y)dy = \int_\T \frac{d}{dy}G_{b_0}(y, \cdot)  h(y) \mu_0(y)dy.$$ Inserting for $h$ the limit $\mathbb W_0 \sim \mathcal N_{b_0}$ of $\sqrt T(b - \hat b_T)$ gives the desired form of the limiting covariance. Finally, the limit distribution of the MAP estimate follows from the same (in fact simpler) arguments and Theorem \ref{mapas}.

\section{Appendix: Some basic facts on the elliptic PDEs involved}\label{pdeeee}

We record here some basic facts about elliptic PDEs and refer to, e.g., Chapter II.3 in \cite{BJS64} as a reference for standard background material in the periodic setting considered here. The generator $L=L_b$ of the diffusion process given in (\ref{eq:generator}) is a strongly elliptic second order partial differential operator. We will  suppress the dependence on $b$ in most of what follows; all that is required is that $b$ is `smooth enough', and $b \in V_J^{\otimes d}$ for a $S$-regular wavelet basis with $S$ large enough will be sufficient throughout. The maximum principle for elliptic operators (see \cite{GT98, BJS64}) implies that any (strong and then also weak) periodic solution of the Laplace equation
\begin{equation}\label{lapl}
Lu = 0 \text{ on } \T^d
\end{equation}
equals a constant. The adjoint operator $L^*=L^*_b$ was defined in (\ref{eq:cgen}), and in the periodic setting considered here the operators $(L, L^*)$ form a Fredholm pair on $L^2(\T^d)$, see p.175f.~in \cite{BJS64}. As a consequence, the inhomogeneous equation
\begin{equation} \label{inhomd}
Lu = f,\quad f \in L^2(\T^d),
\end{equation}
has a solution $u$ if and only if $\langle f, m \rangle_{L^2}=0$ for every  solution $m \in L^2(\T^d)$ of
\begin{equation} \label{homcon}
L^*m = 0 \text{ on } \T^d.
\end{equation}
By the Fredholm property the kernel of $L^*$ has the same dimension as the kernel of $L$ and inspection of the form of $L^*$ shows that the solutions $m \in L^2(\T^d)$ to (\ref{homcon}) are determined up to a normalising constant. It  follows that 
\begin{equation} \label{stark}
L^*m = 0 ~\iff ~ m \in \mathcal K= \{c \mu: c \in \mathbb R\},
\end{equation}
where $\mu > 0$ is the unique solution $m$ (`invariant measure') satisfying $\int_{\T^d} m=1$. Positivity of $\mu$ can be deduced from appropriate heat kernel estimates: in fact (arguing, e.g., as on p.167f. in \cite{N97}) the solution $\mu$ can be seen to be Lipschitz continuous and bounded away from zero on $\T^d$, and $\|\mu\|_{Lip}$ is bounded by a fixed constant that only depends on $d$ and on an upper bound for $\|b\|_\infty$, proving in particular Proposition \ref{as:true_drift}.


We can now state the following basic result for the PDE (\ref{inhomd}).

\begin{lemma}\label{regest}
Let $t \ge 2$ and assume $b \in C^{t-2}(\T^d)$. For any $f \in L^2_\mu(\T^d)$, there exists a unique solution $L_b^{-1}[f] \in  L^2_0(\T^d)$ of equation (\ref{inhomd}) satisfying $L_b L_b^{-1}[f]=f$ almost everywhere. Moreover, $$\|L_b^{-1}[f]\|_{H^t} \lesssim \|f\|_{H^{t-2}},$$
with constants depending on $t,d$ and on an upper bound $B$ for $\|b\|_{B^{t-2}_{\infty \infty}}$.
\end{lemma}
\begin{proof}
By standard Sobolev space theory and definition of the Laplacian we have for any $u \in \mathcal H \equiv H^t \cap \{u: \langle u, 1 \rangle_{L^2} =0\}$ the inequality 
\begin{equation}\label{laplaceest}
\|u\|_{H^t} \lesssim \|\Delta u\|_{H^{t-2}}.
\end{equation}
Indeed, for $\{e_k: k =(k_1, \dots, k_d) \in \mathbb Z^d\}$ the usual trigonometric basis of $L^2(\T^d)$ we have $\langle u, e_0 \rangle_{L^2}=\langle u, 1 \rangle_{L^2}=0$, $\langle \Delta u, e_k \rangle = -(2\pi)^2\sum_j k_j^2 \langle u, e_k \rangle, $ and $\sup_{k \neq 0} (1+\|k\|^2)/\|k\|^2<\infty$, which gives the result using the characterisation of Sobolev norms in the basis $\{e_k\}$. We then also have, by the triangle inequality and (\ref{mult}),
\begin{equation}\label{initio}
\|u\|_{H^t} \lesssim \|L u\|_{H^{t-2}} + \|b . \nabla u\|_{H^{t-2}} \lesssim \|Lu\|_{H^{t-2}} + \|b\|_{B^{t-2}_{\infty \infty}} \|u\|_{H^{t-1}}
\end{equation}
for all $u \in \mathcal H$, with constants depending on $t,d$. We now deduce from this the inequality
\begin{equation}\label{key}
\|u\|_{H^t} \lesssim \|Lu\|_{H^{t-2}}~ \quad \forall u \in \mathcal H.
\end{equation}
Indeed, if the latter inequality does not hold true, then there exists a sequence $u_m \in \mathcal H$ such that $\|u_m\|_{H^t}=1$ for all $m$ but $\|Lu_m\|_{H^{t-2}} \to 0$ as $m \to \infty$. At the same time, by compactness, $u_m$ converges in $\|\cdot\|_{H^{t-1}}$-norm (if necessary along a subsequence) to some $u\in \mathcal H$ satisfying $Lu=0$. Using (\ref{initio}) with fixed constant depending only on $B, t, d$, we see that $u_m$ is also Cauchy in $H^t$, and its limit must necessarily satisfy $\|u\|_{H^t}=1$. However, as remarked after (\ref{lapl}), the only solution $u \in \mathcal H$ to $Lu=0$ on $\T^d$ equals $u=const=0$, a contradiction to $\|u\|_{H^t}=1$, proving (\ref{key}). 

By the Fredholm property and (\ref{stark}), a solution $u_f$ to (\ref{inhomd}) exists whenever $\int f d\mu = 0$, and for $f \in H^{t-2}(\T^d)$ any such solution belongs to $H^t(\T^d)$ (see Theorem 3.5.3 in \cite{BJS64}, which is proved for smooth $b$, but the proof remains valid for $b \in C^{t-2}(\T^d)$). The weak maximum principle (p.179 in \cite{GT98}) now implies that $u_f$ is unique up to an additive constant, and applying (\ref{key}) to the unique selection $u_f=L^{-1}[f] \in \mathcal H$ completes the proof. 
\end{proof}

We next obtain corresponding results for the adjoint PDE. It follows from (\ref{stark}) that the unique element $m \in \mathcal K$ satisfying $\int_{\T^d} m=0$ must necessarily vanish identically, and we can study the solution operator $(L^*)^{-1}$ of the inhomogeneous adjoint PDE
\begin{equation} \label{cpdeih}
L^* u = f ~\text{on } \T^d,
\end{equation}
which assigns to any $f \in L^2_0(\T^d)$ the unique  solution $u=(L^*)^{-1}[f] \in L^2_0(\T^d)$. Indeed, using the Fredholm property from Section 3.6 in \cite{BJS64} in a reverse way (with $L$ equal to our $L^*$ so that the new $L^*$ is our $(L^*)^*=L$), we see that solutions $u=u_f$ to (\ref{cpdeih}) exist for any periodic $f$ for which $\int f=0$ (since solutions to $Lu=0$ equal constants), and if $u_1, u_2$ are two such solutions, so that $L^*(u_1-u_2)=0$ and $\int u_1 = \int u_2$, then necessarily $u_1=u_2$ by what precedes. 
\begin{lemma}\label{regestc}
Let $t \ge 2$ and assume $b \in C^{t-1}(\T^d)$. Then for any $f \in H^{t-2}(\T^d) \cap L^2_0(\T^d),$ we have $(L_b^*)^{-1}[f] \in H^{t}(\T^d)$ and $$\|(L_b^*)^{-1}[f]\|_{H^t} \lesssim \|f\|_{H^{t-2}},$$
with constants depending on $t,d$ and on an upper bound $B$ for $\|b\|_{B^{t-1}_{\infty \infty}}$. 
\end{lemma}
\begin{proof}
The proof is similar to the one of Lemma \ref{regest} after deriving the basic inequality
\begin{equation}\label{initiostar}
\begin{split}
\|u\|_{H^t} & \lesssim \|L^* u\|_{H^{t-2}} + \|b . \nabla u + div(b) u\|_{H^{t-2}} \\
& \lesssim \|L^*u\|_{H^{t-2}} + \|b\|_{B^{t-2}_{\infty \infty}} \|u\|_{H^{t-1}} + \|b\|_{B^{t-1}_{\infty \infty}} \|u\|_{H^{t-2}}
\end{split}
\end{equation}
in analogy to (\ref{initio}).
\end{proof}

We can also give a version of Lemma \ref{regestc} with $t=0$. Since $(L^*_b)^{-1}[f] \in L^2_0(\T^d)$ we have for all $f \in L^2_0(\T^d)$ and $\bar \phi = \phi - \int \phi d\mu_b$ the estimate
\begin{align} 
\|(L^*_b)^{-1}[f]\|_{L^2} &= \sup_{\|\phi\|_{L^2} \le 1} \left|\int (L^*_b)^{-1}[f] L_b L_b^{-1}[\bar \phi]  \right| = \sup_{\|\phi\|_{L^2} \le 1} \left|\int f L_b^{-1}[\bar \phi ]  \right| \notag \\
&\leq \|f\|_{H^{-2}} \sup_{\|\phi\|_{L^2}\le 1} \|L_b^{-1}[\bar \phi ]\|_{H^2} \lesssim \|f\|_{H^{-2}},\label{negtreg}
\end{align}
where we have used Lemma \ref{regest} with $t=2$ in the last inequality and where the constants in the last inequality depend only on $d$ and on bounds for $\|b\|_{B^1_{\infty \infty}}$ and $\|\mu_b\|_{L^2}$.

\subsubsection{Refinements on the Besov scale}

For the the proofs of Theorems \ref{invabvm} and \ref{invaclt} we need more refined regularity estimates for the solutions of the PDE involved, replacing the Sobolev norms in Lemma \ref{regest} by appropriate Besov norms. The inequality
\begin{equation}\label{besovmap}
\|L_b^{-1}[f]\|_{B^t_{1\infty}} \lesssim \|f\|_{B^{t-2}_{1\infty}}, ~t-2 \ge 0, \forall f \in L^2_{\mu_b}(\T^d),
\end{equation} 
with constants depending on $b$ only via a bound $B$ for $\|b\|_{B^{t-1}_{\infty \infty}}$, is proved in the same way as Lemma \ref{regest}, replacing the basic inequality (\ref{laplaceest}) by its analogue for Besov norms 
\begin{equation}\label{laplacebesov}
\|u\|_{B^t_{1\infty}} \lesssim \|\Delta u\|_{B^{t-2}_{1\infty}} \quad \forall u \in B^t_{1\infty} \cap \{u:\langle u, e_0 \rangle =0\},
\end{equation} 
which is proved as follows: for all $u$ such that $\langle u, e_0 \rangle = 0$ and $\mathbb N_0 = \mathbb N \cup \{0\}$, an equivalent Littlewood-Paley norm on any Besov space $B^r_{1\infty}$ is given by 
\begin{equation*}
\|u\|_{B^r_{1\infty}} = \sup_{j \in \mathbb N_0} 2^{jr} \big\|\sum_{k \in \mathbb Z, k \neq 0} \psi_j(k) \langle u, e_k \rangle e_k \big\|_{L^1(\T^d)},
\end{equation*}
where the $\psi_j = \psi(\cdot/2^j), \supp(\psi) \in (1/2, 2)^d$ form a Littlewood-Paley resolution of unity, see p.162f.~in \cite{ST87}. Then as after (\ref{laplaceest}),
\begin{align*}
\|u\|_{B^t_{1\infty}} &= \sup_{j \in \mathbb N_0} 2^{jt} \big\|\sum_{k \in \mathbb Z, k \neq 0} \frac{1}{4\pi^2 \|k\|^2} \psi_j(k) \langle \Delta u, e_k \rangle e_k \big\|_{L^1(\T^d)} \\
&= \sup_{j \in \mathbb N_0} 2^{j(t-2)} \big\|\sum_{k \in \mathbb Z, k \neq 0} M_j(k)\psi_j(k) \langle \Delta u, e_k \rangle e_k \big\|_{L^1(\T^d)}
\end{align*}
where $M_j= M(\cdot/2^j)$ and $M=\Phi/(4\pi^2\|\cdot\|^2)$ with $\Phi$ a smooth function supported in $(1/4,9/4)^d$ such that $\Phi =1$ on $(1/2,2)^d$. By a standard Fourier multiplier inequality (e.g., Lemma 4.3.27 in \cite{ginenickl2016}, which easily generalises to $d>1$) the last norm can be estimated by
$$\sup_{j \in \mathbb N_0} 2^{j(t-2)} \big\|\sum_{k \in \mathbb Z, k \neq 0} \psi_j(k) \langle \Delta u, e_k \rangle e_k \big\|_{L^1(\T^d)} \times \|F^{-1}M_j\|_{L^1(\mathbb R^d)},$$ where $F^{-1}$ is the inverse Fourier transform. Since $\Phi$ is smooth and supported in $(-1/4,3/4)^d$, both $M$ and $F^{-1}M$ belong to the Schwartz-class $\mathcal S$, so that (\ref{laplacebesov}) follows from $$\sup_j \|F^{-1}[M_j]\|_{L^1(\mathbb R^d)}=\|F^{-1}[M]\|_{L^1(\mathbb R^d)}<\infty.$$ 

\textbf{Acknowledgements.} RN was supported by the European Research Council under ERC grant No. 647812 (UQMSI). We would like to thank James Norris for helpful discussions and the Associate Editor and two referees for helpful comments and for drawing several references to our attention.

\bibliography{diffusion_bvm}{}

\begin{thebibliography}{10}

\bibitem{A18}
{\sc Abraham, K.}
\newblock Nonparametric {B}ayesian posterior contraction rates for scalar
  diffusions with high-frequency data.
\newblock {\em Bernoulli, to appear, arXiv:1802.05635\/} (2018).

\bibitem{AS18b}
{\sc Aeckerle-Willems, C., and Strauch, C.}
\newblock Concentration of scalar ergodic diffusions and some statistical
  implications.
\newblock {\em arXiv:1807.11331\/} (2018).

\bibitem{AS18a}
{\sc Aeckerle-Willems, C., and Strauch, C.}
\newblock Sup-norm adaptive simultaneous drift estimation for ergodic
  diffusions.
\newblock {\em arXiv:1808.10660\/} (2018).

\bibitem{ALS13}
{\sc Agapiou, S., Larsson, S., and Stuart, A.~M.}
\newblock Posterior contraction rates for the {B}ayesian approach to linear
  ill-posed inverse problems.
\newblock {\em Stochastic Process. Appl. 123}, 10 (2013), 3828--3860.

\bibitem{B11}
{\sc Bass, R.~F.}
\newblock {\em Stochastic processes}.
\newblock Cambridge Univ. Press, Cambridge, 2011.

\bibitem{BJS64}
{\sc Bers, L., John, F., and Schechter, M.}
\newblock {\em Partial differential equations}.
\newblock Lectures in Applied Mathematics, Vol. III. Wiley \& Sons, Inc.\, New
  York-London-Sydney, 1964.

\bibitem{BPRF06}
{\sc Beskos, A., Papaspiliopoulos, O., Roberts, G.~O., and Fearnhead, P.}
\newblock Exact and computationally efficient likelihood-based estimation for
  discretely observed diffusion processes.
\newblock {\em J. R. Stat. Soc. Ser. B Stat. Methodol. 68}, 3 (2006), 333--382.

\bibitem{BFS16}
{\sc Bladt, M., Finch, S., and S{\o}rensen, M.}
\newblock Simulation of multivariate diffusion bridges.
\newblock {\em J. R. Stat. Soc. Ser. B. Stat. Methodol. 78}, 2 (2016),
  343--369.

\bibitem{castillo2014}
{\sc Castillo, I.}
\newblock On {B}ayesian supremum norm contraction rates.
\newblock {\em Ann. Statist. 42}, 5 (2014), 2058--2091.

\bibitem{C17}
{\sc Castillo, I.}
\newblock P\'{o}lya tree posterior distributions on densities.
\newblock {\em Ann. Inst. Henri Poincar\'{e} Probab. Stat. 53}, 4 (2017),
  2074--2102.

\bibitem{CN13}
{\sc Castillo, I., and Nickl, R.}
\newblock Nonparametric {B}ernstein--von {M}ises {T}heorems in {G}aussian white
  noise.
\newblock {\em Ann. Statist. 41}, 4 (2013), 1999--2028.

\bibitem{CN14}
{\sc Castillo, I., and Nickl, R.}
\newblock On the {B}ernstein--von {M}ises phenomenon for nonparametric {B}ayes
  procedures.
\newblock {\em Ann. Statist. 42}, 5 (2014), 1941--1969.

\bibitem{CR15}
{\sc Castillo, I., and Rousseau, J.}
\newblock A {B}ernstein--von {M}ises theorem for smooth functionals in
  semiparametric models.
\newblock {\em Ann. Statist. 43}, 6 (2015), 2353--2383.

\bibitem{D05}
{\sc Dalalyan, A.}
\newblock Sharp adaptive estimation of the drift function for ergodic
  diffusions.
\newblock {\em Ann. Statist. 33}, 6 (2005), 2507--2528.

\bibitem{DR06}
{\sc Dalalyan, A., and Rei\ss, M.}
\newblock Asymptotic statistical equivalence for scalar ergodic diffusions.
\newblock {\em Probab. Theory Related Fields 134}, 2 (2006), 248--282.

\bibitem{DR07}
{\sc Dalalyan, A., and Rei\ss, M.}
\newblock Asymptotic statistical equivalence for ergodic diffusions: the
  multidimensional case.
\newblock {\em Probab. Theory Related Fields 137}, 1-2 (2007), 25--47.

\bibitem{DLSV13}
{\sc Dashti, M., Law, K. J.~H., Stuart, A.~M., and Voss, J.}
\newblock {MAP} estimators and their consistency in {B}ayesian nonparametric
  inverse problems.
\newblock {\em Inverse Problems 29}, 9 (2013), 095017, 27.

\bibitem{D02}
{\sc Dudley, R.~M.}
\newblock {\em Real analysis and probability}.
\newblock Cambridge University Press, Cambridge, 2002.

\bibitem{D14}
{\sc Dudley, R.~M.}
\newblock {\em Uniform central limit theorems}, second~ed.
\newblock Cambridge University Press, New York, 2014.

\bibitem{KE86}
{\sc Ethier, S.~N., and Kurtz, T.~G.}
\newblock {\em Markov processes, Characterization and convergence}.
\newblock John Wiley \& Sons, Inc., New York, 1986.

\bibitem{ghosal2000}
{\sc Ghosal, S., Ghosh, J.~K., and van~der Vaart, A.~W.}
\newblock Convergence rates of posterior distributions.
\newblock {\em Ann. Statist. 28}, 2 (2000), 500--531.

\bibitem{GvdV17}
{\sc Ghosal, S., and van~der Vaart, A.~W.}
\newblock {\em Fundamentals of Nonparametric Bayesian Inference}.
\newblock Cambridge University Press, New York, 2017.

\bibitem{GT98}
{\sc Gilbarg, D., and Trudinger, N.~S.}
\newblock {\em Elliptic partial differential equations of second order}.
\newblock Classics in Mathematics. Springer-Verlag, Berlin, 2001.
\newblock Reprint of the 1998 edition.

\bibitem{GN08}
{\sc Gin\'e, E., and Nickl, R.}
\newblock Uniform central limit theorems for kernel density estimators.
\newblock {\em Probab. Theory Related Fields 141}, 3-4 (2008), 333--387.

\bibitem{GN09}
{\sc Gin\'e, E., and Nickl, R.}
\newblock An exponential inequality for the distribution function of the kernel
  density estimator, with applications to adaptive estimation.
\newblock {\em Probab. Theory Related Fields 143}, 3-4 (2009), 569--596.

\bibitem{ginenickl2016}
{\sc Gin\'e, E., and Nickl, R.}
\newblock {\em Mathematical foundations of infinite-dimensional statistical
  models}.
\newblock Cambridge University Press, New York, 2016.

\bibitem{GK18}
{\sc Giordano, M., and Kekkonen, H.}
\newblock {Bernstein-von Mises theorems and uncertainty quantification for
  linear inverse problems}.
\newblock {\em arXiv:1811.04058\/} (2018).

\bibitem{GS14}
{\sc Gugushvili, S., and Spreij, P.}
\newblock Nonparametric {B}ayesian drift estimation for multidimensional
  stochastic differential equations.
\newblock {\em Lith. Math. J. 54}, 2 (2014), 127--141.

\bibitem{horn2013}
{\sc Horn, R.~A., and Johnson, C.~R.}
\newblock {\em Matrix analysis}, second~ed.
\newblock Cambridge University Press, Cambridge, 2013.

\bibitem{KS18}
{\sc Knapik, B., and Salomond, J.-B.}
\newblock A general approach to posterior contraction in nonparametric inverse
  problems.
\newblock {\em Bernoulli 24}, 3 (2018), 2091--2121.

\bibitem{KvdVvZ11}
{\sc Knapik, B., van~der Vaart, A.~W., and van Zanten, J.~H.}
\newblock Bayesian inverse problems with {G}aussian priors.
\newblock {\em Ann. Statist. 39}, 5 (2011), 2626--2657.

\bibitem{K04}
{\sc Kutoyants, Y.~A.}
\newblock {\em Statistical inference for ergodic diffusion processes}.
\newblock Springer Series in Statistics. Springer-Verlag London, Ltd., London,
  2004.

\bibitem{LLL11}
{\sc L\"{o}cherbach, E., Loukianova, D., and Loukianov, O.}
\newblock Penalized nonparametric drift estimation for a continuously observed
  one-dimensional diffusion process.
\newblock {\em ESAIM Probab. Stat. 15\/} (2011), 197--216.

\bibitem{MNP17}
{\sc Monard, F., Nickl, R., and Paternain, G.~P.}
\newblock Efficient nonparametric {B}ayesian inference for {$X$}-ray
  transforms.
\newblock {\em Ann. Statist. 47}, 2 (2019), 1113--1147.

\bibitem{N17}
{\sc Nickl, R.}
\newblock Bernstein-von {M}ises theorems for statistical inverse problems {I}:
  Schr\"odinger equation.
\newblock {\em Journal of the European Mathematical Society, to appear;
  arXiv:1707.01764\/} (2017).

\bibitem{NS17b}
{\sc Nickl, R., and S\"ohl, J.}
\newblock Bernstein-von {M}ises theorems for statistical inverse problems {II}:
  compound {P}oisson processes.
\newblock {\em arXiv:1709.07752\/} (2017).

\bibitem{NS17a}
{\sc Nickl, R., and S\"ohl, J.}
\newblock Nonparametric {B}ayesian posterior contraction rates for discretely
  observed scalar diffusions.
\newblock {\em Ann. Statist. 45}, 4 (2017), 1664--1693.

\bibitem{NvdGW18}
{\sc Nickl, R., van~de Geer, S., and Wang, S.}
\newblock {Convergence rates for penalised least squares estimators in
  PDE-constrained regression problems}.
\newblock {\em arXiv:1809.08818\/} (2018).

\bibitem{N97}
{\sc Norris, J.~R.}
\newblock Long-time behaviour of heat flow: global estimates and exact
  asymptotics.
\newblock {\em Arch. Rational Mech. Anal. 140}, 2 (1997), 161--195.

\bibitem{PPRS12}
{\sc Papaspiliopoulos, O., Pokern, Y., Roberts, G.~O., and Stuart, A.~M.}
\newblock Nonparametric estimation of diffusions: a differential equations
  approach.
\newblock {\em Biometrika 99}, 3 (2012), 511--531.

\bibitem{PV01}
{\sc Pardoux, E., and Veretennikov, A.~Y.}
\newblock On the {P}oisson equation and diffusion approximation. {I}.
\newblock {\em Ann. Probab. 29}, 3 (2001), 1061--1085.

\bibitem{pokern2013}
{\sc Pokern, Y., Stuart, A.~M., and van Zanten, J.~H.}
\newblock Posterior consistency via precision operators for {B}ayesian
  nonparametric drift estimation in {SDE}s.
\newblock {\em Stochastic Process. Appl. 123}, 2 (2013), 603--628.

\bibitem{R13}
{\sc Ray, K.}
\newblock Bayesian inverse problems with non-conjugate priors.
\newblock {\em Electron. J. Stat. 7\/} (2013), 2516--2549.

\bibitem{R14}
{\sc Ray, K.}
\newblock {\em Asymptotic theory for Bayesian nonparametric procedures in
  inverse problems}.
\newblock PhD thesis, University of Cambridge, 2014.

\bibitem{R17}
{\sc Ray, K.}
\newblock Adaptive {B}ernstein--von {M}ises theorems in {G}aussian white noise.
\newblock {\em Ann. Statist. 45}, 6 (2017), 2511--2536.

\bibitem{revuz1999}
{\sc Revuz, D., and Yor, M.}
\newblock {\em Continuous martingales and {B}rownian motion}, third~ed.,
  vol.~293.
\newblock Springer-Verlag, Berlin, 1999.

\bibitem{RR12}
{\sc Rivoirard, V., and Rousseau, J.}
\newblock Bernstein-von {M}ises theorem for linear functionals of the density.
\newblock {\em Ann. Statist. 40}, 3 (2012), 1489--1523.

\bibitem{runst1996}
{\sc Runst, T., and Sickel, W.}
\newblock {\em Sobolev spaces of fractional order, {N}emytskij operators, and
  nonlinear partial differential equations}.
\newblock Walter de Gruyter \& Co., Berlin, 1996.

\bibitem{vdMS17b}
{\sc Schauer, M., van~der Meulen, F., and van Zanten, H.}
\newblock Guided proposals for simulating multi-dimensional diffusion bridges.
\newblock {\em Bernoulli 23}, 4A (2017), 2917--2950.

\bibitem{ST87}
{\sc Schmeisser, H.-J., and Triebel, H.}
\newblock {\em Topics in {F}ourier analysis and function spaces}.
\newblock A Wiley-Interscience Publication. John Wiley \& Sons, Ltd.,
  Chichester, 1987.

\bibitem{S13}
{\sc Schmisser, E.}
\newblock Penalized nonparametric drift estimation for a multidimensional
  diffusion process.
\newblock {\em Statistics 47}, 1 (2013), 61--84.

\bibitem{S15}
{\sc Strauch, C.}
\newblock Sharp adaptive drift estimation for ergodic diffusions: the
  multivariate case.
\newblock {\em Stochastic Process. Appl. 125}, 7 (2015), 2562--2602.

\bibitem{S16}
{\sc Strauch, C.}
\newblock Exact adaptive pointwise drift estimation for multidimensional
  ergodic diffusions.
\newblock {\em Probab. Theory Related Fields 164}, 1-2 (2016), 361--400.

\bibitem{S18}
{\sc Strauch, C.}
\newblock Adaptive invariant density estimation for ergodic diffusions over
  anisotropic classes.
\newblock {\em Ann. Statist. 46}, 6B (2018), 3451--3480.

\bibitem{vdMS17}
{\sc van~der Meulen, F., and Schauer, M.}
\newblock Bayesian estimation of discretely observed multi-dimensional
  diffusion processes using guided proposals.
\newblock {\em Electron. J. Stat. 11}, 1 (2017), 2358--2396.

\bibitem{vdMvZ13}
{\sc van~der Meulen, F., and van Zanten, H.}
\newblock Consistent nonparametric {B}ayesian inference for discretely observed
  scalar diffusions.
\newblock {\em Bernoulli 19}, 1 (2013), 44--63.

\bibitem{vdmeulen2006}
{\sc van~der Meulen, F.~H., van~der Vaart, A.~W., and van Zanten, J.~H.}
\newblock Convergence rates of posterior distributions for {B}rownian
  semimartingale models.
\newblock {\em Bernoulli 12}, 5 (2006), 863--888.

\bibitem{vdvaart1998}
{\sc van~der Vaart, A.~W.}
\newblock {\em Asymptotic statistics}.
\newblock Cambridge Univ. Press, Cambridge, 1998.

\bibitem{vdvaart2008}
{\sc van~der Vaart, A.~W., and van Zanten, J.~H.}
\newblock Rates of contraction of posterior distributions based on {G}aussian
  process priors.
\newblock {\em Ann. Statist. 36}, 3 (2008), 1435--1463.

\bibitem{vanwaaij2016}
{\sc van Waaij, J., and van Zanten, H.}
\newblock Gaussian process methods for one-dimensional diffusions: optimal
  rates and adaptation.
\newblock {\em Electron. J. Stat. 10}, 1 (2016), 628--645.

\bibitem{V10}
{\sc Vershynin, R.}
\newblock Introduction to the non-asymptotic analysis of random matrices.
\newblock {\em arXiv:1011.3027\/} (2010).

\end{thebibliography}
\bibliographystyle{acm}

\end{document}